\documentclass[12pt,reqno]{amsart}

\usepackage{amssymb}
\usepackage[T1]{fontenc}
\usepackage[latin1]{inputenc} 
\usepackage{dsfont, mathrsfs}
\usepackage{a4wide}  
\usepackage{stmaryrd}
\usepackage{amsthm}
\usepackage{graphicx}
\usepackage{color,xcolor, soul}  
\usepackage{xspace}
\usepackage{yhmath, epigraph}
\usepackage{caption, subcaption}
\usepackage{tikz}
\usetikzlibrary{positioning}
\usetikzlibrary{calc}
\usepackage[
        colorlinks,
        hyperfootnotes=true,
        pagebackref,
        hyperindex,
        citecolor=purple,
        linkcolor=blue]{hyperref}
\usepackage{framed} 
\colorlet{shadecolor}{gray!40}
\usepackage{bbm}
\usepackage[bbgreekl]{mathbbol}
\DeclareSymbolFontAlphabet{\mathbb}{AMSb}
\DeclareSymbolFontAlphabet{\mathbbl}{bbold}

\usepackage{verbatim}
\DeclareFontFamily{U}{mathx}{\hyphenchar\font45}
\DeclareFontShape{U}{mathx}{m}{n}{
      <5> <6> <7> <8> <9> <10>
      <10.95> <12> <14.4> <17.28> <20.74> <24.88>
      mathx10
      }{}
\DeclareSymbolFont{mathx}{U}{mathx}{m}{n}
\DeclareFontSubstitution{U}{mathx}{m}{n}
\DeclareMathAccent{\widecheck}{0}{mathx}{"71}
\DeclareMathAccent{\wideparen}{0}{mathx}{"75}

\allowdisplaybreaks

\theoremstyle{plain}
\newtheorem{theorem}{Theorem}[section]
\newtheorem{lemma}[theorem]{Lemma}

\newtheorem{definition}[theorem]{Definition}

\newtheorem{question}[theorem]{Question}

\theoremstyle{remark}
\newtheorem{remark}[theorem]{Remark}

\newcommand{\Cc}{\mathbb{C}}

\newcommand{\Ee}{\mathbb{E}}

\newcommand{\Nn}{\mathbb{N}}
\newcommand{\Pp}{\mathbb{P}}

\newcommand{\Rr}{\mathbb{R}}

\newcommand{\Tt}{\mathbb{T}}
\newcommand{\Uu}{\mathbb{U}}

\newcommand{\Yy}{\mathbb{Y}}
\newcommand{\Zz}{\mathbb{Z}}

\newcommand{\Un}{\mathds{1}}

\def\ee{ { \mathbbm{e} } }

\newcommand{\Ae}{\mathcal{A}}
\newcommand{\Be}{\mathcal{B}}
\newcommand{\Ce}{\mathcal{C}}
\newcommand{\De}{\mathcal{D}}
\newcommand{\Eee}{\mathcal{E}}
\newcommand{\Fe}{\mathcal{F}}

\newcommand{\He}{\mathcal{H}}
\newcommand{\Ie}{\mathcal{I}}
\newcommand{\Je}{\mathcal{J}}
\newcommand{\Ke}{\mathcal{K}}
\newcommand{\Le}{\mathcal{L}}
\newcommand{\Me}{\mathcal{M}}

\newcommand{\Qe}{\mathcal{Q}}
\newcommand{\Ree}{\mathcal{R}}

\newcommand{\Ue}{\mathcal{U}}

\newcommand{\We}{\mathcal{W}}

\newcommand{\Ze}{\mathcal{Z}}

\newcommand{\Heb}{\boldsymbol{\mathcal{H}}}

\newcommand{\Keb}{\boldsymbol{\mathcal{K}}}

\newcommand{\Reb}{\boldsymbol{\mathcal{R}}}

\newcommand{\Ns}{\mathscr{N}}

\newcommand{\Mg}{{\mathfrak{M}}}

\newcommand{\Sg}{{\mathfrak{S}}}

\newcommand{\Xg}{{\mathfrak{X}}}

\newcommand{\gammab}{\boldsymbol{\gamma}}

\newcommand{\lambdab}{\boldsymbol{\lambda}}
\newcommand{\varphib}{\boldsymbol{\varphi}}
\newcommand{\thetab}{\boldsymbol{\theta}}
\newcommand{\taub}{\boldsymbol{\tau}}

\newcommand{\psib}{\boldsymbol{\psi}}

\newcommand{\sigmab}{\boldsymbol{\sigma}}

\newcommand{\Gammab}{\boldsymbol{\Gamma}}
\newcommand{\Lambdab}{\boldsymbol{\Lambda}}

\newcommand{\Thetab}{\boldsymbol{\Theta}}

\newcommand{\Ab}{{\boldsymbol{A}}}
\newcommand{\Bd}{{\boldsymbol{B}}}

\newcommand{\Gb}{{\boldsymbol{G}}}
\newcommand{\Hb}{{\boldsymbol{H}}}
\newcommand{\Jb}{{\boldsymbol{J}}} 
\newcommand{\Kb}{{\boldsymbol{K}}}
\newcommand{\Lb}{{\boldsymbol{L}}} 
\newcommand{\Mb}{{\boldsymbol{M}}}

\newcommand{\Tb}{{\boldsymbol{T}}}

\newcommand{\Xb}{{\boldsymbol{X}}}
\newcommand{\Zb}{{\boldsymbol{Z}}}

\newcommand{\pb}{{\boldsymbol{p}}}
\newcommand{\qb}{{\boldsymbol{q}}}

\newcommand{\tb}{{\boldsymbol{t}}}
\newcommand{\ub}{{\boldsymbol{u}}}

\newcommand{\xb}{{\boldsymbol{x}}}

\newcommand{\ensemble}[1]{ \left\lbrace #1 \right\rbrace } 
\newcommand{\prth}[1]{{\left( #1 \right)}}
\newcommand{\crochet}[1]{{\left[ #1 \right]}}

\newcommand{\abs}[1]{{\left| #1 \right|}}  
\newcommand{\norm}[1]{\left| \! \left| #1 \right| \! \right|}

\newcommand{\Esp}[1]{ \Ee{\prth{ #1 }} }  
\newcommand{\Prob}[1]{ \Pp \prth{ #1 } }  
\newcommand{\Espr}[2]{ \Ee_{#1} \prth{ #2 } } 
\newcommand{\Proba}[2]{ \Pp_{#1} \prth{ #2 } } 

\def\inv{^{-1}} 

\def\longlongrightarrow{\hspace{+0.1ex} - \hspace{-1.1ex} - \hspace{-1.1ex} - \hspace{-1.1ex}\longrightarrow  } 
 
\newcommand{\tendvers}[2]{ \underset{#1 \rightarrow #2}{\longlongrightarrow} }  
\newcommand{\cvlaw}[2]{\stackrel{\Le}{\underset{#1 \, \rightarrow \, #2}{\longlongrightarrow}}}  
\newcommand{\equivalent}[1]{ {\underset{#1 }{\sim} } }  

\newcommand{\bracket}[1]{\left\langle #1 \right\rangle}
\newcommand{\Unens}[1]{ \Un_{ \ensemble{#1} } }
\def\eqlaw{\stackrel{\Le}{=}}
\newcommand{\pe}[1]{{\left[ #1 \right]}}

\newcommand{\tr}[1]{\operatorname{tr}\prth{ #1 } }

\newcommand{\ActsOn}{\curvearrowright}

\def\trace{ \operatorname{tr} }

\def\Mat{ \operatorname{Mat} }

\def\im{{\operatorname{Im}}}

\def\DetMeas{ \operatorname{DetMeas} }
\def\PPD{ \operatorname{DPP} }

\def\soft{{ \operatorname{soft} }}

\def\gue{{ GU\!E }}

\def\PoPl{ \operatorname{PoPl} }

\def\Schur{ \operatorname{Schur} }

\def\DPP{ \operatorname{DPP} }
\def\KPZ{ {K\!P\!Z} }

\def\GinE{ \operatorname{GinE} }

\def\Exp{ \operatorname{Exp} }

\def\Geom{ \operatorname{Geom} }
\def\Poisson{ \operatorname{Po} }
\def\Gumbel{ \operatorname{Gb} }
\def\TW{\mathcal{T}\mathcal{W}}

\def\Ai{{\operatorname{Ai}}}
\def\Bi{{\operatorname{Bi}}}

\def\sinc{{\operatorname{sinc}}}

\def\id{{\operatorname{id}}}

\def\PHO{ \operatorname{PHO} }

\def\geq{\geqslant}
\def\leq{\leqslant}

\let\oldforall\forall
\def\forall{\oldforall\,} 

\let\oldexists\exists
\def\exists{\oldexists\,}

\newcommand{\emailhref}[1]{ \email{\href{mailto:#1}{#1}} }

\definecolor{rougeclair}{rgb}{1,.65,.65}

\definecolor{pink}{RGB}{219, 48, 122}
\definecolor{purple}{RGB}{128, 0, 128}
\definecolor{vertcac}{RGB}{32, 114, 18}
\definecolor{ocre}{RGB}{120, 30, 40}
\definecolor{rose}{RGB}{250, 0, 70}
\definecolor{white}{RGB}{255, 255, 255}

\sethlcolor{rougeclair}

\newcommand{\blue}[1]{{\color{black}#1\color{black}}}
\newcommand{\purple}[1]{{\color{purple}#1\color{black}}} 
\newcommand{\pink}[1]{{\color{pink}#1\color{black}}}
\newcommand{\red}[1]{{\color{red}#1\color{black}}}

\def\www{\tau}

\def\AOP{Ann. Probab. }
\def\CMP{Commun. Math. Phys. }
\def\CPAM{Comm. Pure Appl. Math. }

\def\IMRN{Int. Math. Res. Not. }
\def\JAMS{J Amer. Math. Soc. }

\def\PTRF{Probab. Theory Related Fields }

\title[Max-independence structures in integrable probability]{Remarkable structures in integrable probability, I: max-independence structures} 
\author[Y. Barhoumi-Andr\'eani]{Yacine Barhoumi-Andr\'eani}
\address{Department of Algebra and Logic, 
Institute of Mathematics and Informatics, 
Bulgarian Academy of Sciences, 
Acad. Georgi Bonchev Str., Block 8, 1113 Sofia (Bulgaria).}
\emailhref{yacine.barhoumi@math.bas.bg}
\date{\today}

\begin{document}
\begin{abstract}

We analyse in a systematic way the occurrences of a remarkable structure in the theory of integrable probability that we call a ``max-independence structure'', when random variables are constructed as a maximum of a sequence of independent random variables. 

The list of treated examples contains~: the $ GU\!E $ and $ GO\!E $ Tracy-Widom distributions, the extreme eigenvalues/eigenangles of random hermitian/unitary matrices (and in particular the historical example of the $ GU\!E $ extreme eigenvalues), the Hopf-Cole solution to the KPZ equation with Dirac initial condition (continuum random polymer) and the symmetric Schur measure. In this last case, the largest part of the underlying random partition is the maximum of an i.i.d randomisation of the deterministic sequence of negative integers. In the case of the $ GU\!E $ Tracy-Widom distribution, the random variables use a rescaling of the prolate hyperspheroidal wave functions that were introduced by Heurtley and Slepian in the context of circular optical mirrors.

To illustrate the utility of such a structure, we rescale the largest eigenvalue of the $ GU\!E $ written as a maximum of $N$ independent random variables with the classical Poisson approximation for sums of indicators. We use for this the Okamoto-Noumi-Yamada theory of the sigma-form of the Painlev\'e equation applied to random matrix theory by Forrester-Witte. By doing so, we find a new expression for the cumulative distribution function of the $ GU\!E $ Tracy-Widom distribution which is shown to be equivalent to the classical one using manipulations \`a la Forrester-Witte.
\end{abstract}

\maketitle

\setcounter{tocdepth}{1}
\tableofcontents 


%
%

\newpage
\section{Introduction}\label{Sec:Intro}

\subsection{Motivations}\label{SubSec:Intro:Motivations}

The classical Tracy-Widom distributions are indexed by a parameter $ \beta \in \ensemble{1, 2, 4} $. Of particular interest is the case $ \beta = 2 $ ($ \gue $ Tracy-Widom distribution) that initially appeared in the fluctuations of the largest eigenvalue of a $ \gue $ random matrix \cite{TracyWidomPLB, TracyWidomCMP} and that meanwhile exhibited a universal feature in a variety of models: longuest increasing subsequence of a random uniform permutation \cite{BaikDeiftJohansson}, last passage percolation \cite{JohanssonLPP}, log-Gamma polymer partition function \cite{BorodinCorwinRemenik, COSZ, NguyenZygouras, OSZ}, asymmetric simple exclusion process \cite{TWasep1, TWasep2, TWasep3}, value in $0$ of the Hopf-Cole solution to the KPZ equation \cite{AmirCorwinQuastel, BorodinGorinKPZ, CalabreseLeDoussalRosso, DotsenkoKPZ, SasamotoSpohn1, SasamotoSpohn2, SasamotoSpohn3}, O'Connell-Yor semidiscrete directed polymer partition function \cite{OConnellYor, BorodinCorwinFerrari, ImamuraSasamoto}, etc. Surveys documenting this rich universally class include \cite{JohanssonRandomGrowthRMT, JohanssonHouches, MajumdarSurveyUlam, MajumdarSchehrTopVp, BorodinGorinSurvey, BorodinPetrovIntProb} and we refer to them for further references. 


\medskip

So far, three descriptions of the $ \gue $ Tracy-Widom distribution were given:
\begin{enumerate}


\item its probability cumulative function is given by a Fredholm determinant \cite{TracyWidomCMP} under the form
\begin{align}\label{EqTW:detFredholm}
\Prob{ \TW_2 \leq s } := \det(I- \Kb_{\! \Ai_s})_{L^2(\Rr_+)} 
\end{align}

Here, $ \TW_2 $ is a random variable Tracy-Widom-distributed, $ \Kb_{\! \Ai_s} $ is an operator acting on $ L^2(\Rr_+) $ according to $ \Kb_{\! \Ai_s} f = \int_{\Rr_+} K_{\! \Ai_s}(\cdot, y) f(y) dy $ (for all $f \in L^2(\Rr_+)$) and $ K_{\! \Ai_s} $ is the ``kernel''/function given by 
\begin{align}\label{Def:Kernel:AiWithInt}
K_{\! \Ai_s}(x, y) := \int_{\Rr_+} \Ai_s(x + u) \Ai_s(y + u) du, \qquad \Ai_s := \Ai(\cdot + s)
\end{align}

The Airy function $ \Ai $ is the solution of the ordinary differential equation $ y''(t) = t y(t) $ with initial condition $ y(0) = 3^{-2/3} \Gamma(\frac{2}{3})\inv $ and decay at infinity satisfying $ y(x) \to 0 $ when $ x \to +\infty $. An integral representation of this function is given by 
\begin{align*}
\Ai(t) = \int_{\Rr} \cos\prth{ t \xi + \frac{\xi^3}{3} } \frac{d\xi}{2\pi}
\end{align*}

$ $


\item This last probability cumulative function is also given by 
\begin{align}\label{EqTW:Painleve}
\Prob{ \TW_2 \leq s } := \exp\prth{ - \int_\Rr (x-t)_+ \, q(t)^2  dt  }, \qquad x_+ := \max\{x, 0\}
\end{align}
where $ q $ is the Hastings-McLeod solution of the Painlev\'e II equation given by\footnote{The general Painlev\'e II equation is $ q''(t) = t q(t) + 2 q(t)^3 + \alpha $ for $ \alpha \in \Rr $ ; the Hastings-McLeod equation is $ q''(t) = t q(t) + 2 q(t) \abs{q(t)}^\beta $. The Hastings-McLeod solution corresponds to $ \alpha = 0 $/$ \beta = 2 $ and is the only case where solutions can be bounded on $ \Rr $, see \cite{HastingsMcLeod}.} \cite[(1.11), (1.12)]{TracyWidomCMP} 
\begin{align}\label{Def:PainlevéII}
q''(t) = t q(t) + 2 q(t)^3 + \alpha, \hspace{+1cm} \alpha = 0, \qquad  q(t) \equivalent{t \to +\infty }  \Ai(t), \qquad q(t) \equivalent{t \to -\infty }  \sqrt{-\frac{t}{2} }
\end{align}

The Painlev\'e II ODE admits moreover a representation as a solution of an integrable PDE given in terms of a Lax pair/isomonodromy problem discovered by Flaschka and Newell \cite{FlaschkaNewell}/\cite[32.4 (iii) p. 728]{NISThandbook}. Bloemendal and Virag proved that the solution of this PDE involving half of this Lax pair can be expressed as a probability of non-explosion for the solution of a stochastic Riccatti equation related with limits of spiked matrix models \cite[(1.7), (1.8)]{BloemendalVirag} (see also the introduction of \cite{BorotNadal}).


$ $


\item A last description uses the Edelman-Sutton theory of stochastic operators \cite{EdelmanSutton, SuttonPhD} formalised by Ramirez, Rider and Virag \cite{RamirezRiderVirag} and is valid not only for the classical cases $ \beta \in \ensemble{1, 2, 4} $ but for for all $ \beta > 0 $. Define the Stochastic Airy Operator by
\begin{align*}
SAO_\beta := \frac{d^2}{dx^2} - x - \frac{2}{\sqrt{\beta} } \dot{B}(x)
\end{align*}
where $ \dot{B} $ is a white noise (the distributional derivative of a Brownian motion), i.e. a Gaussian distribution whose covariance ``function'' is given by $ \Ee ( \dot{B}(x) \dot{B}(y) ) = \delta_0(x - y) $. This operator is of the Sturm-Liouville/killed diffusion type and acts on the modified Sobolev space 
\begin{align*}
H^1_{0, *} := \ensemble{ f \in \Ce^1(\Rr_+) \, / \, f(0) = 0, \int_{\Rr_+} \prth{ f'(x)^2 + (1 + x)f(x)^2 } dx < +\infty }
\end{align*}

The classical theory of compact operator can be proven to apply a.s. to the random operator $ SAO_\beta $ in the space $ H^1_{0, *} $ \cite{EckhardtGesztesyNicholsTeschl, Minami}~; its eigenvalues can be written as $ 1 > \lambda_1(\beta) > \lambda_2(\beta) > \dots > 0 $. Using the minimax representation of Courant-Fischer, one thus has $ \lambda_1(\beta) = \sup_{f \in  H^1_{0, *}, \norm{f}_* = 1} \bracket{f, SAO_\beta f}_{\!*} $ (with obvious notations). Passing to the limit on tri-diagonal models of random matrices yields $ \TW_\beta \eqlaw \lambda_1(\beta) $ (see \cite[conj. 3.10-6.5]{EdelmanSutton} \& \cite[p. 921]{RamirezRiderVirag}) and in particular, using an integration by parts
\begin{align}\label{EqTW:SAO}
\TW_2 \eqlaw \sup_{f \in H^1_{0, *}, \norm{f}_* = 1 } \ensemble{ \sqrt{2} \int_{\Rr_+} f(x)^2 dB(x) - \int_{\Rr_+} \prth{ f'(x)^2 + x f(x)^2} dx } 
\end{align}

\end{enumerate}

\medskip

In view of these last results, a natural question can be posed:

\begin{question}\label{Question:TW}
Can one find an algebraic operation that allows to get Tracy-Widom distributions from independent random variables~?
\end{question}

Question \ref{Question:TW} was asked e.g. during a workshop at the American Mathematics Institute organised by T. Tao and V. Vu in the fall semester 2010 \cite[prob. (15)]{AIMproblems}\footnote{
The question asks for a sequence of i.i.d. random variables, but they can always be obtained out of independent random variables by inversion of the cumulative distribution function, see remark~\ref{Rk:MaxWithIIDs}.
}, in \cite[pb. 7, after (16)]{DeiftOpenProblems} and in \cite[Qu. 2.4]{BarhoumiKac1}.

\medskip

We answer question \ref{Question:TW} with the following description of $ \TW_2 $ (theorems~\ref{Thm:TW2max} and~\ref{Thm:TW2max:bis} in the sequel)~:

\begin{shaded}
\begin{theorem}[$ \TW_2 $ is a maximum of a sequence of independent random variables]\label{Thm:TW2max}
There exist independent random variables $ (Z_k(\Ai))_{k \geq 0} $ such that
\begin{align}\label{EqMax:TW2}
\boxed{\TW_2 \eqlaw \max_{k \geq 0} Z_k(\Ai)  }  
\end{align}
\end{theorem}
\end{shaded}

Two random variables $ Z_k(\Ai) $ are defined in lemma~\ref{Lemma:ExistenceZkAi}/\eqref{Def:Law:PsiSquare:phi} and in theorem~\ref{Thm:TW2max:bis}/\eqref{Def:Law:PsiSquare:TW2bis}. The description given in \eqref{Def:Law:PsiSquare:TW2bis} uses the $ L^2(\Rr_+) $-normalised eigenvectors of $ \Hb(\Ai_s) $ or equivalently the eigenvectors of the Tracy-Widom commuting differential operator $ \frac{d}{dx}x\frac{d}{dx} - x(x + s) $ acting on $ L^2(\Rr_+) $ \cite[IV-B p. 165]{TracyWidomCMP} which are a particular rescaling of the \textit{prolate hyperspheroidal wave functions} introduced independently by Slepian \cite{SlepianIV} and Hurtley \cite{HeurtleyI, HeurtleyII} 
(see also \cite{FriedenPSWF}, \cite[ch. 30.12 p. 704]{NISThandbook} and \cite{CasperGrunbaumYakimovZurrianAlgAiryComm, GrunbaumVasquezZubelli})~; we prove this last fact and give a quick summary of results on these functions in Annex~\ref{Annex:PSWF}.

\medskip
\subsection{Max-independence structures}\label{SubSec:Intro:MaxIndepStruct}

\subsubsection{Detropicalised motivation~: independence for traces and total parts}

Two of the most classical models in integrable probability are given by the $ \gue_N $ and the Schur measure, defined respectively in \eqref{Def:Law:GUE} and \eqref{Def:Schur:Measure}. The eigenvalues $ (\lambdab_{k, N})_{1 \leq k \leq N} $ of $ \Mb_{\! N} \sim \gue_N $ satisfies the identity in law (which is in fact a deterministic identity given by the invariance by conjugacy of the trace)
\begin{align*}
\sum_{k = 1}^N \lambdab_{k, N} \eqlaw \sum_{k = 1}^N \Mb_{\! N}[k, k]
\end{align*}

The sum in the RHS is composed of independent (identically distributed) Gaussians. An identity of the same type is given for the parts of the random partition $ \lambdab^{\Schur(\pb, \qb)} := (\lambdab_k^{\Schur(\pb, \qb)})_{k \geq  1} \sim \Pp_{\Schur(\pb, \qb)} $ where $ \pb := (p_k)_k $, $ \qb := (q_k)_{k \geq 1} $ and $ p_i, q_j \in (0, 1) $~:
\begin{align*}
\sum_{k \geq 1}  \lambdab_k^{\Schur(\pb, \qb)} \eqlaw \sum_{i, j \geq 1} \Geom(p_i q_j)
\end{align*}

Here, $ \Geom(q) $ designates a random variable geometrically distributed with the convention $ \Prob{\Geom(q) = k} = (1 - q) q^k \Unens{k \in \Nn } $~; see \eqref{Eq:SchurMeas:LawTotaSum} for a proof. In the same vein as for the $ \gue $, the random variables in the RHS of the previous identity are independent.

\medskip

One could call such identities in law between sums of dependent random variables and sums of independent random variables a \textit{sum-independence structure}. In this paper, we will be concerned with a \textit{tropical} version (see \cite{KirillovTropical}) of these two identities~:
\begin{align*}
\max_{k \geq 1}  \lambdab_k  \eqlaw \max_{k \geq 1} W_k
\end{align*}
where the random variables $ (W_k)_{k \geq 1} $ in the RHS are independent. We call such an identity a \textit{max-independence structure}\footnote{It also includes the case of a minimum, so we could also call it a \textit{tropical independence structure}, but in the same vein convexity relates to concavity, we will stick to this terminology.}.

\medskip
\subsubsection{A previous instance in the literature}

Max-independence structures for eigenvalues of random matrix ensembles made a first apparition in the Ginibre ensemble \cite{Ginibre}. This ensemble, noted $ \GinE_N $ is the probability space $ (\Me_N(\Cc), \Pp_N) $ where $ \Me_N(\Cc) $ is the space of $ N \times N $ complex matrices and $ \Pp_N $ is the Gaussian measure $ \Pp_N(dM) = \pi^{-N^2} e^{- \trace(M^* M) } dM $. If $ (\lambdab_{k, N})_{1 \leq k \leq N} $ denotes the eigenvalues of $ \Gb_{\! N} \sim \GinE_N $, ordered by modulus size, i.e. $ \abs{\lambdab_{1, N} } \geq \abs{\lambdab_{2, N}} \geq \dots $, one has the formula
\begin{align}\label{EqLaw:Ginibre}
\abs{\lambdab_{1, N}} \eqlaw \max_{1 \leq k \leq N} \sqrt{\gammab_{ k}}, \qquad \gammab_{ k} \sim \Gammab( k)
\end{align}
with independent random variables $ \gammab_{ k} \eqlaw \sum_{j = 1}^k \ee^{(j)} $, the $ \ee^{(j)} $ being i.i.d. exponential random variables. Using \eqref{EqLaw:Ginibre}, Rider showed that \cite[thm. 1]{RiderGinibre}
\begin{align}\label{CvLaw:Ginibre:Rider}
\sqrt{4N \alpha_N}\prth{\! \frac{\abs{\lambdab_{1, N}}}{\sqrt{N}} - 1 - \sqrt{\frac{\alpha_N}{4N}}\, } \cvlaw{N}{+\infty} \Gumbel(1), \qquad \alpha_N := \log(N/2\pi) - 2\log\log(N)
\end{align}
where $ \Gumbel(1) $ designates the Gumbel distribution, i.e. the law of $ -\log\ee $.

See also \cite[ch. 4.3.7, ex. 4.7.2, ch. 6.4]{BenHougKrishnapurPeresVirag} for more details and the blog \cite{ChaffaiBlogGinibre} for a pedagogical introduction.

\medskip

The presence of such an independence structure in $ \GinE_N $ asks legitimately the question of its generality in any matrix ensemble. But more generally, since the universality class of $ \TW_2 $ previously described is so rich, a more natural question can be asked:

\begin{question}\label{Q:IndependenceStructure}
Is the max-independence structure already present before renormalisation in all the models previously cited ?
\end{question}

A positive answer to question~\ref{Q:IndependenceStructure} is provided in the following cases that will be described in more details throughout the article~:
\begin{itemize}

\item the continuum random polymer (\S~\ref{SubSec:TW:HMHoperators}), 

\item random matrix ensembles (\S~\ref{Sec:RMT}), 

\item the symmetric Schur measure (\S~\ref{Sec:SchurMeasure}).

\end{itemize}


In the case of the symmetric Schur measure, one has a very remarkable structure~: an i.i.d. randomisation of the (deterministic) sequence of integers (see \S~\ref{Subsec:SchurMeasure:General}). Each of these cases has an equality in law analogous to \eqref{EqMax:TW2}. For instance, in the case of the $ GU\!E $, one has

\begin{shaded}
\begin{theorem}[Theorem~\ref{Thm:GUEmax} in the sequel]\label{Thm:Init:GUEmax}
There exist independent random variables $ (W_k)_{0 \leq k \leq N - 1} $ such that, if $ \lambdab_{1, N} $ designates the largest eigenvalue of a random matrix $ \Hb_{\!N}\sim GU\!E_N $,
\begin{align}\label{EqMax:GUE:init}
\boxed{ \lambdab_{1, N} \eqlaw \max_{0 \leq k \leq N - 1} W_k }
\end{align}
\end{theorem}
\end{shaded}

While $ W_0 \sim \Ns(0, 1) $, the law of $ W_k $ for $ k \geq 1 $ is explicitely given in theorem~\ref{Thm:GUEmax} but requires to use the $ \sigma $-form of the Painlev\'e IV equation studied by Okamoto \cite{Okamoto0, Okamoto1, Okamoto2, Okamoto3}, simplified by Noumi-Yamada \cite{NoumiYamadaPainleve, NoumiPainleveBook} and applied in Random Matrix Theory by Forrester-Witte \cite{ForresterWitteTau2and4} (see also \cite{ForresterWitteTau3and5, ForresterWittePainleve3p5} and \cite[ch. 8 p. 328]{ForresterBook} for other cases that we will not treat in this paper).

\medskip
\subsection{Rescaling of the largest eigenvalue of the $ GU\!E $}\label{SubSec:Intro:GUE}

The first historical appearance of the Tracy-Widom distribution in \cite{TracyWidomPLB, TracyWidomCMP} was obtained by the rescaling of the largest eigenvalue of the $ GU\!E_N $ when $ N\to+\infty $. From this perspective, the max-independence structure given by \eqref{EqMax:GUE:init}/\eqref{EqMax:GUE} is particularly interesting as one can use the classical Poisson approximation for a maximum of independent random variables: writing
\begin{align*}
\ensemble{\max_{0 \leq k \leq N - 1} W_k \leq x} = \ensemble{\sum_{k = 0}^{N - 1} \Unens{W_k > x} = 0}
\end{align*}
and analysing the parametric sum of independent indicators $ S_N(x) := \sum_{k = 0}^{N - 1} \Unens{W_k > x} $, one can perform a Poisson approximation $ S_N(x) \approx \Poisson(\mu_N(x)) $ where $ \mu_N(x) := \Esp{S_N(x)} = \sum_{k = 0}^{N - 1} \Prob{W_k > x} $. The error in such an approximation is moreover explicit, given for instance by the Chen-Stein method \cite{Chen}. As a result, one can try to put in a more classical probabilistic perspective the result of Tracy and Widom, for instance in the vein of Rider's analysis of the edge fluctuations in the complex Ginibre ensemble \cite{RiderGinibre}. Such an analysis is performed in \S~\ref{SubSec:RMT:Painleve} and its main result reads:

\begin{shaded}
\begin{theorem}[Theorem~\ref{Thm:NewExprTW2} in the sequel]\label{Thm:Init:NewExprTW2}
One has
\begin{align}\label{EqTW:PainlevureNouveau:Init}
\Prob{\frac{\lambdab_{1, N} - 2 \sqrt{N}}{N^{-1/6}} \leq \sqrt{2}\, t} \tendvers{N}{+\infty } \Prob{\TW_2 \leq \sqrt{2}\, t} = \exp\prth{ - \int_\Rr (x - t)_+ \Qe(t)^2 dt }  
\end{align}
with $  x_+ := x \Unens{x \geq 0} $ and 
\begin{align}\label{Def:Q^2PII:Init}
\Qe(x) :=   \exp\prth{ - \int_x^{+\infty} \delta \sigma }, \qquad \delta \sigma := \sigma(\cdot ; 1) - \sigma(\cdot ; 0)
\end{align}
where $ \sigma(\cdot ; a) $ is the solution of the $ \sigma $-form of the Painlev\'e II equation given in \eqref{Def:Painleve:SigmaForm:II}.
\end{theorem}
\end{shaded}

We would like to emphasize that the very structure of the limit is already present for fixed $N$ with a function $ \delta \sigma^{(N)} $ in $ \mu_N(x) $ (which is in fact itself already of this form with a Stieltjes integral), and that it is preserved in the form~\eqref{EqTW:PainlevureNouveau:Init} after passing to the limit.

Comparing \eqref{EqTW:Painleve} and \eqref{EqTW:PainlevureNouveau:Init}, one sees that $ q = \Qe $~; we give a direct proof of this fact in lemma~\ref{Lemma:q=Q} using the Forrester-Witte theory.

\medskip

In fact, given the discrepancy between \eqref{CvLaw:Ginibre:Rider} and \eqref{EqTW:PainlevureNouveau:Init}, one can moreover ask the question~:

\begin{question}\label{Q:KPZUniversalityMaxIndep}
Which features of the max-independence structure allows to conclude that one is in the KPZ universality class as opposed to (say) the Gumbel universality class~?
\end{question}

We will see that the very structure of the proof by Poisson approximation allows to give a criteria that answers question~\ref{Q:KPZUniversalityMaxIndep} in theorems~\ref{Thm:CvMaxIndep} and \ref{Thm:CvMaxRandomisedIntegers}. 

In both theorems, this will be a (right) Large Deviation Principle for the sequence $ (W_k)_k $ that will trigger the apparition of the limiting distribution, be it the Tracy-Widom, the Gumbel distribution or any other extreme value distribution. We would like to emphasize that the case of the Schur measure is particularly interesting~: the max-independence structure is given by an i.i.d. randomisation of the sequence of the integers (theorem~\ref{Thm:SchurMeas}) and its rescaling using the Poisson approximation gives \textit{directly} the convolutive structure \textit{without any approximation} other than the Poisson one (see the proof of theorem~\ref{Thm:CvMaxRandomisedIntegers}).

\medskip
\subsection{Organisation of the paper}\label{SubSec:Intro:Plan}

The plan of the article is as follows~: 
\begin{itemize}

\item we prove theorem \eqref{Thm:TW2max} in section \ref{Sec:TW} and extend it to the case of $ \TW_1 $ in \S~\ref{SubSec:TW:GOE} and to the continuum random polymer in \S~\ref{SubSec:TW:HMHoperators}.

\medskip
\item We investigate independence structure in random matrix ensembles in section~\ref{Sec:RMT} and prove there theorems~\ref{Thm:Init:GUEmax}/\ref{Thm:GUEmax} and \ref{Thm:Init:NewExprTW2}/\ref{Thm:NewExprTW2} in addition to theorem~\ref{Thm:CvMaxIndep} that answers question~\ref{Q:KPZUniversalityMaxIndep}. We would like to draw attention here to the problem~\ref{Q:CouplingGUE:max} which is, in the author's opinion, the most important in the article, and probably the tip of a very big iceberg.

\medskip
\item We analyse the case of the symmetric Schur measure in section~\ref{Sec:SchurMeasure}, starting with the Plancherel measure~; we give a general criteria of convergence towards any limiting distribution in theorem~\ref{Thm:CvMaxRandomisedIntegers}.  

\medskip
\item We give in Annex~\ref{Annex:SymFunc} some relevant notations on symmetric functions and the Schur measure and in Annex~\ref{Annex:PSWF} a summary of results on hyperspheroidal wave functions.

\end{itemize}

\medskip
\medskip
\section{$ \TW_2 $ is a maximum of independent random variables}\label{Sec:TW}
\medskip
\subsection{First approach}\label{SubSec:TW:FirstApproach}

\subsubsection{Prerequisites}

For an operator $ \Kb_{\! t} : L^2(\Rr_+) \rightarrow L^2(\Rr_+) $ of kernel $ K_t $, define $ \dot{K}_t := \frac{d}{dt} K_t $ and $ \dot{\Kb}_{\! t} : L^2(\Rr_+) \rightarrow L^2(\Rr_+) $ its associated operator. For $ \phi, \psi \in L^2(\Rr_+) $, we define $ \phi_t := \phi(\cdot + t) $, $ \psi_t := \psi(\cdot + t) $ and 
\begin{align*}
K_t(x, y) := \int_{\Rr_+} \phi_t(x + u) \psi_t(y + u) du  = \int_t^{+\infty} \phi(x + u) \psi(y + u) du 
\end{align*}

Define $ \Hb(\phi) $ to be the Hankel operator of symbol $ \phi $ acting on $ L^2(\Rr_+) $, i.e. the operator of kernel $ (x, y) \mapsto \phi(x + y) $. Then, one has
\begin{align*}
\Kb_{\! t} = \Hb(\psi_t) \Hb(\phi_t)  
\end{align*}

Set 
\begin{align*}
\psi^*(f) := \bracket{f, \psi}_{\! L^2(\Rr_+)}
\end{align*}

Then, 
\begin{align}\label{Eq:Operator:DotHsquare}
\dot{K}_t(x, y) = -\phi(x + t) \psi(y + t) = -\phi_t\otimes \psi_t(x, y) \qquad\Longleftrightarrow\qquad \dot{\Kb}_{\! t} = -\phi_t \otimes \psi_t^*
\end{align}

\medskip

Suppose now that $ \Kb_{\! t} $ has all its eigenvalues in $ (0, 1) $. Writing the following determinants and traces on $ L^2(\Rr_+) $, we have
\begin{align*}
\frac{\det(I - \Kb_{\! t})}{\det(I - \Kb_{\! a})} & = \exp\prth{ \trace(\log(I - \Kb_{\! t}) - \log(I - \Kb_{\! a} ) ) } \\
                   & = \exp\prth{ -\int_a^t \trace( (I - \Kb_{\! s})\inv \dot{\Kb}_{\! s} ) ds } \\
                   & = \exp\prth{ \int_a^t \trace( (I - \Kb_{\! s})\inv \phi_s \otimes \psi_s^* ) ds } \\
                   & = \exp\prth{ \int_a^t \bracket{  (I - \Kb_{\! s})\inv \phi_s ,  \psi_s }_{\! L^2(\Rr_+) }  ds }
\end{align*}

For $ a = +\infty $, as $ \Kb_{\! a} = 0 $, we get
\begin{align*}
\det(I - \Kb_{\! t})_{\! L^2(\Rr_+) }  & = \exp\prth{ -\int_t^{+\infty} \bracket{  (I - \Kb_{\! s})\inv \phi_s ,  \psi_s }_{\! L^2(\Rr_+) }  ds } \\
                 & = \exp\prth{ -\int_t^{+\infty} \sum_{\ell \geq 0} \bracket{  \Kb_{\! s}^\ell \phi_s ,  \psi_s }_{\! L^2(\Rr_+) }  ds } 
\end{align*}

And in the case where $ \phi = \psi $, one finally gets
\begin{align}\label{Eq:Fredholm:HankelSquareWithExp}
\det\prth{ I - \Hb(\phi_t)^2}_{\! L^2(\Rr_+) }  = \prod_{\ell \geq 0} \exp\prth{ -\int_t^{+\infty} \norm{  \Hb(\phi_s)^\ell \phi_s  }_{ L^2(\Rr_+) }^2  ds }
\end{align}

Suppose that $ s \mapsto \norm{  \Hb(\phi_s)^\ell \phi_s  }_{\! L^2(\Rr_+) }^2 \notin L^1(\Rr) $. Then, one clearly gets the existence of a random variable $ Z_\ell(\phi) $ such that
\begin{align}\label{Def:Law:PsiSquare:phi}
\Prob{Z_\ell(\phi) \leq t} =  \exp\prth{ -\int_t^{+\infty} \norm{  \Hb(\phi_s)^\ell \phi_s  }_{ L^2(\Rr_+) }^2  ds }
\end{align}
which implies that
\begin{align}\label{Eq:DetMax:Phi}
\det(I - \Hb(\phi_t)^2\, )_{\! L^2(\Rr_+) } = \prod_{\ell \geq 0} \Prob{Z_\ell(\phi) \leq t} = \Prob{\max_{\ell \geq 0} Z_\ell(\phi) \leq t} 
\end{align}

\subsubsection{Proof of theorem~\ref{Thm:TW2max}}

We now consider the case $ \phi = \Ai $ in \eqref{Def:Law:PsiSquare:phi}. It is clear that \eqref{Def:Kernel:AiWithInt} translates at the operator level into
\begin{align}\label{Eq:SquareHankel:Airy}
\Kb_{\! \Ai_s} = \Hb(\Ai_s)^2
\end{align}

To prove theorem~\ref{Thm:TW2max}, one uses \eqref{Eq:DetMax:Phi} and the

\begin{shaded}
\begin{lemma}[Existence of a random variable $ Z_\ell(\Ai) $]\label{Lemma:ExistenceZkAi}
One has
\begin{align*}
s \mapsto \norm{  \Hb(\Ai_s)^\ell \Ai_s  }_{\! L^2(\Rr_+) }^2 \notin L^1(\Rr)
\end{align*}
\end{lemma}
\end{shaded}

\begin{proof}
For $ \blue{\ell = 0} $, one has for $ s < 0 $
\begin{align*}
\norm{  \Ai_s  }_{\! L^2(\Rr_+) }^2 & = \int_{\Rr_+} \Ai(s + x)^2 dx = \int_s^{+\infty} \Ai(x)^2 dx \geq  \int_s^0 \Ai(x)^2 dx =  \abs{s} \int_{-1}^0 \Ai(x\abs{s})^2 dx \\
                  & = \Omega( \sqrt{\abs{s}} )
\end{align*}
which shows that the function is not integrable in a neighbourhood of $ -\infty $.

To show that
\begin{align*}
C := \int_{\Rr_+} \Ai(x)^2 dx  < \infty  
\end{align*}
we have used the estimate \cite[(2.44) p. 14]{ValleeSoares}/\cite{KatoriTanemuraAiry}
\begin{align*}
\Ai(x) = O(x^{-1/4}e^{- \frac{2}{3} x^{3/2} }), \qquad x \to +\infty  
\end{align*}
and to show that
\begin{align*}
\int_{-1}^0 \Ai(x\abs{s})^2 dx = \Omega\prth{\frac{1}{\sqrt{\abs{s}} }}, \qquad s \to +\infty
\end{align*}
we have used \cite[(9.8.1) p. 199, (9.8.20), (9.8.21) p. 200]{NISThandbook} 
\begin{align*}
\Ai(-x) & := M(x) \sin\theta(x), \qquad M(x)^2 := \Ai(-x)^2 + \Bi(-x)^2, \qquad \theta(x) := \arctan\prth{ \frac{ \Ai(-x)}{\Bi(-x) } } \\
M(x)^2 & \sim \frac{1}{ \pi \sqrt{x} } \sum_{k \geq 0} \frac{ x^{-3k} }{k! } (-1)^k\frac{1 \cdot 3 \cdot 5 \cdots (6k - 1) }{96^k }, \quad \theta(x) \sim \frac{\pi}{4} + \frac{2}{3} x^{\frac{3}{2} } \prth{ 1 - \frac{5}{32} x^{-3} +  O(x^{-6}) }
\end{align*}
which gives in particular  
\begin{align}\label{Eq:Estimate:AiryMinusInf}
\Ai(-x)^2 \sim \frac{1}{\pi \, x^{1/2}} \sin\prth{ \frac{\pi}{4} + \frac{2}{3} x^{\frac{3}{2} } }^2, \qquad x \to +\infty
\end{align}

\medskip

For $ \blue{\ell = 1} $, one has
\begin{align*}
\norm{ \Hb(\Ai_s)  \Ai_s  }_{ L^2(\Rr_+) }^2 & = \int_{\Rr_+} \prth{ \int_{\Rr_+} \Ai_s(x + y) \Ai_s(y) dy }^{\!\! 2} dx   =: \int_{\Rr_+} K_{\! \Ai}(s + x, s)^2 dx
\end{align*}

One also has
\begin{align}\label{Def:Kernel:AiChristoffelDarboux}
K_{\! \Ai }(x, y) = \frac{ \Ai(x )\Ai'(y ) - \Ai(y )\Ai'(x ) }{x-y} \Unens{x \neq y} + (\Ai'(x)^2 - x \Ai(x)^2)\Unens{x = y}
\end{align}
and for $ x \neq y $, 
\begin{align*}
K_{\! \Ai }(x, y) & = - \Ai(x)\Ai(y) \frac{  \frac{\Ai'(y)}{\Ai(y)} - \frac{\Ai'(x)}{\Ai(x) } }{y - x} \\
                  & = - \Ai(x)\Ai(y) \int_0^1 \prth{ \frac{\Ai'}{\Ai} }'(t x + \overline{t} y) dt, \qquad \overline{t} := 1 - t
\end{align*}

Moreover, using $ \Ai''(x) = x\Ai(x) $, one gets
\begin{align*}
\prth{ \frac{\Ai'}{\Ai} }'\!\!\!(x) = \frac{\Ai''(x)\Ai(x) - \Ai'(x)^2}{\Ai(x)^2} = x - \prth{ \frac{\Ai'(x)}{\Ai(x)} }^{\! 2}
\end{align*}
hence
\begin{align*}
K_{\! \Ai }(x, y)  & = \Ai(x)\Ai(y) \prth{ - \frac{x + y}{2} + \int_0^1 \prth{ \frac{\Ai'}{\Ai}(t x + \overline{t} y) }^{\! 2} dt } \\
                 & =: \Ai(x)\Ai(y)  R_\Ai(x, y)  
\end{align*}
with
\begin{align*}
R_f(x, y) :=    - \frac{x + y}{2} + \int_0^1 \prth{ \frac{f'}{f}(t x + \overline{t} y) }^{\! 2} dt  
\end{align*}

This implies for $ s = - \abs{s} < 0 $
\begin{align*}
\norm{ \Hb(\Ai_s)  \Ai_s  }_{ L^2(\Rr_+) }^2 & = \int_{\Rr_+} K_{\! \Ai}(s + x, s)^2 dx \\
                 & = \int_{\Rr_+} \Ai(x + s)^2 \Ai(s)^2  R_\Ai(s + x, s)^2 dx =  \Ai(s)^2 \int_s^{+\infty } \Ai(X)^2   R_\Ai(X, s)^2 dX \\ 
                 & \geq \frac{\Ai(s)^2}{4} \int_s^{+\infty } \Ai(X)^2  \prth{ X + s }_-^2 dX = \frac{\Ai(s)^2}{4} \int_s^{-s} \Ai(X)^2  \prth{ X + s }^2 dX\\
                 & = \frac{\Ai(s)^2}{4}  \abs{s}^3  \int_{-1}^1 \Ai( \abs{s} u)^2  \prth{u - 1}^2 du \\
                 & \geq \frac{\Ai(s)^2}{4}  \abs{s}^3 \int_{-1}^0 \Ai( \abs{s}u)^2  \prth{1 - u}^2 du \\
                 & = \Omega\prth{  s^2 } \qquad\mbox{using \eqref{Eq:Estimate:AiryMinusInf}.}
\end{align*}
%
%
%
%
%
%

As a result, $ s \mapsto \norm{ \Hb(\Ai_s)   \Ai_s  }_{ L^2(\Rr_+) }^2 $ is not integrable on $ \Rr $.

\medskip\medskip 

In the general case, one has for $ \blue{\ell = 2m} $ and $ s < 0 $ 
\begin{align*}
\norm{ \Hb(\Ai_s)^{2m} \! \Ai_s  }_{ L^2(\Rr_+) }^2 & = \int_{\Rr_+} \prth{ \int_{\Rr_+} K_{\! \Ai_s}^{*m}(x , y ) \Ai_s(y) dy }^{\!\! 2} dx \\
                 & = \int_{\Rr_+}  \prth{ \int_{\Rr_+ \times \Rr_+^{m - 1} } K_{\! \Ai_s}(x, t_1) \prod_{j = 1}^{m - 2} K_{\! \Ai_s}(t_j, t_{j + 1})  \,  K_{\! \Ai_s} (t_{m - 1}, y) d\tb \, \Ai_s(y) dy }^{\!\! 2} dx \\
                 & = \int_{\Rr_+} \Ai_s(x)^2 \Bigg( \int_{\Rr_+ \times \Rr_+^{m - 1} } \prod_{j = 1}^{m - 2} \Ai_s(t_j)^2 \\
                 & \hspace{+3cm} \times R_{\! \Ai_s}(x, t_1) \prod_{j = 1}^{m - 2} R_{\Ai_s}(t_j, t_{j + 1})  \,  R_{\Ai_s} (t_{m - 1}, y) d\tb \, \Ai_s(y)^2  dy \Bigg)^{\!\! 2} dx \\
                 & \geq \frac{1}{2^m} \int_{\Rr_+} \Ai_s(x)^2 \Bigg( \int_{\Rr_+ \times \Rr_+^{m - 1} } \prod_{j = 1}^{m - 2} \Ai_s(t_j)^2 \\
                 & \hspace{+1cm} \times (2\abs{s} - x - t_1)_{\! +} \prod_{j = 1}^{m - 2} (2\abs{s} - t_j - t_{j + 1})_{\! +}  \,  (2\abs{s} - t_{m - 1} - y)_{\! +} d\tb \, \Ai_s(y)^2  dy \Bigg)^{\!\! 2} dx \\
                 & = \frac{1}{2^m} \int_{[s, +\infty)} \Ai(X)^2 \Bigg( \int_{[s, +\infty)^m } \prod_{j = 1}^{m - 2} \Ai(T_j)^2 \\
                 & \hspace{+1cm} \times ( X + T_1)_{\! -} \prod_{j = 1}^{m - 2} (T_j + T_{j + 1})_{\! -}  \,  ( T_{m - 1} + Y)_{\! -} d\Tb \, \Ai (Y)^2  dY \Bigg)^{\!\! 2} dX \\
                 & \geq \frac{\abs{s}^{4m + 1}}{2^m} \int_{[-1, 0)} \Ai(\abs{s}u)^2 \Bigg( \int_{[-1, 0)^m } \prod_{j = 1}^{m - 2} \Ai(\abs{s}\tau_j)^2 \\
                 & \hspace{+1cm} \times ( u + \tau_1)_{\! -} \prod_{j = 1}^{m - 2} (\tau_j + \tau_{j + 1})_{\! -}  \,  ( \tau_{m - 1} + v)_{\! -}  \, \Ai (\abs{s}v)^2  dv \, d\taub \Bigg)^{\!\! 2} du \\ 
                 & = \Omega\prth{ \abs{s}^{3m + 1/2} } \qquad \mbox{using \eqref{Eq:Estimate:AiryMinusInf}. }
\end{align*}

\medskip 

And for $ \blue{\ell = 2m + 1} $ and $ s < 0 $, 
\begin{align*}
\!\!\! \norm{ \Hb(\Ai_s)^{2m + 1} \! \Ai_s  }_{ L^2(\Rr_+) }^2 & = \int_{\Rr_+} \prth{ \int_{\Rr_+}    K_{\! \Ai_s}^{*m} * H_{\! \Ai_s}(x , y ) \Ai_s(y) dy }^{\!\! 2} dx \\
                 & = \int_{\Rr_+}  \prth{ \int_{\Rr_+ \times \Rr_+^m } K_{\! \Ai_s}(x, t_1) \prod_{j = 1}^{m - 1} K_{\! \Ai_s}(t_j, t_{j + 1})  \,  H_{\! \Ai_s} (t_m, y) d\tb \, \Ai_s(y) dy }^{\!\! 2} dx \\ 
                 & = \int_{\Rr_+}  \prth{ \int_{ \Rr_+^m } K_{\! \Ai_s}(x, t_1) \prod_{j = 1}^{m - 1} K_{\! \Ai_s}(t_j, t_{j + 1})  \,  K_{\! \Ai_s} (t_m  , 0) d\tb  }^{\!\! 2} dx \\ 
                 & =  \Ai_s(0)^2 \int_{\Rr_+} \Ai_s(x)^2 \Bigg( \int_{ \Rr_+^m } \prod_{j = 1}^m \Ai_s(t_j)^2  \\
                 & \hspace{+3cm} \times R_{\! \Ai_s}(x, t_1) \prod_{j = 1}^{m - 1} R_{\Ai_s}(t_j, t_{j + 1})  \, R_{\! \Ai_s}(t_m , 0) d\tb   \Bigg)^{\!\! 2} dx \\
                 & \geq \frac{\Ai_s(0)^2 }{2^m} \int_{\Rr_+} \Ai_s(x)^2 \Bigg( \int_{ \Rr_+^m } \prod_{j = 1}^m \Ai_s(t_j)^2   \\
                 & \hspace{+3cm} \times (2\abs{s} - x - t_1)_{\! +} \prod_{j = 1}^m (2\abs{s} - t_j - t_{j + 1})_{\! +}  \,  (2\abs{s} - t_m )_{\! +} d\tb \Bigg)^{\!\! 2} dx \\
                 & = \frac{\Ai(s)^2}{2^m} \int_{[s, +\infty)} \Ai(X)^2 \Bigg( \int_{[s, +\infty)^m } \prod_{j = 1}^m \Ai(T_j)^2 \\
                 & \hspace{+3cm} \times ( X + T_1)_{\! -} \prod_{j = 1}^{m - 1} (T_j + T_{j + 1})_{\! -}  \,  (\abs{s} - T_m)_{\! +} d\Tb   \Bigg)^{\!\! 2} dX \\
                 & \geq \frac{\abs{s}^{4m + 3}}{2^m} \Ai(-\abs{s})^2 \int_{[-1, 0)} \Ai(\abs{s}u)^2 \Bigg( \int_{[-1, 0)^m } \prod_{j = 1}^m \Ai(\abs{s}\tau_j)^2 \\
                 & \hspace{+4cm} \times ( u + \tau_1)_{\! -} \prod_{j = 1}^{m - 1} (\tau_j + \tau_{j + 1})_{\! -}  \,  ( 1 - \tau_m)_{\! +}  \,  d\taub \Bigg)^{\!\! 2} du \\ 
                 & = \Omega\prth{ \abs{s}^{3m + 2} } \qquad \mbox{using \eqref{Eq:Estimate:AiryMinusInf}. }
\end{align*}

This concludes the proof.
\end{proof}

\medskip
\subsection{Second approach}\label{SubSec:TW:SecondApproach}

\subsubsection{Prerequisites}

Before proving theorem~\ref{Thm:TW2max}, we introduce a slight variation of \cite[lem. 1]{Fuchs}~:

\begin{shaded}
\begin{lemma}[Fixed interval Fuchs lemma]\label{Lemma:Fuchs:Fixed}
For $ I \subset \Rr $, let $ \Kb_{\! t} : L^2(I) \to L^2(I) $ be self-adjoint ($ \Kb_{\! t}^* = \Kb_{\! t} $). Let $ (g_t, \lambda_t) $ be such that $ \Kb_{\! t} g_t = \lambda_t g_t $. Last, suppose that $ t \mapsto (\Kb_{\! t}, g_t, \lambda_t) $ is differentiable. Then,
\begin{align*}
\dot{\lambda}_t = \frac{\bracket{\dot{\Kb}_{\! t} g_t, g_t}_{\!\!L^2(I)}}{\norm{g_t}^2_{L^2(I)}}
\end{align*}
\end{lemma}
\end{shaded}


\begin{proof}
Differentiating $ \Kb_{\! t} g_t = \lambda_t g_t $ yields
\begin{align*}
\dot{\Kb}_{\! t} g_t + \Kb_{\! t} \dot{g}_t = \dot{\lambda}_t g_t + \lambda_t \dot{g}_t
\end{align*}

Taking the scalar product with $ g_t $ and dropping the index $ L^2(I) $ then yields
\begin{align*}
\bracket{\dot{\Kb}_{\! t} g_t, g_t} + \bracket{\Kb_{\! t} \dot{g}_t, g_t} = \dot{\lambda}_t  \norm{g_t}^2 + \lambda_t \bracket{\dot{g}_t, g_t}
\end{align*}

But 
\begin{align*}
\delta_t & :=  \bracket{\Kb_{\! t} \dot{g}_t, g_t} - \lambda_t \bracket{\dot{g}_t, g_t} \\
              & = \bracket{\dot{g}_t,\Kb_{\! t}^* g_t} - \lambda_t \bracket{\dot{g}_t, g_t} \\
              & = \bracket{\dot{g}_t, (\Kb_{\! t} - \lambda_t) g_t} \qquad \mbox{since }  \Kb_{\! t}^* = \Kb_{\! t}  \\
              & = 0
\end{align*}
hence the result.
\end{proof}

\medskip
\subsubsection{Proof of theorem~\ref{Thm:TW2max}}

$ $

\begin{shaded}
\begin{theorem}[Max-independence structure in $ \TW_2 $, bis repetita]\label{Thm:TW2max:bis}
One has
\begin{align}\label{EqMax:TW2:bis}
\boxed{\TW_2 \eqlaw \max_{k \geq 0} Z_k'(\Ai)  }  
\end{align}
%
%
where $ (Z_k'(\Ai))_{k \geq 1} $ are independent random variables with density
\begin{align}\label{Def:Law:PsiSquare:TW2bis}
f_{Z_k'(\Ai)}(s) := \bracket{ \Ai_s, \Psi_k^{(s)} }^{\!\!2}_{\!\! L^2(\Rr_+)}, \qquad \norm{ \Psi_k^{(s)} }^2_{  L^2(\Rr_+)} = 1
\end{align}
where $ \Psi_k^{(s)} $ is the $ k $-th eigenvector of $ \Hb(\Ai_s) $ on $ L^2(\Rr_+) $ or equivalently the eigenvector of the \textit{Tracy-Widom commuting operator} on $ L^2(\Rr_+) $ given by 
\begin{align}\label{Def:Operators:TracyWidomCommuting}
\Le_{TW, s} := DXD - X(X + s) : f(x) \mapsto \frac{d}{dx}\prth{ x \frac{df}{dx}} - x(x + s) f(x)
\end{align}
\end{theorem}
\end{shaded}


\medskip
\begin{proof} 
Recall the Lidskii formula for trace-class operators \cite[thm. 3.12.2]{SimonOperatorTh} 
\begin{align}\label{Eq:Fredholm:Lidskii}
\det(I - \Kb)_{L^2(\Rr_+) } = \prod_{k \geq 1} (1 - \lambda_k), \qquad \Kb\psi_k = \lambda_k \psi_k
\end{align}

Applying this formula for $ \Hb(\Ai_s)^2 $ does not immediately give the result as one does not have any information on the dependency in $s$ of the $k$-th eigenvalue $ \lambda_k(s) $.

Nevertheless, for a normalised eigenvector $g_t$ ($ \norm{g_t}_{L^2(\Rr_+)}^2 = 1$) associated with an eigenvalue $ \lambda(t) $, one gets
\begin{align*}
\dot{\lambda}(t) & = \bracket{ \dot{\Kb}_{\!\Ai, t} g_t, g_t}_{L^2(\Rr_+)} = -\bracket{\Ai_t \otimes \Ai_t^* g_t, g_t}_{L^2(\Rr_+)} = - \bracket{\Ai_t  , g_t}_{L^2(\Rr_+)}^2
\end{align*}
as, for $ \Kb_{\!\Ai, t} = \Hb(\Ai_s)^2 \ActsOn L^2(\Rr_+) $, one has $ \dot{\Kb}_{\!\Ai, t} = -\Ai_t \otimes \Ai_t^* $ using \eqref{Eq:Operator:DotHsquare}.

Moreover, $ \Kb_{\!\Ai, +\infty} = 0 $ and $ \Kb_{\!\Ai, -\infty} = I $ hence $ \lambda_k(\infty) = 0 $ and $ \lambda_k(-\infty) = 1 $ and
\begin{align*}
& \lambda(t) = \blue{-}\int_t^{+\infty} \dot{\lambda}(s) ds = \blue{+}\int_t^{+\infty} \bracket{ \Ai_s , g_s }_{L^2(\Rr_+)}^2  ds, \\
& \int_\Rr \bracket{ \Ai_s , g_s }_{L^2(\Rr_+)}^2  ds  = 1
\end{align*}


The Lidskii formula then implies
\begin{align*}
\det\prth{I - \Hb(\Ai_t)^2}_{L^2(\Rr_+) } & = \prod_{k \geq 1} (1 - \lambda_k(t)) \\
                & = \prod_{k \geq 1} \prth{ 1 - \int_t^{+\infty} \bracket{\Ai_s  , \Psi_s^{(k)}}_{L^2(\Rr_+)}^2 ds } \\
                & = \prod_{k \geq 1} \prth{   \int_{-\infty}^t  \bracket{\Ai_s  , \Psi_s^{(k)}}_{L^2(\Rr_+)}^2 ds }  \\
                & = \prod_{k \geq 1} \prth{   \int_{-\infty}^t  f_{Z_k'(\Ai)}(s) ds }\\
                & = \Prob{\max_{k \geq 1} Z_k'(\Ai) \leq t} 
\end{align*}
which gives the result.

It remains to give a short description of the eigenvectors $ (\Psi_k^{(s)})_{k \geq 1} $ which are the eigenvectors of $ \Le_{TW, s} $, i.e. a particular rescaling of the \textit{prolate hyperspheroidal wave functions}. This fact is proven in Annex~\ref{Annex:PSWF}.
\end{proof}

\medskip

\begin{remark}
An interesting question would be to prove (or disprove) that $ Z_k(\Ai) \eqlaw Z_k'(\Ai) $. A priori, max-independence structures are not unique since partitioning the set on which the maximum is taken gives again such a structure~; one could have for instance $ Z_k(\Ai) \eqlaw \max_{j \in A_i} Z_j'(\Ai) $ for some good sets $ (A_i)_i $ that partition $ \Nn $ (or the contrary, i.e. $ Z_k'(\Ai) $ expressed with $ Z_k(\Ai) $). The existence of a ``minimal'' decomposition could thus also be asked.
\end{remark}

\medskip
\subsubsection{Another expression of $ Z_k'(\Ai) $}

$ $

\begin{shaded}
\begin{lemma}[Another expression of the law of $ Z_k'(\Ai) $]\label{Lemma:TW2max:ter}
One also has
\begin{align}\label{Def:Law:PsiSquare:TW2ter}
\Prob{ Z_k'(\Ai) > s } := \exp\prth{ -\int_{-\infty}^s \Psi_k^{(t)}(0)^2 dt } , \qquad \norm{ \Psi_k^{(s)} }^2_{  L^2(\Rr_+) } = 1
\end{align}
\end{lemma}
\end{shaded}

Before proving lemma~\ref{Lemma:TW2max:ter}, we recall the original version of \cite[lem. 1]{Fuchs}~:

\begin{shaded}
\begin{lemma}[Varying interval Fuchs lemma]\label{Lemma:Fuchs:Varying}
Let $ I_t := [t, +\infty) \subset \Rr $. Let $ \Kb \equiv \Kb_{\! t} : L^2(I_t) \to L^2(I_t) $ be a self-adjoint ($ \Kb_{\! t}^* = \Kb_{\! t} $) such that the kernel $K$ of $ \Kb $ does not depend on $t$. Let $ (g_t, \lambda_t) $ be such that $ \Kb_{\! t} g_t = \lambda_t g_t $. Last, suppose that $ t \mapsto (g_t, \lambda_t) $ is differentiable. Then,
\begin{align*}
\dot{\lambda}_t = -\lambda_t \frac{ g_t(t)^2 }{\norm{g_t}^2_{L^2(I_t)}}
\end{align*}
\end{lemma}
\end{shaded}

For the reader's convenience, we give a proof~:

\begin{proof}
Differentiating $  \Kb_{\! t} g_t = \lambda_t g_t $ yields
\begin{align*}
-K(\cdot, t) g_t(t) + \Kb_{\! t} \dot{g}_t = \dot{\lambda}_t g_t + \lambda_t \dot{g}_t
\end{align*}

In particular
\begin{align*}
K(\cdot, t) g_t(t) + \dot{\lambda}_t g_t  =  \Kb_{\! t} \dot{g}_t - \lambda_t \dot{g}_t 
                 =  (\Kb_{\! t} - \lambda_t I ) \dot{g}_t  
                 &  \in \im( \Kb_{\! t} - \lambda_t I  ) \\
                 & \perp \ker(\Kb_{\! t} - \lambda_t I ) = \Rr g_t \quad\mbox{as }\Kb_{\! t}^* = \Kb_{\! t}.
\end{align*}

Taking the scalar product with $ g_t $ and dropping the index $ L^2(I_t) $ then yields
\begin{align*}
0 & = g_t(t) \bracket{K(\cdot, t) , g_t} + \dot{\lambda}_t \bracket{g_t, g_t} \\
              & = g_t(t) \Kb_{\! t}g_t(t) + \dot{\lambda}_t \norm{g_t}^2 \\
              & = \lambda_t g_t(t)^2 + \dot{\lambda}_t \norm{g_t}^2
\end{align*}
hence the result.
\end{proof}

\medskip
\begin{remark}\label{Rk:SpecialFuchs}
$ $

$ \bullet $ The original approach of Fuchs \cite{Fuchs} amounts to the identity $ \lambda_t \int_s^{+\infty} g_t g_s = \lambda_s \int_t^{+\infty} g_t g_s $ (with $ \int_t^{+\infty} g_t^2 = 1 $) which gives the result after division by $ t - s $ and the limit $ t \to s $.

$ \bullet $ The two versions of the Fuchs lemma~\ref{Lemma:Fuchs:Fixed} and~\ref{Lemma:Fuchs:Varying} can be set on the same page if one differentiates the kernel $  K_{ t}(x, y) = \Unens{x \geq t} K(x, y) $ in the sense of distributions (with a slight abuse of notation)~:
\begin{align*}
\bracket{  \dot{\Kb}_{\! t} g_t , g_t}  & = -\bracket{ \delta_0(t - \cdot) \Kb  g_t , g_t}  \\
                & = -\lambda_t \bracket{ \delta_0(t - \cdot)    g_t , g_t}  \\
                & = -\lambda_t g_t(t)^2
\end{align*}

$ \bullet $ If one supposes that $ \Kb_{\! t} $ has a parametric kernel and acts on $ L^2([t, +\infty)) $, one easily finds the following modification of lemmas~\ref{Lemma:Fuchs:Fixed} and \ref{Lemma:Fuchs:Varying}~:
\begin{align}\label{Eq:FuchsLemmaGeneralEquation}
\dot{\lambda}_t = -\lambda_t \frac{ g_t(t)^2 }{\norm{g_t}^2_{L^2(I_t)}} + \frac{\bracket{\dot{\Kb}_{\! t} g_t, g_t }_{\!\! L^2(I_t)} }{\norm{g_t}^2_{L^2(I_t)}}
\end{align}

$ \bullet $ A direct corollary of lemma~\ref{Lemma:Fuchs:Varying} is that for all $ a \in \Rr $
\begin{align*}
\lambda_s = \lambda_a \exp\prth{ -\int_a^s \frac{g_t(t)^2}{\norm{g_t}^2_{L^2(I_t) }} dt }
\end{align*}

If in addition $ \Kb_{-\infty} = I $ and $ \Kb_{+\infty} = 0 $, one can take $ a = -\infty $ to get 
\begin{align}\label{Eq:Fredholm:VPvaryingInterval}
\lambda_s = \exp\prth{ -\int_{-\infty}^s \frac{g_t(t)^2}{\norm{g_t}^2_{L^2(I_t) }} dt }, 
\qquad
\int_\Rr \frac{g_t(t)^2}{\norm{g_t}^2_{L^2(I_t) }} dt = +\infty 
\end{align}
in which case $ \lambda_s $ defines a random variable $ Z $ such that $ \lambda_s = \Prob{Z > s} $.
\end{remark}


\begin{proof}[Proof of lemma~\ref{Lemma:TW2max:ter}]
If the operator $ \Kb_{\! \Ai} \equiv \Kb_{\! \Ai_0} $ defined in \eqref{Def:Kernel:AiWithInt} acts on $ L^2([s, +\infty)) $, then, by an obvious change of variable~:
\begin{align*}
\det\prth{I - \Kb_{\! \Ai_s} }_{L^2(\Rr_+)} = \det\prth{I - \Kb_{\! \Ai} }_{L^2([s, +\infty))}
\end{align*}

The varying interval Fuchs lemma~\ref{Lemma:Fuchs:Varying} and the Lidskii formula \eqref{Eq:Fredholm:Lidskii} can then be combined to give 
\begin{align*}
\Prob{\TW_2 \leq s} & = \det\prth{I - \Kb_{\! \Ai} }_{ L^2([s, +\infty)) } \\
                & = \prod_{k \geq 1} (1 - \widetilde{\lambda}_k(s)) \\
                & = \prod_{k \geq 1} \prth{ 1 - \exp\prth{ -\int_{-\infty}^s \widetilde{\Psi}_k^{(t)}\! (t)^2 dt }  }  \qquad\mbox{with \eqref{Eq:Fredholm:VPvaryingInterval}.}
\end{align*}

Here, $ \widetilde{\Psi}_k^{(t)} $ is the $k$-th eigenvector of $ \Kb_{\! \Ai} \ActsOn L^2([t, +\infty)) $ or equivalently of $ \widetilde{\Le}_{TW, t} := D(X - t)D - X(X - t) $. It is clear that 
\begin{align*}
\int_{  [t, +\infty) } K(x, y) f(y) dy = \lambda f(x) 
\quad\Longleftrightarrow\quad
\int_{  \Rr_+ } K(X + t, Y + t) f(Y + t) dY = \lambda f(X + t)
\end{align*}

As a result, $\widetilde{\Psi}_k^{(t)} = \Psi_k^{(t)}(\cdot - t) $ and $\widetilde{\Psi}_k^{(t)}(t) = \Psi_k^{(t)}(0) $. One sees moreover that $ \widetilde{\lambda}_k(s) = \lambda_k(s) $, which proves that we still have an expression of the law of $ Z_k'(\Ai) $. This concludes the proof.
\end{proof}


\begin{remark}\label{Rk:MaxWithIIDs}
Define
\begin{align*}
\Phi_k(t) := \int_{-\infty}^t \Psi_k^{(s)}(0)^2 ds
\end{align*}

Then, writing $ \Phi_k\inv $ for the inverse bijection of $\Phi_k$, one has
\begin{align*}
Z_k'(\Ai) \eqlaw \Phi_k\inv(\ee), \qquad \ee \sim \Exp(1)
\end{align*}
and one has expressed $ \TW_2 \eqlaw \max_{k \geq 1} \Phi_k\inv(\ee_k) $ with an i.i.d. sequence $ (\ee_k)_{k \geq 1} $, answering the original question of \cite[prob. (15)]{AIMproblems}.
\end{remark}

\medskip
\subsection{Extension to operators $ \Hb(\phi_s) \Mg_f^2 \Hb(\phi_s) $}\label{SubSec:TW:HMHoperators}

\subsubsection{First approach}

One classical generalisation of the previous result is the case where $ K_{\! s} $ is the kernel of an operator $ \Kb_{\! s} $ acting on $ L^2(\Rr_+) $ that reads
\begin{align}\label{Def:Kernel:HankelMultHankel}
K_{\! s}(x, y) = \int_\Rr \phi_s(x + u) \phi_s(y + u) f(u)^2 du
\end{align}
with $  \phi_s f \in L^2(\Rr) $ for all $s$ (say). We also suppose that $ \Kb_{\! s} $ has all its eigenvalues in $ (0, 1) $ for all $s$. 

Defining the multiplication operator 
\begin{align}\label{Def:Operators:Multiplier}
\Mg_f : g \mapsto fg
\end{align}
the underlying operator reads
\begin{align}\label{Def:Operators:HankelMultHankel}
\Kb_{\! s} = \Hb(\phi_s) \Mg_f^2 \Hb(\phi_s) = \Lb_s \Lb_s^*, \qquad \Lb_s := \Hb(\phi_s) \Mg_f
\end{align}

The relevant modification of \eqref{Eq:Fredholm:HankelSquareWithExp} is 
\begin{align}\label{Eq:Fredholm:MultHankelSquareWithExp}
\det\prth{ I - \Hb(\phi_s) \Mg_f^2 \Hb(\phi_s)}_{\! L^2(\Rr_+) }  = \prod_{\ell \geq 0} \exp\prth{ -\int_t^{+\infty} \bracket{  (\Lb_s \Lb_s^*)^\ell \phi_s  , \phi_s f^2}_{ L^2(\Rr_+) }   ds }
\end{align}

Usually, one does not have $ \Lb_s \Lb_s^* =  \Lb_s^* \Lb_s $ hence cannot write this last scalar product as $ \norm{   \Lb_s^\ell \phi_s  }_{ L^2(\Rr_+) }^2 $, but the operator is clearly positive and one can define the square root $ \Mb_s := (\Lb_s \Lb_s^*)^{1/2} $ in which case one gets $ \norm{   \Mb_s^\ell \phi_s  }_{ L^2(\Rr_+) }^2 $. Such a square root is nevertheless not easily expressible as an explicit operator such as $ \Lb_s $.

\medskip

One classical such case is given by the value in $ x = 0 $ of the solution to the Stochastic Heat Equation (SHE) with Dirac initial condition \cite{AmirCorwinQuastel, BertiniGiacomin, BorodinGorinKPZ, CalabreseLeDoussalRosso, DotsenkoKPZ, SasamotoSpohn1, SasamotoSpohn2, SasamotoSpohn3} also called continuum random polymer. This equation reads
\begin{align}\label{Eq:SHE}
\begin{cases}
\partial_t u = \frac{1}{2} \Delta u + \dot{\xi} u \\
u(t = 0, \cdot) = u_0 \equiv \delta_0
\end{cases}
\end{align}

If one defines $ u(x, t) = e^{ h(x, t)} $, then, $h$ is the Hopf-Cole solution of the Kardar-Parisi-Zhang (KPZ) equation \cite{BertiniGiacomin}.

A key formula gives a determinantal structure for $ h(0, t) $ with $ u_0 = \delta_0 $ \cite{AmirCorwinQuastel, BorodinGorinKPZ, CalabreseLeDoussalRosso, DotsenkoKPZ, SasamotoSpohn1, SasamotoSpohn2, SasamotoSpohn3}~:
\begin{align}\label{Eq:DetInSHE}
\Prob{  \frac{h(0, t) + \frac{t}{24} }{\gamma_t} + \frac{\Gumbel(1)}{\gamma_t} \leq s  } = \det\prth{ I - \Kb_{\!\KPZ(s, t)} }_{L^2(\Rr_+)}, \qquad \gamma_t := \prth{ \frac{t}{2} }^{1/3}
\end{align}
where $ \Gumbel(1) $ is a Gumbel-distributed random variable independent of $ h(0, t) $ (for all $t$) and with\footnote{
In \cite{AmirCorwinQuastel, SasamotoSpohn1, SasamotoSpohn2, SasamotoSpohn3}, the factor $  \frac{du}{ 1 + e^{-\gamma_t (u-s) }} $ is replaced by $ \frac{du}{ 1 - e^{-\gamma_t (u-s) } } $ with a principal value for the integral. The form we give is from \cite{BorodinGorinKPZ, CalabreseLeDoussalRosso, DotsenkoKPZ} and \cite[ch. 5.4.3]{BorodinCorwinMacDo}.
}
\begin{align}\label{Def:Kernel:KPZ}
K_{\! \KPZ(s, t)}(x, y) := \int_\Rr \Ai_s(x + u)\Ai_s(y + u) \frac{du}{ 1 + e^{-\gamma_t u } }
\end{align}
which is exactly of the type \eqref{Def:Kernel:HankelMultHankel}, the underlying operator being thus of the type \eqref{Def:Operators:HankelMultHankel}. 

It was moreover remarked by Johansson \cite[prop. 1.4]{JohanssonGumbelTW} that the RHS is 
\begin{align*}
\Prob{ \lambdab_1^{(t)} \leq s } 
\end{align*}
where $ (\lambdab^{(t)}_k)_{ k \geq 1} $ is the point process given by the pointwise randomisation and re-ordering
\begin{align*}
(\lambdab^{(t)}_k)_{ k \geq 1} :\eqlaw (\lambdab_k + R_k/\gamma_t)^>_{k \geq 1}, \qquad (\lambdab_k)_{k \geq 1} \sim \PPD(K_{\! \Ai})
\end{align*}
with $ (R_k)_{k \geq 1} $ an i.i.d. sequence with law $ R_1 $ given by the \textit{Fermi factor}
\begin{align}\label{EqLaw:FermiFactor}
\Prob{R_1 \leq x} = F(x) :=  \frac{1}{1 + e^{- x} }  \qquad \Longleftrightarrow\qquad R_1 \eqlaw \Gumbel(1) - \Gumbel'(1)
\end{align}
where $ \Gumbel(1) $ and $ \Gumbel'(1) $ are two independent Gumbel-distributed random variables. As a result, 
\begin{align}\label{EqLaw:KPZwithMaxNonIndep}
\lambdab^{(t)}_1 \eqlaw \max_{k \geq 1}\ensemble{ \lambdab_k + R_k/\gamma_t}
\end{align}

This last equality in law is equivalent, after integration on the $ R_k $s, to the tensorial identity of Fredholm determinants obtained in \cite{AmirCorwinQuastel, BorodinGorinKPZ, SasamotoSpohn1, SasamotoSpohn2, SasamotoSpohn3} building on \cite{BertiniGiacomin, TWasep1, TWasep2, TWasep3}.

From this, one gets the equality in law
\begin{align*}
\frac{h(0, t) + \frac{t}{24} }{\gamma_t} + \frac{\Gumbel(1)}{\gamma_t} \eqlaw \lambdab_1^{(t)} \eqlaw \max_{k \geq 1}\ensemble{ \lambdab_k + R_k/\gamma_t}
\end{align*}
and using a slight modification of the proof of theorem~\ref{Thm:TW2max}, one gets the existence of independent random variables $ Z_k^{(t)}(\Ai) $ such that
\begin{align}\label{EqMax:KPZ}
\boxed{\lambdab_1^{(t)} \eqlaw \max_{k \geq 0}\ensemble{ Z_k^{(t)}(\Ai) }}
\end{align}
with
\begin{align*}
\Prob{Z_k^{(t)}(\Ai) \leq s} = \exp\prth{- \!\! \int_s^{+\infty} \!\! \bracket{  (\Hb(\Ai_u)\Mg_{f_t}\Hb(\Ai_u))^k \Ai_u  , \Ai_u f_t}_{ \! L^2(\Rr_+) } \! du \! } , \quad f_t(x) = F(\gamma_t x)
\end{align*}

\medskip

The modification of the proof of lemma~\ref{Lemma:ExistenceZkAi} goes as follows~: write
\begin{align*}
K_{\KPZ(s, t)}(x, y) & = \int_{\Rr } \Ai(x +u)\Ai(y + u ) \Prob{ R \leq \gamma_t (u - s) } du \\
              & = \Esp{ \int_{\Rr } \Ai(x +u)\Ai(y + u )  \Unens{ R/\gamma_t + s \leq u  }  du } \\
              & = \Esp{ \int_{s + \frac{R}{\gamma_t} }^{+\infty} \Ai(x +u)\Ai(y + u )   du } \\
              & = \Esp{ K_{\Ai, s + \frac{R}{\gamma_t}}(x, y) }
\end{align*}

As a result, 
\begin{align}\label{Eq:Kernel:KPZasRandomisedAiry}
\Kb_{\KPZ(s, t)}  = \Esp{ \Kb_{\Ai_{ s + R/\gamma_t } }} = \Esp{ \Hb(\Ai_{ s + R/\gamma_t })^2 }
\end{align}
and, setting $ R_t := \gamma_t\inv R_1 $ and $ \Kb_{\! s} \equiv \Kb_{\!\KPZ(s, t)}  $,
\begin{align*}
\det(I - \Kb_{\! s}) = \prod_{\ell \geq 0} \exp\prth{-\Esp{ \int_s^{+\infty}\bracket{ \Hb \prth{ \Ai_{ u + R_t^{(1)} } }^{\! 2} \cdots \Hb\prth{ \Ai_{ u + R_t^{(\ell)} } }^{\! 2} \, \Ai_{u + R_t} ,  \Ai_{u + R_t} }_{\!\! L^2(\Rr_+) }  du } }
\end{align*}

One can thus mimic the computations of lemma~\ref{Lemma:ExistenceZkAi} a.s. for the random times $ s + R_t^{(i)} $ and $ s \ll 0 $, which amounts to do it for $ s \ll 0 $, giving the same randomised estimates of $ \Omega(\abs{s}^m) $ with $s$ replaced by $ s + \min_\ell R_t^{(\ell)}  $, and finally taking the expectation at the end, which is allowed as $ R_t \eqlaw R/\gamma_t $ and $ R $ has moments at every order. Details are left to the reader.

\medskip
\subsubsection{Second approach}

Using the fact that
\begin{align*}
\det\prth{ I -  \Ab\Bd  }_{L^2(I)}  = \det\prth{ I - \Bd \Ab  }_{L^2(I)} 
\end{align*}
for all trace-class operators $ \Ab, \Bd \ActsOn L^2(I) $, one gets
\begin{align*}
\Prob{\lambdab_1^{\KPZ(t)} \leq s} & = \det\prth{ I - \Kb_{\KPZ(s, t)} }_{L^2(\Rr_+)} \\
                   & = \det\prth{ I - \Hb(\Ai_s) \Mg_{f_t^2} \Hb(\Ai_s) }_{L^2(\Rr_+)}, \qquad f_t(x) := \frac{1}{\sqrt{ 1 + e^{-\gamma_t x} }}, \ \gamma_t := \prth{\tfrac{t}{2}}^{\!\frac{1}{3}} \\
                   & = \det\prth{ I -  \Mg_{f_t} \Hb(\Ai_s)^2 \Mg_{f_t}  }_{L^2(\Rr_+)}  
\end{align*}

As
\begin{align*}
\frac{d}{ds}\Mg_{f_t} \Hb(\Ai_s)^2 \Mg_{f_t} & = \Mg_{f_t} \frac{d}{ds} \Hb(\Ai_s)^2 \Mg_{f_t}  = \Mg_{f_t} (\Ai_s \otimes \Ai_s^* ) \Mg_{f_t} \\
              & = (f_t \Ai_s) \otimes (f_t \Ai_s)^*
\end{align*}
the application of the Fuchs lemma~\ref{Lemma:Fuchs:Fixed} to $\Kb_{\!\KPZ(s, t)}$ gives 
\begin{align*}
\Prob{\lambdab_1^{\KPZ(t)} \leq s} & = \prod_{k \geq 1} \int_{-\infty}^s \bracket{ f_t \Ai_u, \Psi_k^{(u)} }^{\!2}_{\!\!L^2(\Rr_+)} du  
\end{align*}

We have thus proven the~:

\begin{shaded}
\begin{theorem}[Max-independence structure in the $\KPZ$ equation]\label{Thm:KPZmax}
Let $ (Z_k^{\KPZ(t)})_{k \geq 1} $ be an independent sequence of random variables with law given by the density
\begin{align}\label{Def:Law:PsiSquare:KPZ}
f_{ Z_k^{\KPZ(t)} }(s) := \bracket{ f_t \Ai_s, \Psi_k^{(s)}}_{\!\! L^2(\Rr_+)}^{\! 2} 
\end{align}

Then,
\begin{align}\label{EqMax:KPZ:bis}
\boxed{ \lambdab_1^{\KPZ(t)} \eqlaw \max_{k \geq 1} Z_k^{\KPZ(t)} }
\end{align}
\end{theorem}
\end{shaded}

\begin{remark}
Translating the $\KPZ$ kernel gives 
\begin{align*}
\det\prth{I - \Kb_{\KPZ(s, t)}}_{L^2(\Rr_+)} = \det\prth{I - \widetilde{\Kb}_{\KPZ(s, t)}}_{L^2([s, +\infty))}
\end{align*}
where
\begin{align*}
\widetilde{K}_{\KPZ(s, t)}(x, y) = \int_{\Rr_+} \Ai(x + u)\Ai(y + u) f_t(u - s)^2 du
\end{align*}
which would be the kernel of $ \Hb(\Ai)\Mg_{f_t(\cdot - s)^2}\Hb(\Ai) $ if it were acting on $ L^2(\Rr_+) $ (but since we are acting on $ L^2([s, +\infty)) $, this is not a correct interpretation). 

The general Fuchs lemma with a parametric kernel $ K_s $ on a varying interval $ L^2([s, +\infty)) $ given in \eqref{Eq:FuchsLemmaGeneralEquation} gives then for any $ \Psi^{(s)} \in \ensemble{\Psi_k^{(s)}}_{k \geq 1} $
\begin{align*}
\dot{\lambda}(s) & =  - \lambda(s) \Psi^{(s)}(0)^2 + \int_{[s, +\infty)^3}   \Ai(x + u)\Ai(y + u) 2 (f_t f_t')(u - s) du  \widetilde{\Psi}^{(s)}(x)  \widetilde{\Psi}^{(s)}(y) dx dy \\
                 & = - \lambda(s) \Psi^{(s)}(0)^2 + \int_{[s, +\infty) }  \bracket{ \Ai_{s + u},\Psi^{(s)}  }_{L^2(\Rr_+)}^{\! 2} 2 (f_t f_t')(u - s) du  
\end{align*}

This first order affine equation $ \dot{\lambda} + a \lambda = b $ can be easily integrated to give another expression of $ \lambda $, of the form $ \lambda_s = \int_s^{+\infty} (e^{\int_s^v a}) b(v) dv $, but it seems less interesting than~\eqref{Def:Law:PsiSquare:KPZ}.
\end{remark}


\begin{remark}
Tsai \cite[(1.3)]{TsaiKPZ} computes the large deviation functional for $ h(0, t) - \frac{t}{12} $ (with initial condition $ \delta_0 $). With the max-independence structure, this becomes a problem of large deviations of a max of independent random variables, and the Poisson approximation approach described in \S~\ref{SubSec:Intro:GUE} applies. Nevertheless, the random variables are not so explicit and a study of the prolate hyperspheroidal functions is necessary beforehand.
\end{remark}

\medskip 
\subsection{Extension to $ \TW_1 $}\label{SubSec:TW:GOE}

Recall that the power $ \alpha > 0 $ of a cumulative distribution function $ F_X : x \mapsto \Prob{X \leq x} $ defines a random variable $ X^{(\alpha)} $. If $ \alpha \in \Nn^* $, one has $ X^{(\alpha)} \eqlaw \max_{1 \leq k \leq \alpha} X_k $ where $ (X_k)_k $ is a sequence of i.i.d. copies of $X$. In the general case, it seems that such a construction is not available.

The case of the $ GO\!E $ Tracy-Widom distribution can then be treated analogously to the $ GU\!E $ one~:

\begin{shaded}
\begin{theorem}[Max-independence structure in $ \TW_1 $]\label{Thm:TW1max}
One has with independent random variables in the RHS
\begin{align}\label{EqMax:TW1}
\boxed{\TW_1 \eqlaw \max\!\ensemble{ Q^{(1/2)}, \max_{k \geq 0} Z_k(\Ai)^{(1/2)} } }
\end{align}
where $ X^{(1/2)} $ is the random variable built out of $X$ via $ \Prob{X \leq x}^{1/2} = \Prob{ X^{(1/2) } \leq x} $, $ (Z_k(\Ai))_k $ are defined in theorem~\ref{Thm:TW2max} and $ \Prob{Q \leq s} = e^{ - \int_s^{+\infty} q(x) dx } $ with $q$ the Hastings-McLeod solution of the Painlev\'e II equation defined in \eqref{Def:PainlevéII}.
\end{theorem}
\end{shaded}

\begin{proof}
Using the result of \cite[(53)]{TracyWidomGOEgse}, one has
\begin{align}\label{EqTW:TW1}
\Prob{\TW_1 \leq s} = \sqrt{ \Prob{ \TW_2 \leq s } e^{- \int_s^{+\infty} q(x) dx } } 
\end{align}

Using $ \Prob{X \leq x}^{1/2} := \Prob{ X^{(1/2) } \leq x} $, one thus has 
\begin{align*}
\Prob{\TW_2 \leq s }^{1/2} = \prod_{k \geq 1} \Prob{Z_k(\Ai) \leq s}^{1/2} =: \prod_{k \geq 1} \Prob{Z_k^{(1/2)}(\Ai) \leq s} = \Prob{ \max_{k \geq 1} Z_k(\Ai)^{(1/2)} \leq s }
\end{align*}
and it remains to show that there exists $ Q $ such that $ \Prob{Q \leq s} = e^{ -  \int_s^{+\infty} q(x) dx } $. It is clear that $ F_q(s) := e^{ -   \int_s^{+\infty} q(x) dx } \to 1 $ when $ s \to +\infty $. Proving that $ F_q(s) \to 0 $ when $ s \to -\infty $ is equivalent to $ \int_\Rr q = +\infty $, and showing that $ F_q $ is increasing is equivalent to $ q \geq 0 $ on $ \Rr $ by differentiation. Such properties are granted for $q$, see e.g. \cite{HastingsMcLeod} (the fact that $ \int_\Rr q = +\infty $ is a consequence of the asymptotics $ q(t) \sim \abs{t/2}^{1/2} $ when $ t \to -\infty $). 
\end{proof}


\begin{remark}
One might be tempted to use the Ferrari-Spohn characterisation \cite[Prop. 1/(16)]{FerrariSpohnGOE} that reads  
\begin{align}\label{Eq:TW1:FerrariSpohn}
\Prob{\TW_1 \leq s} = \det(I - \Hb(\Ai_s))_{L^2(\Rr_+)}  
\end{align}
but $ \Hb(\Ai_s) $ is not a positive operator, and a structure similar to the one used for the $ \gue $ is not obvious to obtain.

Nevertheless, \eqref{Eq:TW1:FerrariSpohn} allows to give another decomposition of the random variable $ Q $ defined in the proof of theorem~\ref{Thm:TW1max} using the Fuchs lemma~\ref{Lemma:Fuchs:Fixed}~:
\begin{align*}
\Prob{Q \leq s} & =  e^{- \int_s^{+\infty} q(x) dx } = \frac{\Prob{\TW_1 \leq s}^2 }{ \Prob{ \TW_2 \leq s } } = \frac{\det(I - \Hb(\Ai_s))_{L^2(\Rr_+)}^2}{ \det(I - \Hb(\Ai_s)^2)_{L^2(\Rr_+)} } \\
               & = \frac{\det(I - \Hb(\Ai_s))_{L^2(\Rr_+)} }{ \det(I + \Hb(\Ai_s) )_{L^2(\Rr_+)} } \\
               & = \det\prth{ I - 2 \Hb(\Ai_s) (I + \Hb(\Ai_s))\inv }_{L^2(\Rr_+)} 
\end{align*}

Define
\begin{align*}
\Lb_s := 2 \Hb(\Ai_s) (I + \Hb(\Ai_s))\inv \ActsOn L^2(\Rr_+)
\end{align*}

If $ \mu_k(s) $ denotes the $k$-th eigenvalue of $ \Lb_s $ for the normalised eigenvector $ \psi_{k, s} $, then
\begin{align*}
\mu_k(s)
               & = 2 \int_s^{+\infty } \bracket{ \frac{d}{ds}\prth{ \Hb(\Ai_t) (I + \Hb(\Ai_t))\inv } \psi_{k, t}, \psi_{k, t} }_{\!\! L^2(\Rr_+)} dt  \\ 
               & = 2 \int_s^{+\infty } \bracket{    (I + \Hb(\Ai_t))^{-2}   \Ai_t\otimes \Ai_t^* \psi_{k, t}, \psi_{k, t} }_{L^2(\Rr_+)} dt  \\
               & = 2 \int_s^{+\infty } \bracket{    (I + \Hb(\Ai_t))^{-2}   \Ai_t , \psi_{k, t} }_{L^2(\Rr_+)}  \bracket{  \Ai_t , \psi_{k, t} }_{L^2(\Rr_+)} dt  \\
               & = 2 \int_s^{+\infty } \bracket{    (I + \Hb(\Ai_t))\inv   \Ai_t , \psi_{k, t} }_{L^2(\Rr_+)}^2 dt 
\end{align*}

One has moreover $ \mu_k(-\infty) = 1 $ as $ \Lb_{-\infty} = I $ and $ \mu_k(+\infty) = 0 $ as $ \Lb_{+\infty} = 0 $. We thus deduce the existence of a random variable $ Q_k $ such that 
\begin{align*}
\Prob{ Q_k > s } = 2 \int_s^{+\infty } \bracket{    (I + \Hb(\Ai_t))\inv   \Ai_t , \psi_{k, t} }_{L^2(\Rr_+)}^2 dt 
\end{align*}
and, with independent random variables,
\begin{align}\label{EqMax:Q}
Q \eqlaw \max_{k \geq 1} Q_k 
\end{align}

As a result, 
\begin{align*}
\TW_1 \eqlaw \max_{k \geq 1}\ensemble{\max\ensemble{ Q_k^{(1/2)}, Z_k(\Ai)^{(1/2)} }}
\end{align*}

\end{remark}


\begin{remark}
The case of the $ GS\!E $ Tracy-Widom distribution is yet to be obtained. The description similar to \eqref{EqTW:TW1} is  \cite[(54)]{TracyWidomGOEgse}
\begin{align*}
\Prob{\TW_4 \leq s/\sqrt{2} } = \sqrt{\Prob{\TW_2 \leq s}} \cosh(U(t))^2, \qquad U(t) := \frac{1}{2} \int_t^{+\infty} q  
\end{align*}
but the $ \cosh $ term does not seem to be expressible as a probability. Other expressions are available, such as the Ferrari-Spohn one \cite[(35)]{FerrariSpohnGOE} that reads
\begin{align*}
\Prob{\TW_4 \leq s/\sqrt{2} } =  \frac{1}{2}\prth{ \det\prth{I - \Hb(\Ai_s)}_{L^2(\Rr_+)}  + \det\prth{I + \Hb(\Ai_s)}_{L^2(\Rr_+)}  }
\end{align*}
but the same problem remains. 
\end{remark}

$ $

\medskip
\section{Max-independence structure in random matrix ensembles}\label{Sec:RMT}

\subsection{The $ \gue $}\label{Subsec:RMT:GUE}

The Gaussian Unitary Ensemble of size $N$, in short $ \gue_N $, is the probability space $ (\He_N(\Cc), \Pp_N) $ where $ \Pp_N $ is the Gaussian measure given by 
\begin{align}\label{Def:Law:GUE}
\Pp_N(dM) = e^{ - \trace(M M^*) / 2 } \frac{dM}{(2\pi)^{N^2/4}}
\end{align}

A disintegration of $ \Mb \sim \gue_N $ according to the map $ M \mapsto U^* \Lambda U $ where $ U \in \Ue_N(\Cc) $ and $ \Lambda $ is a diagonal real matrix gives the law of the eigenvalues $ \lambdab = (\lambdab_{1, N}, \dots, \lambdab_{N, N}) $ 
\begin{align*}
\Proba{N}{ \lambdab \in d\xb } = \frac{1}{N!\, \Ze_N } \Delta(\xb)^2 e^{- \abs{\xb}^2/2 } d\xb, \quad \Delta(\xb) := \prod_{1 \leq i < j \leq N }(x_i - x_j), \quad \Ze_N := (2\pi)^{N/2} \prod_{k = 0}^{N - 1} k!
\end{align*}

In particular, supposing that $ \lambdab_{1, N} \geq   \cdots \geq \lambdab_{N, N} $, one has
\begin{align*}
\Prob{ \lambdab_{1, N} \leq s } = \frac{  \Ze_N(s) }{ \Ze_N(\infty) }, \qquad  \Ze_N(s) := \frac{1}{N!} \int_{ (-\infty, s)^N } \Delta(\xb)^2 e^{- \abs{\xb}^2/2 } d\xb
\end{align*}

We now recall, for the reader's convenience, some classical facts about this ensemble that one can find e.g. in \cite{AndersonGuionnetZeitouni, BenHougKrishnapurPeresVirag, ForresterBook, MethaBook, TaoRMT}. 

\medskip

The Andr\'eieff-Heine-Szeg\"o formula \cite[p. 24]{Szego} reads for $ f_i, g_i \in L^2(\Rr, \mu) $
\begin{align}\label{Eq:AndreieffHeineSzego}
\frac{1}{N!} \! \int_{\Rr^N} \! \det\prth{f_{i }(t_j) }_{1 \leq i, j \leq N} \det\prth{g_{i }(t_j) }_{1 \leq i, j \leq N} d\mu^{\otimes N}(\tb) = \det\prth{ \int_\Xg f_{i } g_{j } d\mu  }_{1 \leq i, j \leq N}
\end{align}

This is the continuous version\footnote{Using a discrete measure $ \mu $, one gets \eqref{Eq:CauchyBinet} from \eqref{Eq:AndreieffHeineSzego}~; but one can pass to the limit with a Riemann sum starting from \eqref{Eq:CauchyBinet} to obtain \eqref{Eq:AndreieffHeineSzego} which makes these two formuli equivalent.} of the Cauchy-Binet formula
\begin{align}\label{Eq:CauchyBinet}
\det\prth{ \sum_{\ell \geq 1} A_{i, \ell} B_{\ell, j} }_{1 \leq i, j \leq N} & = \sum_{1 \leq \ell_1 < \ell_2 < \cdots < \ell_N } \det\prth{ A_{i, \ell_j} }_{1 \leq i, j \leq N} \det\prth{ B_{ \ell_i, j} }_{1 \leq i, j \leq N} \notag \\
               & = \frac{1}{N! } \sum_{1 \leq \ell_1 \neq \ell_2 \neq \cdots \neq \ell_N } \det\prth{ A_{i, \ell_j} }_{1 \leq i, j \leq N} \det\prth{ B_{ \ell_i, j} }_{1 \leq i, j \leq N} \notag \\
               & = \frac{1}{N! } \sum_{ \ell_1, \dots , \ell_N \geq 0 } \det\prth{ A_{i, \ell_j} }_{1 \leq i, j \leq N} \det\prth{ B_{ \ell_i, j} }_{1 \leq i, j \leq N} 
\end{align}

In the case of $ d\mu(x) := e^{-x^2/2} \frac{dx}{\sqrt{2\pi}} $, one has
\begin{align*}
\Ze_N(s) & := \frac{1}{N!} \int_{ (-\infty, s)^N } \Delta(\xb)^2 e^{- \abs{\xb}^2/2 } \frac{d\xb}{(2\pi)^{N/2} } \\
             & =   \int_{ (-\infty, s)^N } \det\prth{P_{i - 1}(x_j)}_{1 \leq i, j \leq N} \det\prth{Q_{i - 1}(t_j) }_{1 \leq i, j \leq n} e^{- \abs{\xb}^2/2 } \frac{d\xb}{(2\pi)^{N/2} }
\end{align*}
for \textit{any} family of monic polynomials $ (P_k, Q_k)_{k \geq 0} $ satisfying $ \deg(P_k) = \deg(Q_k) = k $ (taking linear combinations in the Vandermonde determinant). The choice of the truncated Hermite polynomials $ P_k = Q_k = H_k(\cdot \vert s) $, i.e. the monic orthogonal polynomials for $ L^2( (-\infty, s), \mu) $ gives a diagonal matrix, and the Andr\'eieff-Heine-Szeg\"o formula \eqref{Eq:AndreieffHeineSzego} then implies
\begin{align*}
\Ze_N(s) =  \prod_{k = 0}^{N - 1} \int_{-\infty}^s H_k(x\vert s)^2 e^{-x^2/2} \frac{dx}{\sqrt{2\pi} }
\end{align*}

The formula 
\begin{align*}
\Prob{\lambdab_{1, N} \leq s} = \prod_{k = 0}^{N - 1} \frac{\norm{H_k(\cdot \vert s) }_{L^2( (-\infty, s), \mu) }^2}{\norm{H_k(\cdot \vert \infty) }_{L^2( (-\infty, \infty), \mu) }^2}
\end{align*}
is \textit{not} of the form
\begin{align*}
\Prob{\lambdab_{1, N} \leq s} = \prod_{k = 0}^{N - 1} \int_{-\infty}^s f_{W_k}(x) dx = \Prob{ \max_{0 \leq k \leq N-1} W_k \leq s} 
\end{align*}
since the polynomials depend on $s$. The following theorem allows to circumvent this fact \`a la Ginibre:

\begin{shaded}
\begin{theorem}[Max-independence structure in the $ \gue $]\label{Thm:GUEmax}
Let $ H_k(x\vert t) $ be the orthogonal polynomial for $ L^2( (-\infty, t), e^{-x^2/2} \frac{dx}{\sqrt{2\pi}}) $ (truncated Hermite polynomial). Define the law of a random variable $ W_k $ by
\begin{align}\label{Def:Law:PsiSquare:GUE}
\Prob{ W_k \leq s } = \int_{-\infty}^s \He_k(x)^2 e^{-x^2/2} \frac{ dx }{ \sqrt{2\pi} \, k! }, \qquad \He_k(x) := H_k( x \vert \blue{x} )
\end{align}
and let $ (W_k)_k $ be a sequence of independent such random variables. Then,
\begin{align}\label{EqMax:GUE}
\boxed{ \lambdab_{1, N} \eqlaw \max_{0 \leq k \leq N - 1} W_k }
\end{align}
\end{theorem}
\end{shaded}

\begin{proof}
A differentiation of $ \int_{-\infty}^s H_k( x \vert s)^2 d\mu(x) $ with $ d\mu(x) := e^{-x^2/2} \frac{dx}{\sqrt{2\pi}} $ gives 
\begin{align*}
\frac{d}{ds} \int_{-\infty}^s H_k( x \vert s)^2 e^{-x^2/2} \frac{dx}{\sqrt{2\pi} } & = H_k( s \vert s)^2 \frac{e^{-s^2/2}}{\sqrt{2\pi} } + \int_{-\infty}^s \frac{\partial}{\partial s} H_k( x \vert s)^2 e^{-x^2/2} \frac{dx}{\sqrt{2\pi} } \\
                 & = H_k( s \vert s)^2 \frac{e^{-s^2/2}}{\sqrt{2\pi} } + 2 \int_{-\infty}^s H_k( x \vert s) \frac{\partial}{\partial s} H_k( x \vert s) e^{-x^2/2} \frac{dx}{\sqrt{2\pi} }
\end{align*}

Since $ H_k(X\vert s) = X^k + \sum_{j = 0}^{k - 1} a_{j, k}(s) X^j $ is a monic orthogonal polynomial, $ \frac{\partial}{\partial s} H_k( X \vert s) = \sum_{j = 0}^{k - 1} \frac{\partial}{\partial s} a_{j, k}(s) X^j $ is of degree $ k - 1 $, hence, by orthogonality $ \bracket{H_k, \frac{\partial}{\partial s} H_k }_{L^2( (-\infty, s], \mu )} = 0 $.

Integrating from $ A $ to $s$ this equality gives
\begin{align*}
\int_{-\infty}^s H_k( x \vert s)^2 d\mu(x) - \int_{-\infty}^A H_k( x \vert A)^2 d\mu(x) = \int_{A}^s H_k( x \vert x)^2 e^{-x^2/2} \frac{dx}{\sqrt{2\pi}}
\end{align*}

Now, one can make $ A \to -\infty $ on the LHS which is finite, and $ \int_{-\infty}^A H_k( x \vert A)^2 d\mu(x) \to 0 $ since $ H_k(X\vert A) $ is a polynomial for all $ A $ and $ H_k(X\vert -\infty) = 0 $, giving the desired result.

Note that $ s \to +\infty $ in this equality gives the renormalisation constant as $ \norm{ H_k(\cdot \vert \infty) }^2_{L^2(\Rr, \mu) } = k! $ for the Gaussian measure. This concludes the proof. 
\end{proof}


\begin{remark}\label{Rk:GUEmax}
%
%
%
$ $
\begin{itemize}
\item The independence structure \eqref{EqLaw:Ginibre} in the Ginibre ensemble can be directly proven starting from the Andr\'eieff-Heine-Szeg\"o formula. Nevertheless, the differentiation step is not necessary as the orthogonal polynomials are directly the power functions on the disk. The proof extends to any complex ensemble with a circular symmetry (see e.g. \cite[ch. 4.3]{BenHougKrishnapurPeresVirag}).

\medskip
\item The argument of differentiation that we use to derive \eqref{Def:Law:PsiSquare:GUE} was also employed by Majumdar-Nadal \cite[\S~3.2]{MajumdarNadal} in the core of their rederivation of the Tracy-Widom distribution (with a method involving norms of orthogonal polynomials originally due to Gross-Matytsin \cite{GrossMatytsin}), but with no probabilistic considerations. 

\medskip
\item The differentiation argument is also the one used in the Fuchs lemma~\ref{Lemma:Fuchs:Fixed}/\ref{Lemma:Fuchs:Varying}, and the norms of orthogonal polynomials are the eigenvalues of the projection operator with the usual Christoffel-Darboux kernel that reads $ K_N(x, y) = \sum_{k = 0}^N H_k(x \vert t)   H_k(y \vert t) $ in $ L^2( (-\infty, t), e^{-x^2/2} \frac{dx}{\sqrt{2\pi}}) $. Normalising each term of the sum by their $ L^2 $-norm gives a sum of weighted projectors, the weight being the eigenvalue, i.e. $ \norm{H_k(\cdot\vert t)}^2_{L^2((-\infty, t], \mu)} $. 
Equivalently, one can use the truncated Hermite polynomials in $ L^2( [t, +\infty), e^{-x^2/2} \frac{dx}{\sqrt{2\pi}}) $ and the eigenvalue will be the norm of these polynomials, $ \widetilde{H}_k(x\vert t) $ (say) on this set. Of course, the formuli are equivalent~: $ k! - \vert\!\vert \widetilde{H}_k(\cdot, t)\vert\!\vert_{[t, +\infty)}^2  = \norm{H_k(\cdot \vert t)}^2_{(-\infty, t]}  $.

This is the exact analogue of the Fuchs lemma that one applies (after an application of the Lidskii formula). The expression of the law uses then the function $ \widetilde{H}_k(x\vert x)^2 $ in the same way \eqref{Def:Law:PsiSquare:TW2bis} uses $ \Psi_k^{(s)}(0)^2 = \widetilde{\Psi}_k^{(s)}(s)^2 \equiv \widetilde{\Psi}_k(s\vert s)^2 $.
\end{itemize}
\end{remark}

\vspace{0.1cm}
\subsection{Generalisation to any matrix ensemble on the real line and the circle}\label{Subsec:RMT:General}

The proofs of the theorems here presented are identical to the previous one and are thus omitted.

\subsubsection{The real line} 

\begin{definition}[Determinantal measure]\label{Definition:DeterminantalMeasure}
Let $ \mu $ be a probability measure on a space $ \Xg $ and let $ (\varphib, \psib) := (\varphi_k, \psi_k)_{k \geq 1} $ be a family of square integrable functions of $ L^2(\Xg, \mu) $. A determinantal measure with parameters $ (\varphib, \psib, \mu) $ is the law of random points $ (\lambdab_{k, N})_{1 \leq k \leq N} $ in $ \Xg $ with distribution (see e.g. \cite{LyonsDPP})
\begin{align}\label{Def:DeterminantalMeasure}
\hspace{-0.25cm}\Pp_{\! N}( \lambdab \in d\xb ) \! = \! \frac{1}{N!\, \det(\int_\Xg \varphi_i \overline{\psi_j} d\mu)_{1 \leq i, j \leq N} } \det\prth{ \varphi_i(x_j) }_{\! 1 \leq i, j \leq N} \overline{\det\prth{ \psi_i(x_j) }_{\! 1 \leq i, j \leq N} } d\mu^{\otimes N}(\xb) 
\end{align}

We note $ \boxed{\lambdab \sim \DetMeas(\varphib, \psib, \mu \vert N)} $ for such a random process. Note that the equality $ \int_{\Xg^N} \Proba{N}{ \lambdab \in d\xb } = 1 $ is the Andr\'eieff-Heine-Szeg\"o formula \eqref{Eq:AndreieffHeineSzego}.
\end{definition}

Theorem \ref{Thm:GUEmax} can be easily modified in the case of a general determinantal measure to give:

\begin{shaded}
\begin{theorem}[Independence structure in matrix ensembles on the real line]\label{Thm:GeneralEmax}
Consider random points $ (\lambdab_{k, N})_{1 \leq k \leq N} $ in $ \Rr $ with distribution 
\begin{align*}
\Proba{N}{ \lambdab \in d\xb } = \frac{1}{N!\, \Ze_N } \Delta(\xb)^2 d\mu^{\otimes N}(\xb), \qquad d\mu(x) := \dot{\mu}(x) dx
\end{align*}
i.e. $ \lambdab \sim \DetMeas(\psib, \psib, \mu \vert N) $ with $ \psi_k(x) := x^{k - 1} $.

Then, one has
\begin{align}\label{EqLaw:LineEnsemble}
\boxed{ \lambdab_{1, N} \eqlaw \max_{1 \leq k \leq N} Y_k }
\end{align}
where $ (Y_k)_k $ is a sequence of independent random variables with distribution given by
\begin{align}\label{Def:Law:PsiSquare:GUEGeneral}
\Prob{ Y_{k + 1} \leq s } = \int_{-\infty}^s P_k( x \vert \color{blue} x \color{black} )^2 \frac{d\mu(x)}{\int_\Rr P_k( y \vert \color{blue} \infty \color{black} )^2 d\mu(y)}
\end{align}
where $ P_k( \cdot \vert t) $ is the orthogonal polynomial for $ L^2( (-\infty, t), \mu) $ . 
\end{theorem}
\end{shaded}

The proof is a straightforward generalisation of the previous one and is thus omitted.

\medskip 
\subsubsection{The circle}

\begin{theorem}[Independence structure in matrix ensembles on the circle]\label{Thm:CircleEmax}
\begin{shaded}
Consider random points $ (e^{i\Thetab_{k, N}})_{1 \leq k \leq N} $ in $ \Uu $ with distribution 
\begin{align*}
\Proba{N}{ \Thetab \in d\thetab } = \frac{1}{N!\, \Ze_N } \abs{\Delta(e^{i\thetab})}^2 d\mu^{\otimes N}(\thetab), \qquad d\mu(\theta) := \dot{\mu}(e^{i\theta}) d\theta, \qquad \theta_k \in [0, 2\pi)
\end{align*}

Then, one has
\begin{align}\label{EqLaw:CircleEnsemble}
\boxed{ \Thetab_{1, N} \eqlaw \max_{1 \leq k \leq N} A_k }
\end{align}
where $ (A_k)_k $ is a sequence of independent random variables with distribution given by
\begin{align}\label{Def:Law:PsiSquare:GUECercle}
\Prob{ A_{k + 1} \leq s } = \int_0^s \abs{P_k( e^{i\theta} \vert \color{blue} \theta \color{black} )}^2 \frac{d\mu(\theta)}{\int_0^{2\pi} \abs{P_k( e^{i\alpha} \vert \color{blue} 2\pi \color{black} )}^2 d\mu(\alpha)}
\end{align}
where $ \theta \mapsto P_k( e^{i \theta} \vert t) $ is the monic OPUC for $ L^2( [0, t), \mu) $ . 
\end{shaded}
\end{theorem}


\begin{remark}\label{Rk:WeylCUEToeplitz}
The Weyl integration formula reads
\begin{align}\label{Eq:WeylIntegration}
\int_{\Ue_N} \det(f(U)) dU = \frac{1}{N!}\oint_{\Uu^N} f^{\otimes N}(\ub) \abs{\Delta(\ub)}^2 \frac{d^*\ub}{\ub}
\end{align}
where $ dU $ is the normalised Haar measure on the unitary group $ \Ue_N $ and $ \det(f) \in L^1(\Ue_N, dU) $.

This formula defines the \textit{Circular Unitary Ensemble} or $ CUE_N $ \cite{DeiftItsKrasovskyImpetus, DiaconisRMTSurvey, ForresterBook} that has been studied extensively, notably using the Andr\'eieff-Heine-Szeg\"o formula \eqref{Eq:AndreieffHeineSzego} that links it to a Toeplitz determinant of symbol $f$ \cite[\textit{fact five}]{DiaconisRMTSurvey}
\begin{align}\label{Eq:CUE=Toeplitz}
\int_{\Ue_N} \det(f(U)) dU = \det(\Tt_N(f)), \qquad \Tt_N(f) := \prth{ \, \crochet{z^{i - j}} \! f(z) }_{1 \leq i, j \leq N}
\end{align}

The previous theorem with $ \dot{\mu} = 1 $ treats thus the case of the $ CU\!E_N $. 
\end{remark}

\medskip

In view of the last results, an interesting question emerges:

\begin{question}\label{Q:CouplingGUE:max}
Can one find an almost sure construction that proves \eqref{EqMax:GUE} or \eqref{EqLaw:LineEnsemble} ?
\end{question}

In fact, question~\ref{Q:CouplingGUE:max} is already relevant at the level of the Ginibre ensemble, as the identity in law $ \lambdab_1^{\GinE_{N + 1}} \eqlaw \max\{ \lambdab_1^{\GinE_N} , \gammab_{N + 1} \} $ is not proven with an almost sure construction. 

\medskip
\begin{remark}\label{Rk:BHNY}
By ``almost sure construction'', we mean a deterministic equality that would imply the probabilistic result once ``probabilised''. Such an identity was for instance given in \cite{BHNY} for the characteristic polynomial of a random matrix distributed according to the $ CUE_N $. Setting $ Z_{U_N}(X) := \det(X I_N - U_N) $, one has $ Z_{U_N}(1) \eqlaw \prod_{k = 1}^N X_k $ where $ X_k $ are explicit independent random variables. Such a decomposition comes from the decomposition into a product of transvections of a unitary matrix which implies the probabilistic identity $ U_N \eqlaw H_{u_1} H_{u_2} \dots H_{u_N} $ where the random transvections $ H_{u_i} : x \mapsto x - \prth{ 1 + \frac{1 - e_1^* u_i}{ 1 - \overline{e_1^* u_i } } }  \frac{ ( e_1 - u_i )^* x}{ \abs{e_1 - u_i}^2 }  (e_1 - u_i) $ are independent (the vector $ u_i $ is selected uniformly on the unit sphere and $ (e_k)_k $ is the canonical basis of $ \Cc^N $). The previous identity in law comes then from the deterministic identity (see \cite{BHNY} for the details) $$ \det(I_N - H_{u_1} H_{u_2} \dots H_{u_N}) = (1 - e_1^* H_{u_1} e_1)\det(I_{N - 1} -  H_{u_2} \dots H_{u_N}) $$ 
Taking the Mellin transform of this identity gives $ \Esp{ Z_{U_N}(1)^s \overline{Z_{U_N}(1)^w} } = \prod_{k = 1}^N \Esp{ X_k^s \overline{X_k^w} } $, an identity equivalent to the celebrated Selberg integral in the particular case $ \beta = 2 $ (see \cite{BNRJacobi} for a generalisation to the $ C\beta E_N $). Such decompositions allow thus to give new proofs of classical identities with almost no computations, and a better understanding of why some ``magical'' identities hold.
\end{remark}


\medskip
\begin{remark}
Although there exists a natural Markovian coupling between $ GU\!E_N $ and $ GU\!E_{N + 1} $, the \textit{minor process} \cite{BorodinGorinBeta}, the Courant-Fisher minimax characterisation of the eigenvalues of a Hermitian matrix is not likely to give the desired result when probabilised with this coupling since it can be stationnary (this is seemingly not the case of $ (W_k)_k $).

Of course, several other Markovian couplings are possible, although arguably less natural than the minor one. We will come back to this problem in the future.  
\end{remark}

\medskip 
\subsection{Note on $ \lambdab_{N, N} $}\label{SubSec:RMT:Min}

We quickly treat the case of the minimal eigenvalue $ \lambdab_{N, N} $ in any previously defined matrix ensemble. This random variable is also the minimum of $ N $ independent random variables, as, using the Andr\'eieff-Heine-Szeg\"o formula \eqref{Eq:AndreieffHeineSzego}, one obtains
\begin{align*}
\Prob{\lambdab_{N, N} \geq s} = \frac{1}{\Ze_N(\infty)} \prod_{k = 0}^{N - 1} \norm{ \widetilde{H}_k(\cdot \vert s) }^2_{L^2([s, +\infty), \mu)}
\end{align*}

This is the half-line $ [s, +\infty) $ which is relevant in this case, and the orthogonal polynomials for the restricted measure were denoted $ \widetilde{H}_k(\cdot \vert s) $ in remark~\ref{Rk:GUEmax}. Using the same manipulations as in theorem~\ref{Thm:GUEmax}, one then finds
\begin{align*}
\Prob{ \lambdab_{N, N} \geq s} = \frac{1}{\Ze_N(\infty)} \prod_{k = 0}^{N - 1} \norm{ \widetilde{H}_k(\cdot \vert \cdot) }^2_{L^2([s, +\infty), \mu)} = \prod_{k = 0}^{N - 1} \frac{\norm{ \widetilde{H}_k(\cdot \vert \cdot) }^2_{L^2([s, +\infty), \mu)}}{\norm{ \widetilde{H}_k(\cdot \vert -\infty) }^2_{L^2(\Rr, \mu)} }
\end{align*}

In the case of the $ \gue_N $, one has for instance with independent $ (\widetilde{W}_k)_k $
\begin{align}\label{EqMax:Min:GUE}
\boxed{ \lambdab_{N, N}^{(GU\!E)} \eqlaw \min_{0 \leq k \leq N - 1} \widetilde{W}_k}
\end{align}
where
\begin{align*}
f_{\widetilde{W}_k }(x) = \frac{\widetilde{H}_k(x \vert x)^2}{\sqrt{k!}} \frac{e^{-x^2/2}}{\sqrt{2\pi}} 
\end{align*}

Here again, the most interesting question is the equivalent of Question~\ref{Q:CouplingGUE:max}, i.e.

\begin{question}\label{Q:CouplingGUE:min}
Can one find an almost sure construction that proves \eqref{EqMax:Min:GUE} or its generalisations to other determinantal measures~?
\end{question}

For instance, the construction $ \Hb_{\! N + 1} = \Hb_{\! N  } \oplus (\Xb_{N + 1}, a_{N + 1}) $ implies a Markovian coupling. Here, $ \oplus $ designates the operation that puts $ \Hb_{\! N  } $ in the top $ N \times N $ block of an $ (N + 1)\times (N + 1) $ matrix that has the vector $ \Xb_{N + 1} $ (of size $N$) in the upper right block, its transconjugated in the lower left block and the real $ a_{N + 1} $ at the $ (N + 1, N + 1) $ index (or the lower right block). 

With such a construction (with no probability), one looks for an equality of the type
\begin{align*}
\lambda_{\min}(\Hb_{\! N + 1}) = \min\ensemble{\lambda_{\min}(\Hb_{\! N }), \ \Fe_{N + 1}( \Xb_{N + 1}, a_{N + 1} )  }
\end{align*}
for a certain functional $ \Fe_{N + 1} $. The probabilisation of such an identity would then imply the result, with $ \widetilde{W}_N \eqlaw \Fe_{N + 1}(\Xb_{N + 1}, a_{N + 1} ) $. We will come back to this problem in the future.

\medskip\medskip
\subsection{Painlev\'e representation of $ \He_n $ and a new expression of $ F_{\TW_2} $}\label{SubSec:RMT:Painleve}

\subsubsection{Painlev\'e representation of $ \He_n $} 

The function $ \He_n $ used in \eqref{Def:Law:PsiSquare:GUE} can be described in a more explicit way. 

\medskip

Recall that the monic orthogonal polynomials $ (P_k)_{k \geq 0} $ for $ L^2([a, b], \mu) $ with $ a, b \in \overline{\Rr} $ are given by the following ``expectations of characteristic polynomials of random matrix ensembles'' \cite[(2.2.10) p. 27]{Szego} 
\begin{align*}
P_n(X) = \frac{ 1 }{ n! \, D_n } \int_{ \crochet{a, b}^n } \prod_{k = 1}^n (X - t_k) \Delta(\tb)^2 d\mu^{\otimes n}(\tb), \qquad D_n := \frac{1}{n!} \int_{ \crochet{a, b}^n }  \Delta(\tb)^2 d\mu^{\otimes n}(\tb)
\end{align*}

In particular, for
\begin{align*}
\mu(dx) := \frac{e^{-x^2/2} }{\sqrt{2\pi}} dx, \qquad a := -\infty, \qquad b := t
\end{align*}
one gets
\begin{align*}
H_n(X\vert t) = \frac{ 1 }{ n! \, D_n(t) } \int_{ (-\infty, t)^n } \prod_{k = 1}^n (X - t_k) \Delta(\tb)^2 d\mu^{\otimes n}(\tb), \qquad D_n(t) := \frac{1}{n!} \int_{ (-\infty, t)^n }  \Delta(\tb)^2 d\mu^{\otimes n}(\tb)
\end{align*}

If $ \Hb_{\! n} \sim GU\!E_n $ with eigenvalues $ \lambdab_{1, n} \geq \lambdab_{2, n} \geq \cdots  $, and if $ Z_{\Hb_{\! n}}(x) := \det(x I_n - \Hb_{\! n}) $, one thus has 
\begin{align*}
H_n(x\vert t) = \frac{\Esp{ Z_{\Hb_{\! n}}(x) \Unens{ \lambdab_{1, n} \leq t} } }{\Prob{\lambdab_{1, n} \leq t } } = \Esp{ Z_{\Hb_{\! n}}(x) \big\vert \lambdab_{1, n} \leq t} 
\end{align*}

Using the notations of \cite{ForresterWitteTau2and4}\footnote{
Note that the Gaussian measure is not the probabilistic one in \cite{ForresterWitteTau2and4} and in \cite{ForresterBook, TracyWidomCMP}, hence some slight changes of conventions.
}, define for all $ a \in \Nn $ 
\begin{align*}
\widetilde{E}_n(x ; a) := \frac{1}{n! } \int_{ (-\infty, x)^n } \prod_{j = 1}^n (x - t_j)^a \Delta(\tb)^2 d\mu^{\otimes n}(\tb) 
\end{align*}

Then,
\begin{align}\label{Eq:HnXXwithEtilda}
\boxed{\He_n(x) := H_n( x \vert x) = \frac{ \widetilde{E}_n(x ; 1) }{ \widetilde{E}_n(x ; 0) } }
\end{align}

The quantities $ \widetilde{E}_n(x ; a) $ were studied by Forrester-Witte in \cite[(1.7)]{ForresterWitteTau2and4}. This article is an application of the Jimbo-Miwa $ \sigma $-form approach to the Painlev\'e IV equation \cite{JimboMiwaSigmaForm} and the $ \tau $-function theory of the Painlev\'e II to VI equations due to Okamoto \cite{Okamoto0, Okamoto1, Okamoto2, Okamoto3} later refined by Noumi-Yamada \cite{NoumiYamadaPainleve}~; see also Forrester \cite[ch. 8 p. 328]{ForresterBook} for a historical and a self-contained introduction on the topic and Noumi \cite{NoumiPainleveBook} for a detailed study of $ P_{\!I\!I} $ and $ P_{\! IV} $. 

It is proven in \cite[(1.8)/(4.14)]{ForresterWitteTau2and4} that
\begin{align*}
\widetilde{E}_n(x ; a) = \widetilde{E}_n(x_0 ; a) \exp\prth{ \int_{x_0}^x U_n(t ; a) dt }
\end{align*}
where $ y = U_n(\cdot \, ; a) $ satisfies the following \textit{$ \sigma $-form of the Painlev\'e IV equation} \cite[(1.9)/(4.15)]{ForresterWitteTau2and4}
\begin{align}\label{Def:Painleve:SigmaForm:IV}
(\sigma \mbox{-} P_{\! IV}^{ }) : 
\begin{cases}
(y'')^2 - 4 (t y' - y)^2 + 4 y' (y' - 2a) (y' + 2\purple{n}) = 0   \\
y(t ) = -2\purple{n} t - \frac{ \purple{n}(a + \purple{n}) }{ t } + O\prth{ \frac{ 1}{t^3} }  \qquad\mbox{when } t \to \blue{-}\infty. 
\end{cases}
\end{align}
the boundary conditions being given in \cite[(1.10)/(4.18)]{ForresterWitteTau2and4}.
%
%
%

Notice that
\begin{align*}
\widetilde{E}_n(x ; a) := \frac{1}{n! } \int_{ (-\infty, x)^n } \prod_{j = 1}^n (x - t_j)^a \Delta(\tb)^2 d\mu^{\otimes n}(\tb) \ \ \equivalent{x \to \blue{+}\infty} \ \ x^{n a} \times \widetilde{E}_n(+\infty ; 0)  
\end{align*}
implying that 
\begin{align*}
\He_n(x) \ \ \equivalent{x \to \blue{+}\infty} \ \ x^n 
\end{align*}

As a result, one cannot choose $ x_0 = +\infty $ except in the case $ a = 0 $ (which is the case considered by Tracy and Widom).

\medskip

For general $ x_0 $, one thus has
\begin{align}\label{Eq:DensityWkGUE}
\He_n(x) e^{-x^2/4} =  \He_n(x_0) e^{-x_0^2/4} \, \exp\prth{ \int_{x_0}^x \prth{ U_n(t ; 1) -  U_n(t ; 0) - \frac{t}{2} } dt }
\end{align}

\medskip

For $ a = 0 $, Tracy and Widom's result 
\begin{align}\label{Eq:CV:GUE2TW}
\frac{\lambdab_{1, N} - 2\sqrt{N}}{N^{-1/6}} \cvlaw{N}{+\infty} \TW_2
\end{align} 
gives with the notations of \cite[(1.18)/(5.2), (5.3)]{ForresterWitteTau2and4} (where $ u $ is replaced by $ \sigma $)
\begin{align*}
\widetilde{E}_N\prth{ 2 \sqrt{N} + N^{-\frac{1}{6}} s \, ; \, a = 0 } \tendvers{N}{+\infty} \exp\prth{-\int_s^{+\infty} r(t) dt }, \qquad r(t) =: \sigma(t ; a = 0)
\end{align*}
and for $ a \geq 1 $ \cite[(1.19)/(5.6), (5.7)]{ForresterWitteTau2and4}
\begin{align*}
C(s_0) \, e^{- a (2 \sqrt{N} + N^{-1/6} s )^2/4}\widetilde{E}_N\prth{ 2 \sqrt{N} + N^{-\frac{1}{6}} s \,  ; \, a } \tendvers{N}{+\infty} \widetilde{E}_\infty^\soft\prth{ s_0  ; \, a } \exp\prth{ \int_{s_0}^s \sigma(t ; a) dt } 
\end{align*} 
were $ y := \sigma(\cdot \, ; a) $ satisfies the following \textit{$ \sigma $- form of the Painlev\'e II equation} \cite[(1.17)/(5.10)]{ForresterWitteTau2and4}
\begin{align}\label{Def:Painleve:SigmaForm:II}
(\sigma\mbox{-}P_{\!I\!I}^{ }) 
\begin{cases}
(y'')^2 + 4 y'\, \prth{ \vphantom{a^{a^{a}}}  t y' - y } - 4 (y')^3 - a^2 = 0\\ 
y(t ) = \frac{t^2}{4} + \frac{ 4 a^2 - 1}{ 8 t } + O\prth{ \frac{ 1}{t^3} }  \qquad\mbox{when } t \to \blue{-}\infty. 
\end{cases}
\end{align}
the boundary conditions being given in \cite[(1.20)/(4.11)]{ForresterWitteTau2and4}
%
%
%

\medskip
\subsubsection{Required estimates for $ f_{W_n} $}\label{SubSubSec:RMT:Painleve:Estimates}

$ $

\medskip

\textbf{Caveat~:} We will now be concerned with the ``physicist convention'' of Tracy-Widom \cite{TracyWidomPLB, TracyWidomCMP} and Forrester-Witte \cite{ForresterWitteTau2and4} which is also the convention of Wu-Xu-Zhao \cite{WuXuZhao}. As a result, the probability expressed in theorem~\ref{Thm:Init:GUEmax} is given with $ t\sqrt{2} $ and the Gaussian density is $ e^{-t^2}/\sqrt{\pi} $. We will set $ \widehat{W}_n := \sqrt{2}W_n $ for the rescaled random variable. 

We will moreover use the function $ \sigma $ which is denoted by $ u $ in \cite{ForresterWitteTau2and4} and the function $p$ which is denoted by $ u $ in \cite{WuXuZhao}. These conventions are the ones of \cite[32.6(iii) p. 729]{NISThandbook} for the Hamiltonian description of the Painlev\'e functions. 

The Hamiltonian description goes as follows~: denote by $ H(q, p, t) $ the Hamiltonian of the Painlev\'e $ I\!I $ equation given by \cite[32.6.9 p. 729]{NISThandbook}
\begin{align*}
H(q, p, t) := \frac{p^2}{2} - \prth{q^2 + \frac{t}{2}}p - \prth{\alpha + \frac{1}{2}}q 
\end{align*}

The Hamilton-Jacobi equations $ q' = \partial_p H $, $ p' = -\partial_q H $ are given by \cite[32.6.10 \& 32.6.11 p. 729]{NISThandbook}
\begin{align*}
q' = p - q^2 - \tfrac{t}{2}, \qquad p' = 2qp + \alpha + \tfrac{1}{2}
\end{align*}

The equation satisfied by $q$ is then $ P_{\!I\!I}(\alpha) $ given by \eqref{Def:PainlevéII} and the equation satisfied by $p$ is the Painlev\'e $ X\!X\!X\!IV $ equation $ P_{\!X\!X\!X\!IV}(\alpha) $ (numbered in this way in the original classification of Painlev\'e and Gambier) given by \cite[32.6.12 p. 729]{NISThandbook} 
\begin{align}\label{Def:Painlevé:PXXXIV}
p p'' = \tfrac{(p')^2}{2}  + 2 p^3 - t p^2 - \tfrac{1}{2} (\alpha + \tfrac{1}{2})^2
\end{align}

Last, the equation satisfied by the function $ \sigma : t \mapsto H(q(t), p(t), t) $ is $ \sigma\mbox{-}P_{\!I\!I}(\alpha) $ given by \eqref{Def:Painleve:SigmaForm:II}. This function allows to ``reconstruct'' $ (p, q) $ with the relations \cite[32.6.14 \& 32.6.15 p. 729]{NISThandbook}~:
\begin{align}\label{Eq:Painlevé:PQwithSigma}
q = \frac{ 4 \sigma'' + 2\alpha + 1 }{8 \sigma'}, \qquad p = -2\sigma'
\end{align}

\medskip

In \cite[(1.6)]{WuXuZhao}, Wu, Xu and Zhao study
\begin{align*}
\Hb_{\! n}(x \vert \alpha, \omega) := \frac{1}{n!}\int_{\Rr^n} \prod_{k = 1}^n \abs{x - t_j}^{2\alpha} \prth{ \Unens{t_j \leq x} + \omega \Unens{t_j > x} } \Delta(\tb)^2 \mu^{\otimes n}(d\tb), \qquad \mu(dt) = e^{-t^2}\frac{dt}{\sqrt{\pi}}
\end{align*}

In particular, 
\begin{align*}
\Hb_{\! n}(x \vert \alpha, 0) := \frac{1}{n!}\int_{\Rr^n} \prod_{k = 1}^n \abs{(x - t_j)_+}^{2\alpha}  \Delta(\tb)^2 \mu^{\otimes n}(d\tb) = \widetilde{E}_n(x ; a = 2\alpha)
\end{align*}
as $ x_+ = \abs{x_+} $, and
\begin{align*}
\He_n(x) =  \frac{ \Hb_{\! n}(x \vert 1/2, 0)  }{\Hb_{\! n}(x \vert 0, 0)}
\end{align*}

Wu, Xu and Zhao \cite[thm. 3, (1.24)]{WuXuZhao} prove moreover that (in particular for $ \alpha = \frac{1}{2} $/$ a = 1 $ and $ \pink{\omega} = 0 $)
\begin{align*}
& \frac{d}{d\mu_n} \ln \Hb_{\! n}(\mu_n \vert \alpha, \pink{\omega}) = \sqrt{2n} \prth{ 2\alpha + \frac{\sigma(s\vert \alpha, \pink{\omega})}{n^{1/3} } + \alpha \frac{p(s\vert \alpha, \pink{\omega}) + s}{n^{2/3} } + O_s\prth{\frac{1}{n} } }, \\
& \hspace{+10cm} \mu_n := \sqrt{2n} + n^{-\frac{1}{6}} \frac{s}{\sqrt{2}}
\end{align*}
where $ \sigma(\cdot \vert \alpha = 1/2, \omega = 0) $ is the $ \sigma $-form of the Painlev\'e II function given in \eqref{Def:Painleve:SigmaForm:II} and $ p(\cdot \vert \alpha = \frac{1}{2}, \pink{\omega} = 0 ) $ is solution of the Painlev\'e $ X\!X\!X\!IV $ equation with the following particular boundary condition at $ \infty $ \cite[(1.20) \& (1.21)]{WuXuZhao} for $ \alpha \neq 0 $~:
\begin{align}\label{Def:Painlevé:PXXXIV:bis}
(P_{\!X\!X\!X\!IV})
\begin{cases}
p(s) p''(s)   = \frac{p'(s)^2}{2}  + 4 p(s)^3 + 2s p(s)^2 - 2 \alpha^2 \\
\qquad \ p(s) = \frac{\alpha}{\sqrt{s}}(1 - \alpha s^{-3/2} + O(s^{-3}) ), \qquad s \to \infty, \quad \arg(s) \in (-\frac{\pi}{3}, \frac{\pi}{3}]
\end{cases}
\end{align}

Last, \cite[(1.22), (1.17)]{WuXuZhao}/\cite[(5.11), \textbf{--}\,]{ForresterWitteTau2and4}~:
\begin{align}\label{Eq:Painleve:SigmaFormEstimates:II}
\begin{aligned} 
\sigma(s \vert \alpha, \pink{\omega} = 0) & = -\frac{s}{2} + \frac{16\alpha^2 - 1}{8} s^{-2} + O(s^{-7/2}), \qquad s \to -\infty \\
\sigma(s \vert \alpha, \pink{\omega} = 0) & = - \alpha \sqrt{s} \prth{ 1 + \alpha s^{-3/2} + O(s^{-3}) }, \qquad s \to \infty, \quad \arg(s) \in (-\frac{\pi}{3}, \frac{\pi}{3}]
\end{aligned}
\end{align}

As a result, using the physicist convention for the Gaussian, one has 
\begin{align*}
f_{W_n}(s) & = \frac{\He_n(s)^2 e^{-\blue{s^2} }}{n!\Ze_n\sqrt{2\pi} } = \prth{ \frac{ \Hb_{\! n}(s \vert 1/2, 0)  }{\Hb_{\! n}(s \vert 0, 0)} }^2 \frac{  e^{-\blue{s^2} }}{n!\Ze_n\sqrt{ \pi} } \\
               & = \exp\prth{ 2\ln\He_n(c) +  2\int_c^s \frac{d}{dt}\ln\He_n(t) dt } \frac{  e^{-\blue{s^2} }}{n!\Ze_n\sqrt{2\pi} } \\
               & = \exp\prth{ 2\ln\He_n(c) +  2\int_c^s \prth{ \frac{d}{dt} \ln \Hb_{\! n}(t \vert 1/2, 0) -   \frac{d}{dt} \ln \Hb_{\! n}(t \vert 0, 0) } dt } \frac{  e^{-\blue{s^2} }}{n!\Ze_n\sqrt{ \pi} } \\
               & = \frac{\He_n(c)^2 e^{-\blue{c^2} }}{n!\Ze_n\sqrt{ \pi} } \exp\prth{ 2\int_c^s \prth{ \frac{d}{dt} \ln \Hb_{\! n}(t \vert 1/2, 0)  -  \frac{d}{dt} \ln \Hb_{\! n}(t \vert 0, 0) - \blue{t} } dt } 
\end{align*}

Set
\begin{align*}
t_n & := \sqrt{2n} + n^{-\frac{1}{6}}\frac{t}{\sqrt{2}}, \qquad
s_n := \sqrt{2n} + n^{-\frac{1}{6}}\frac{s}{\sqrt{2}}, \qquad 
c_n := \sqrt{2n} + n^{-\frac{1}{6}}\frac{c}{\sqrt{2}} \\
\delta \sigma & := \sigma\prth{\cdot \vert \tfrac{1}{2}, \pink{\omega} } - \sigma\prth{\cdot \vert 0, \pink{\omega} } = \sigma\prth{\cdot \vert a = 1, \pink{\omega} } - \sigma\prth{\cdot \vert a = 0, \pink{\omega} },  
\qquad \pink{\omega} := 0
\end{align*}

Then, 
%
%
\begin{align*}
f_{\widehat{W}_n}(s_n)   & =: \frac{\He_n(c_n)^2 e^{-\blue{c_n^2} } }{n! \Ze_n \sqrt{ \pi}} 
                    \exp\prth{ \!  2 \! \int_{c}^{s}  \! \prth{  \! \sqrt{2n} +  \sqrt{2n} \frac{\delta\sigma^{(n)}(t_n)}{n^{1/3}} + \sqrt{2n} \frac{1}{2} \frac{p^{(n)}(t_n\vert \frac{1}{2}, \pink{0}) + t_n}{n^{2/3} }  - \blue{t_n}  \! } n^{ - \frac{1}{6} } dt   }  \\
                & = \frac{\He_n(c_n)^2 e^{-\blue{c_n^2} } }{n! \Ze_n \sqrt{ \pi}} 
                    \exp\prth{ \! n^{ - \frac{1}{6}} 2\sqrt{2n} \int_{c}^{s} \prth{ 1 + \frac{\delta\sigma^{(n)}(t_n)}{n^{1/3}} +   \frac{p^{(n)}(t_n\vert \frac{1}{2}, \pink{0}) + t_n}{2 \, n^{2/3} } - \frac{\blue{t_n}}{\sqrt{2n}} } dt    } \\
                & = \frac{\He_n(c_n)^2 e^{-\blue{c_n^2} } }{n! \Ze_n \sqrt{ \pi} } 
                         \exp\prth{  n^{\frac{1}{3}  } 2\sqrt{2} 
                                  \int_{c}^{s}  \prth{ \frac{\delta\sigma^{(n)}(t_n)}{n^{1/3}} +  \frac{p^{(n)}(t_n\vert \frac{1}{2}, \pink{0}) + t_n}{ 2 \, n^{2/3} }  - \frac{ n^{ - \frac{1}{6} } \blue{t/\sqrt{2} } }{n^{1/2} \sqrt{2}} } dt  } \\
                & = \frac{\He_n(c_n)^2 e^{-\blue{c_n^2} } }{n! \Ze_n \sqrt{ \pi} } 
                         \exp\prth{  n^{\frac{1}{3}  } 2\sqrt{2} 
                                  \int_{c}^{s}  \prth{ \frac{\delta\sigma^{(n)}(t_n)}{n^{1/3}} +  \frac{p^{(n)}(t_n\vert \frac{1}{2}, \pink{0}) + t_n}{ 2 \, n^{2/3} }  - \frac{   \blue{t } }{2\, n^{2/3} } } dt  } \\
                & = \frac{\He_n(c_n)^2 e^{-\blue{c_n^2} } }{n! \Ze_n \sqrt{ \pi} } 
                         \exp\prth{    2\sqrt{2} 
                                  \int_{c}^{s}  \prth{  \delta\sigma(t )  +  \frac{p(t \vert \frac{1}{2}, \pink{0}) + t}{ 2 \, n^{1/3} }  - \frac{   \blue{t } }{2\, n^{1/3} } + O_t\prth{ \frac{1}{n^{2/3}} } } dt  } \\
                & = \frac{\He_n(c_n)^2 e^{-\blue{c_n^2} } }{n! \Ze_n \sqrt{ \pi} } 
                         \exp\prth{    2\sqrt{2} 
                                  \int_{c}^{s}  \prth{  \delta\sigma(t )  +  \frac{p(t \vert \frac{1}{2}, \pink{0})  }{ 2 \, n^{1/3} }    + O_t\prth{ \frac{1}{n^{2/3}} } } dt  }
\end{align*}

Using the second relation of \eqref{Eq:Painlevé:PQwithSigma}, one then has
\begin{align*}
p(t \vert \tfrac{1}{2}, \pink{0} ) = -2\frac{d}{dt}\sigma(t \vert \tfrac{1}{2}, \pink{0} )
\end{align*}
and
\begin{align*}
f_{\widehat{W}_n}(s_n)    
                & = \frac{\He_n(c_n)^2 e^{-\blue{c_n^2}  + 2\sqrt{2} n^{-1/3} \sigma(c \vert 1/2, 0) } }{n! \Ze_n \sqrt{ \pi} } 
                         \exp\prth{    2\sqrt{2} 
                                  \int_{c}^{s}    \delta\sigma(t ) dt - O\prth{ n^{-1/3} \sigma(s \vert \tfrac{1}{2}, 0)} } \\
                & = \frac{\He_n(c_n)^2 e^{-\blue{c_n^2}  + 2\sqrt{2} \, n^{-1/3} \sigma(c \vert 1/2, 0) } }{n! \Ze_n \sqrt{ \pi} } 
                         \exp\prth{    2\sqrt{2} 
                                  \int_{c}^{s}    \delta\sigma(t ) dt}\prth{ 1 - O\prth{ n^{-1/3} \sigma(s \vert \tfrac{1}{2}, 0)} }  
\end{align*}
namely
\begin{align}\label{Eq:Estimate:WuXuZhao}
\begin{aligned}
&   f_{\widehat{W}_n}(s_n) - f_{\widehat{W}_n}(c_n) e^{  2\sqrt{2} \, n^{-1/3} \sigma(c \vert 1/2, 0) }  \exp\prth{    2\sqrt{2} \int_c^s    \delta\sigma(t ) dt}  \\
                                 & \hspace{+7cm} = - O\prth{ n^{-1/3} \sigma(s \vert \tfrac{1}{2}, 0) e^{   2\sqrt{2} \int_c^s \delta\sigma(t ) dt}  }   
\end{aligned}
\end{align}


The proof of theorem~\ref{Thm:NewExprTW2} will impose to choose the constant $ c \equiv c^*_n $ such that
\begin{align*}
n^{1/6} f_{\widehat{W}_n}(c_n) e^{  2\sqrt{2} \, n^{-1/3} \sigma(c \vert 1/2, 0) } = 1, 
\qquad c_n := \sqrt{2n} + n^{-\frac{1}{6}}\frac{c}{\sqrt{2}}
\end{align*}
%
%
%

One can show that $ c^*_n \to +\infty $ and that $ s \mapsto \sigma(s \vert \tfrac{1}{2}, 0) e^{  - 2\sqrt{2} \int_s^{+\infty} \delta\sigma(t ) dt} \in L^1([s, +\infty)) $ for all $s$ using the \S~\ref{SubSubSec:RMT:Painleve:Comparison}.

\medskip
\subsubsection{Poisson approximation for maxima of independent random variables} 

The problem of estimating the \textit{maximum} of a sequence of independent random variables $ (W_k)_{1 \leq k \leq n} $ can be classically restated as a problem of Poisson approximation for \textit{sums} of independent random variables \cite[ch. 3]{Feidt} writing
\begin{align*}
\ensemble{\max_{1 \leq k \leq n} W_k \leq x} = \ensemble{\sum_{k = 1}^n \Unens{W_k > x} = 0}
\end{align*}

Since the random parametric indicators $ B_k := \Unens{W_k > x} $ are Bernoulli random variables of expectation $ p_k := \Prob{W_k > x} $, one can proceed to a Poisson approximation of $ S_n := \sum_{k = 1}^n B_k $ \cite[\S~4.1]{Ross} 
\begin{align*}
\Prob{ S_n = 0} - \Prob{\Poisson\prth{ \sum_{k = 1}^n p_k } = 0 } = O\prth{ \sum_{k = 1}^n p_k^2  \times \min\ensemble{1, \frac{1}{\sum_{k = 1}^n p_k} } }
\end{align*}
with
\begin{align*}
\Prob{\Poisson\prth{ \sum_{k = 1}^n p_k } = 0 } = e^{- \sum_{k = 1}^n p_k } 
\end{align*}

In the $ GU\!E_{n + 1} $ case, this last Poisson approximation is performed for the previous independent random variables $ (W_k)_k $ defined in \eqref{Def:Law:PsiSquare:GUE}. Set
\begin{align*}
s_n := 2\sqrt{n} + n^{-1/6}s
\end{align*}

As the first random variables do not depend on $ n $, the first probabilities are very small~: for $ k = O(1) $
\begin{align*}
p_k := \Prob{W_k > s_n } \leq \frac{1}{s_n^2} \Esp{ \abs{W_k}^2 } = O_s\prth{ \frac{1}{n} }
\end{align*}

We will see in \eqref{Eq:Estimate:limfWNalpha} that $  f_{W_{m_n}}(s_n) \sim n^{-1/6} \Qe(s + c) $ with $ \Qe \in L^2([s, +\infty)) $ for all $s$ if $ m_n = n - c n^{1/3} $. As a result, with $ x_n = 2\sqrt{n} + n^{-1/6} x $
\begin{align*}
p_{m_n} := \Prob{W_{m_n} > s_n } \leq 
               \int_s^{+\infty} \prth{\frac{x_n}{s_n}}^{\!\! 2} f_{W_{m_n}}(x_n) n^{-1/6} dx =  O_s\prth{ n^{-1/3} }
\end{align*}

The relevant part of the maximum mass is thus concentrated on the extreme random variables $ (W_{[n\alpha]})_{1 - \varepsilon_n \leq \alpha \leq 1} $ with $ \alpha = 1 - c n^{-2/3} $ and one can neglect the first random variables to only consider the contribution of the last ones~: with $ m := \pe{ n (1 -  \varepsilon_n) } < n $, $ \varepsilon = c n^{-2/3} $, one has with obvious notations
\begin{align*}
\delta_n  := \abs{\Prob{S_n = 0} - \Prob{S_{m\to n} = 0}} & = \Prob{ S_{m\to n} = 0, S_n - S_{m\to n} \geq 1} \\
             & \leq \Esp{ S_n - S_{m\to n} } = \Esp{ S_m } = \sum_{k = 1}^m p_k
\end{align*}
which implies  
\begin{align*}
\Prob{S_n = 0} = \Prob{S_{m\to n} = 0} + O\prth{ \sum_{k = 1}^m p_k } = \Prob{S_{m \to n} = 0} + O_s\prth{ n^{-1/3} }
\end{align*}
for the choice
\begin{align}\label{Eq:ChoixDeM}
m = n - O(n^{1/3})
\end{align}

For this choice of $m$, we also set
\begin{align*}
P_n := \sum_{k = 1}^n p_k, \qquad P^*_n := \sum_{k = m + 1}^n p_k, \qquad  m := \pe{ n (1 -  \varepsilon_n) } = \pe{n - c n^{1/3}}
\end{align*}
so that
\begin{align*}
P_n = P_n^* + O_s\prth{ n^{-1/3} }  
\end{align*}

Last, we set
\begin{align*}
P_n^{(2)} := \sum_{k = 1}^n p_k^2  =  P_n^{*\!,\, 2} + O_s\prth{ n^{-1/3} }, \qquad P_n^{*\!,\, 2} := \sum_{k = m + 1}^n p_k^2
\end{align*}

If $ P_n = O(1) $, the Poisson approximation amounts thus to 
\begin{align*}
\Prob{ S_n = 0} - \Prob{\Poisson\prth{ P_n } = 0 } = O\prth{ P_n^{(2)} }, 
\end{align*}
with
\begin{align*}
P_n = P_n^* + O_s\prth{n^{-1/3}}, \quad P_n^{(2)} = P_n^{*\!,\, 2} + O_s\prth{ n^{-1/3} }
\end{align*}

\medskip\newpage
\subsubsection{A new expression of $ F_{\TW_2} $}

\begin{theorem}[Expression of $ F_{\TW_2} $]\label{Thm:NewExprTW2}
\begin{shaded}
One has
\begin{align}\label{EqTW:PainlevureNouveau}
\Prob{ \frac{\lambdab_{1, N} - 2\sqrt{N}}{N^{-1/6}} \leq \sqrt{2}\, s } \tendvers{N}{+\infty } \boxed{\Prob{\TW_2 \leq \sqrt{2}\, s} = \exp\prth{ - \Qe^2 * \id_+(s) } }
\end{align}
with $ f*g(x) := \int_\Rr f(t) g(x - t) dt = g*f(x) $, $ \id_+   : x \longmapsto x_+ := x \Unens{x \geq 0} $ and 
\begin{align}\label{Def:Q^2PII}
\Qe(x) :=   \exp\prth{ - \int_x^{+\infty} \delta \sigma }, \qquad \delta \sigma := \sigma(\cdot ; 1) - \sigma(\cdot ; 0)
\end{align}
where $ \sigma(\cdot ; a) $ is the solution of the $ \sigma $-form of the Painlev\'e II equation given in \eqref{Def:Painleve:SigmaForm:II}.
\end{shaded}
\end{theorem}


\begin{proof}
With $ \widehat{W}_k := W_k/\sqrt{2} $, one has
\begin{align*}
\Prob{ \widehat{W}_{[N\alpha]} > \sqrt{2N} + N^{-1/6}2^{-\frac{1}{2}} s } & =  \int_{\sqrt{2N} + N^{-1/6}2^{-\frac{1}{2}} s}^{+\infty} f_{\widehat{W}_{[N\alpha]}} (x)  dx  \\ 
                 & = \int_s^{+\infty} N^{-1/6} f_{\widehat{W}_{[N\alpha]}}(\sqrt{2N} + N^{-1/6}2^{-\frac{1}{2}} y) dy
\end{align*} 
and, with a Stieltjes integral,
\begin{align*}
P^*_N & := \sum_{k = m + 1}^N p_k 
               = \int_m^N p_\pe{u} du    
               =  \int_0^1 p_\pe{m + (N - m)v} (N - m)dv  \\
              & = \int_{  N^{2/3}}^0 p_\pe{N\alpha_N(c)} \, d(N\alpha_N(c))   , \qquad \alpha_N(c) := 1 - \frac{c}{ N^{2/3}}  \\
              & =  \int_0^{  N^{2/3}}  N^{1/3}    \Prob{ \widehat{W}_\pe{N- c N^{1/3}}  \geq s_N} dc  \\
              & = \int_0^{ N^{2/3}} N^{1/3}  \int_s^{+\infty} N^{-1/6} f_{\widehat{W}_\pe{N- c N^{1/3} } }(\sqrt{2N} + 2^{-1/2}N^{-1/6} y) dy \, dc  \\
              & = \int_{(\Rr_+)^2 } \Unens{ c \leq N^{2/3}, y \geq s } N^{ 1/6} f_{\widehat{W}_\pe{N- c N^{1/3} } }(\sqrt{2N} + 2^{-1/2}N^{-1/6} y) dy \, dc 
\end{align*}

We will now show that 
\begin{align}\label{Eq:Estimate:limfWNalpha}
N^{+1/6} f_{\widehat{W}_\pe{N- c N^{1/3} }}(\sqrt{2N} + 2^{-1/2}N^{-1/6} y) \tendvers{N}{+\infty} \Qe(y + c)^2
\end{align}

This implies with the Wu-Xu-Zhao estimate \eqref{Eq:Estimate:WuXuZhao} that $ P^*_N = P^*_\infty(s) + O(N^{-1/3} P^{*\!,\, 3}_\infty  ) $ with
\begin{align*}
P^*_\infty(s) & = \int_s^{+\infty } \int_{\Rr_+} \Qe(y + c)^2  dc \, dy = \int_s^{+\infty } \int_y^{+\infty} \Qe(u)^2 du \, dy   = \int_\Rr \Qe(u)^2 (s - u)_+ du\\
              & = \Qe^2 *\id_+(s)
\end{align*}
and, with $ \sigmab_1 = \sigma(\cdot ; 1) \sim \sigma\mbox{-}P_{\!I\!I}(a = 1) $ (or $ \alpha = \frac{1}{2} $) defined in \eqref{Def:Painleve:SigmaForm:II}
\begin{align*}
P^{*\!,\, 3}_\infty(s) & = \int_s^{+\infty } \int_{\Rr_+} \sigmab_1(y + c) \Qe(y + c)^2  dc \, dy  
                 = (\sigmab_1 \Qe^2) * \id_+(s)
\end{align*}

Moreover,  
\begin{align*}
P^{*\!,\, 2}_N & := \sum_{k = m + 1}^n p_k^2  \\ 
              & =  \int_0^{ N^{2/3}}  N^{1/3}    \Prob{W_{\pe{N\alpha_N(c)} } \geq s_N}^2 dc, \qquad \alpha_N(c) := 1 - \frac{c}{ N^{2/3}} \\
              & = N^{-1/3} \int_0^{  N^{2/3}}  \prth{ N^{1/3}    \Prob{ W_\pe{N\alpha_N(c)}  \geq s_N } }^{\! 2} dc \\
              & =  O(N^{-1/3}) \times   \int_{\Rr_+}  \prth{  \int_s^{+\infty }   \Qe(y + c)^2 dy  }^{\!\! 2} dc \\
              & = O_s\prth{N^{-1/3} }
\end{align*}

The fact that $ \int_s^{+\infty}\sigmab_1^k \Qe^2 < \infty $ (with $ k \in \{0, 1\} $) comes from the asymptotics at infinity given in the second line of \eqref{Def:Painleve:SigmaForm:II} (that implies\footnote{
We will see in lemma~\ref{Lemma:q=Q} that we have in fact $ \Qe = q \sim P_{\!I\!I}(0) $, hence that $ \Qe(x)^2 \sim \Ai(x)^2 $ for $ x \to +\infty $, which is much more refined than the sole difference of equivalents at infinity.
} $ \Qe(x)^2 = O(x^{-2}) $) and \eqref{Eq:Painleve:SigmaFormEstimates:II}.

The two terms $ P^{*\!,\, 2}_N  $ and $ P^{*\!,\, 3}_N  $, coming respectively from \eqref{Eq:Estimate:WuXuZhao} and from the Poisson approximation give an unoptimal speed of convergence\footnote{ 
The optimal speed of convergence is $ O(N^{-2/3}) $. It was obtained by Johnstone and Ma \cite{JohnstoneMa} using a very particular recentering $ \mu_N \neq 2\sqrt{N} $ \cite[(15)]{JohnstoneMa}. We are not concerned with such a problem.
}, but are sufficient for our purpose.

\medskip

We now prove \eqref{Eq:Estimate:limfWNalpha}. Define
\begin{align*}
g_N(M, y) & := N^{ 1/3} f_{\widehat{W}_M}(\sqrt{2N} + 2^{-1/2}N^{-1/6} y), \qquad M := \pe{N- c N^{1/3} }
\end{align*}

Of course, the Wu-Xu-Zhao estimate \eqref{Eq:Estimate:WuXuZhao} already gives the convergence for the right tuning of parameters, but we will first describe a more classical way of obtaining the result without the remainder, in the vein of the Forrester-Witte study \cite{ForresterWitteTau2and4}. 

Integrating at a certain $ x_N $ that will be choosen later and setting $ x_N^* := \frac{x_N - 2\sqrt{N}}{N^{-1/6}} $, one has 
\begin{align*}
g_N([N\alpha], y) & = N^{1/3} g_N([N\alpha], x_N) \\
                & \qquad\qquad  \exp\prth{  2 \sqrt{2} \int_{x_N^*}^y N^{-1/6} \prth{  U_{[N\alpha]}(t ; 1) -  U_{[N\alpha]}(t ; 0) - t }\big\vert_{t = 2\sqrt{N} + N^{-1/6} T } dT }    
\end{align*}

The function $ y := U_{[N\alpha]}(\cdot \, ; a) $ satisfies \eqref{Def:Painleve:SigmaForm:IV} with $ n = [N\alpha] $. We thus consider
\begin{align*}
\sigma_N^{(\alpha)}(\cdot ; a) : t \mapsto  \www_N & \prth{ U_{[N\alpha]}(\mu_N + \www_N t ; a) - a \times (\mu_N + \www_N t) } , \\
                 & \quad \mu_N := \sqrt{2N}, \quad \www_N := \frac{1}{N^{1/6}\sqrt{2} } 
\end{align*}

Setting 
\begin{align*}
Y(t) & := \sigma_N^{(\alpha)}(\cdot ; a)  \\
Y'(t) & := \frac{d}{dt} \sigma_N^{(\alpha)}(t; a) = \www_N^2 \frac{d}{dt} U_{[N\alpha]}( \mu_N + \www_N t ; a) - a\www_N \\
Y''(t) & := \prth{\frac{d}{dt}}^2 \sigma_N^{(\alpha)}(t; a) = \www_N^3 \prth{\frac{d}{dt}}^2 U_{[N\alpha]}( \mu_N + \www_N t ; a)
\end{align*}
so that, dropping the indices $N$, the equation \eqref{Def:Painleve:SigmaForm:IV} becomes after multiplication by $ \www^6 $
\begin{align*}
0 & =   (Y'')^2 - 4 \www^6 \prth{ (\mu + \www t) \tfrac{Y' + a\www}{\www^2} -  \tfrac{Y + a \mu + a \www t}{\www} }^2 + 4 \www^6 \tfrac{Y' + a\www}{\www^2} \prth{ \tfrac{Y' + a\www}{\www^2} - 2a} \prth{ \tfrac{Y' + a\www}{\www^2} + 2\purple{[N\alpha]} }   \\
             & =    (Y'')^2 - 4 \www^2  \prth{ \vphantom{\big(} (\mu  + \www t ) Y' -  \www Y }^2   +  4(Y' + a\www) \prth{  Y' + a\www(1  - 2 \www) } \prth{  Y' + a\www  + 2\purple{[N\alpha]}\www^2 } 
\end{align*}


An expansion that supposes $ Y $ and its derivatives to be of order 1 shows that
\begin{align*}
\Le_N(Y)  & =  (Y'')^2  + 4  (Y')^3  - 4 \prth{   (Y')^2  t  -  Y Y'  } - a^2 \\
          & \hspace{+4cm} + 4 (Y')^2 \prth{ 2 \pink{[N\alpha]  \www^2} -  \pink{\mu^2 \www^2}  }+ O_{Y, t}(\red{\www})
\end{align*}

Since $ \mu^2 = 2N $, $ \Le_N(Y) $ converges iff 
\begin{align*}
\pink{N \www^2}(1 - \alpha) = \pink{N^{2/3}}(1 - \alpha) \tendvers{N}{+\infty} c = O(1) \qquad\Longleftrightarrow\qquad \alpha \equiv \alpha_N =  1 - \frac{c + o(1)}{ N^{2/3} }
\end{align*}

In such a case, $ \sigma_N^{(\alpha)}(\cdot ; a) $ converges locally uniformly to $ \sigma^{(c)}(\cdot ; a) = \sigma(\cdot + c ; a) $ satisfying the following shifted $ \sigma $-form of the Painlev\'e II equation (as setting $ t  \leftarrow t + c $ gives back the original equation) with the same boundary condition as \eqref{Def:Painleve:SigmaForm:II}
\begin{align*}
(Y'')^2  +  4 Y'\prth{ (Y')^2 - (t - c) Y' + Y } - a^2   = 0
\end{align*}

As a result, locally uniformly
\begin{align*}
\delta \sigma_N^{(\alpha)} := \sigma_N^{(\alpha)}(\cdot\, ; 1) - \sigma_N^{(\alpha)}(\cdot\, ; 0) \tendvers{N}{+\infty} \delta \sigma^{(c)} =: \delta \sigma(\cdot + c) = \sigma(\cdot + c ; 1) - \sigma(\cdot + c ; 0)
\end{align*}

This analysis explains the relevance of the rescaling, but does not give any speed of convergence. Such a speed is required to prove the result in a probabilistic way (since the convergence of the densities does not preclude from the usual weak-* pathology of ``escape of the mass at infinity'')~; it is furnished by the Wu-Xu-Zhao estimate \eqref{Eq:Estimate:WuXuZhao} obtained with a Riemann-Hilbert problem.

\medskip


It remains to choose $ x_N $. We set it as the solution $ c $ to the equation
\begin{align*}
N^{1/6} f_{\widehat{W}_N}(c_N) e^{  2\sqrt{2} \, N^{-1/3} \sigmab_1(c) } = 1, 
\qquad c_N := \sqrt{2N} + N^{-\frac{1}{6}}\frac{c}{\sqrt{2}}
\end{align*}
i.e.
\begin{align*}
& N^{1/6} \He_N(x)^2  e^{-x^2} e^{  2\sqrt{2} \, N^{-1/3} \sigmab_1( N^{1/6} (x - \sqrt{2N}) ) }   = \sqrt{ \pi} \\
\quad\Longleftrightarrow\quad & x^2 -  2\ln\He_N(x ) + 2\sqrt{2} \, N^{-1/3} \sigmab_1( N^{1/6} (x - \sqrt{2N}) ) = \frac{\ln(N)}{6} + O(1)
\end{align*}

We are just concerned with the behaviour of $ x_N $ when $ N\to+\infty $ and not its unicity. As a result, if there are several such values, we choose the largest one~; nevertheless, for $ N $ and $x$ big enough ($ x \gg 2\ln(N) $), there is a unique such solution, as seen using the equivalent $ \He_N(x) \sim x^N $ and $ \sigmab_1(x) \sim -\sqrt{x} $ given in \eqref{Eq:Painleve:SigmaFormEstimates:II} when $ x \to +\infty $. 

Taking the logarithm of the equation and replacing $ \He_N(x) $ by $ x^N $ and $ \sigmab_1(s) $ by $ -\sqrt{s} $ yields the approximate equation 
%
%
\begin{align*}
x_N^2 - 2N \ln(x_N) - 2\sqrt{2} \, N^{-1/4} \sqrt{  x_N - \sqrt{2N}  } = \frac{\ln(N)}{6} + O(1)
\end{align*}

Since $ x_N - \sqrt{2N} \geq 0 $, one can neglect the term $  N^{-1/4} \sqrt{  x_N - \sqrt{2N}  } $ for the first order of the equation. Setting $ x_N^2 = N t_N $, the equation is equivalent to $ t_N - \ln(t_N) = N\ln(N) + \frac{\ln(N)}{6} + O(1) $ whose solution is given by means of the Lambert $W$ function \cite[\S~4.13 p. 111]{NISThandbook}, i.e. $ t_N = W( N\ln(N) + \frac{\ln(N)}{6} + O(1) ) $. The asymptotic of this function at infinity is  $ W(x) = \ln(x) - \ln\ln(x) + o(1) $ \cite[(4.13.10) p. 111]{NISThandbook}, hence $ t_N = \ln(N) + O(\ln\ln(N)) $, resulting in 
\begin{align*}
x_N = \sqrt{N \ln(N)} + O(\sqrt{N}\ln\ln(N))
\end{align*}

Of course, as long as $ \alpha_N = O(1) $, this behaviour is valid for $ x_{\pe{N\alpha_N}} $. In the case where $ \alpha_N $ is close to $0$, we cut the integral up to $ 2 N^{2/3} - \varepsilon $ and choose $ \varepsilon \ll N^{-1/6} $ so that $ x_{\pe{N\alpha_N}} \gg \sqrt{2 \, N} $~; the remaining integral on $ (2 N^{2/3} - \varepsilon, 2 N^{2/3} ) $ is bounded by the sup of the probability (which is 1) and $ \varepsilon \times N^{ 1/6} = o(1) $. The final result is thus 
\begin{align*}
x_N \gg \sqrt{N} \qquad\Longrightarrow\qquad x_N^* - \sqrt{2N} \to +\infty
\end{align*}

This concludes the proof.
\end{proof}

\medskip
\subsubsection{Comparison of expressions}\label{SubSubSec:RMT:Painleve:Comparison}

The expression \eqref{EqTW:PainlevureNouveau}, namely
\begin{align*}
\Prob{\TW_2 \leq s} = \exp\prth{ - \Qe^2 * \id_+(s)   }
\end{align*}
can be compared with the expression \eqref{EqTW:Painleve} that reads, with $q$ given in \eqref{Def:PainlevéII}~: 
\begin{align*}
\Prob{\TW_2 \leq s} = \exp\prth{ - q^2 * \id_+(s)   }
\end{align*}

Taking the log, differentiating twice in $s$ and taking the square root as both functions are positive (see e.g. \cite[32.3.6 p. 727]{NISThandbook} for the positivity of $q$) gives the identity  
\begin{align*}
q = \Qe    
\end{align*}

Identities between different $ \sigmab_{\!a} := \sigma(\cdot ; a) $ were already noticed, for instance \cite[prop. 27]{ForresterWitteTau2and4} gives $ \sigmab_{\!2} = \frac{\sigmab_{\!0}'}{\sigmab_{\!0}} + \sigmab_{\!0} $ and \cite[(5.28)/(2.7) \& (5.27)]{ForresterWitteTau2and4} gives $ \sigmab_{\!0}' = -\qb_0^2 $ (see also \cite[\S~2.4]{ForresterWittePainleve2RMTcft}). This is thus natural to try to prove that $ q = \Qe $ with the theory furnished in \cite{ForresterWitteTau2and4}.

\medskip
\begin{shaded}
\begin{lemma}[Direct proof that $ q = \Qe $]\label{Lemma:q=Q}
Define
\begin{align*}
\sigmab_{\!a} := \sigma(\cdot \, ; a) \quad\mbox{solution of \eqref{Def:Painleve:SigmaForm:II}},   \qquad \qb_\alpha := q(\cdot ; \alpha)\quad\mbox{solution of \eqref{Def:PainlevéII}}
\end{align*}

Then, $ q = \Qe $, or, equivalently (by logarithmic differentiation and value $ 0 $ at infinity)
\begin{align}\label{Eq:q=Q:withDiff}
\frac{\qb_0'}{\qb_0} = \sigmab_{\!1} - \sigmab_{\!0} 
\end{align}
\end{lemma}
\end{shaded}
\medskip

\begin{proof}[First proof]
The link between $ \qb_\alpha $ and $ \sigmab_{\! a} $ is given by \cite[(5.9)/(5.26)]{ForresterWitteTau2and4}~: 
\begin{align*}
\alpha & = a - \tfrac{1}{2}, \\
-2^{-1/3} \sigmab_{\!a}(-2^{-1/3} t) & = \frac{1}{2} \qb_{a - 1/2}'(t)^2 - \frac{1}{2} \prth{ \qb_{a - 1/2}(t)^2 + \frac{t}{2}}^{\!\! 2} - a \qb_{a - 1/2}(t)
\end{align*}

As a result, setting $ \widehat{t} := -2^{-1/3} t $ and specialising $ a \in \ensemble{0, 1} $ yields
\begin{align*}
\begin{cases}
-2^{-1/3} \sigmab_{\!0}(\widehat{t}\,) = \frac{1}{2} \qb_{ - 1/2}'(t)^2 - \frac{1}{2} \prth{ \qb_{ - 1/2}(t)^2 + \frac{t}{2}}^{\! 2}   \\
-2^{-1/3} \sigmab_{\!1}(\widehat{t}\,) = \frac{1}{2} \qb_{1/2}'(t)^2 - \frac{1}{2} \prth{ \qb_{1/2}(t)^2 + \frac{t}{2}}^{\! 2} -   \qb_{ 1/2}(t)
\end{cases}
\end{align*}

There are moreover general identities between $ \qb_0 $ and $ \qb_{\pm 1/2} $ given in\footnote{
The second equality in \eqref{Eq:Painlevé:LinkBetweenPII} is called the \textit{inverse Gambier identity} in \cite[(2.48)]{ForresterWittePainleve2RMTcft}.
} \cite[(3.7)]{ForresterWitteTau2and4}~: 
\begin{align}\label{Eq:Painlevé:LinkBetweenPII}
\varepsilon \in \ensemble{\pm 1}, \qquad 
\begin{cases}
- \qb_0(\widehat{t}\,)^2 = \varepsilon 2^{-1/3}\prth{ \qb_{\varepsilon/2 }'(t) - \varepsilon \qb_{\varepsilon/2 }(t)^2 - \tfrac{\varepsilon}{2} t }  \\
\frac{\qb_0'(\,\widehat{t}\,\,)}{\qb_0(\,\widehat{t}\,\,)} = \varepsilon 2^{1/3} \qb_{\varepsilon/2}(t)
\end{cases}
\end{align}

This implies 
\begin{align*}
2^{-1/3}(\sigmab_{\!1} - \sigmab_{\!0})(\widehat{t}\,) & = \qb_{1/2}(t) + R
\end{align*}
with 
\begin{align*}
2R & :=  \qb_{ - 1/2}'(t)^2 -  \prth{ \qb_{ - 1/2}(t)^2 + \frac{t}{2}}^{\!\! 2}  - \prth{ \qb_{   1/2}'(t)^2 -  \prth{ \qb_{   1/2}(t)^2 + \frac{t}{2}}^{\!\! 2} } \\
               & = \prth{ \qb_{ - 1/2}'(t) - \qb_{ - 1/2}(t)^2 - \frac{t}{2} }\prth{\qb_{ - 1/2}'(t) + \qb_{ - 1/2}(t)^2 +  \frac{t}{2} } \\
               & \hspace{+4cm} - \prth{ \qb_{ 1/2}'(t) - \qb_{ 1/2}(t)^2 - \frac{t}{2} }\prth{\qb_{ 1/2}'(t) + \qb_{ 1/2}(t)^2 +  \frac{t}{2} }\\
               & = \prth{ \qb_{ - 1/2}'(t) - \qb_{ - 1/2}(t)^2 - \frac{t}{2} } \qb_0(\widehat{t})^2 2^{1/3} \\
               & \hspace{+4cm} + \qb_0(\widehat{t})^2 2^{1/3} \prth{\qb_{ 1/2}'(t) + \qb_{ 1/2}(t)^2 +  \frac{t}{2} } \qquad\mbox{using \eqref{Eq:Painlevé:LinkBetweenPII} }\\
               & = \qb_0(\widehat{t})^2 2^{1/3} \prth{ \qb_{ - 1/2}'(t) - \qb_{ - 1/2}(t)^2   + \qb_{ 1/2}'(t) + \qb_{ 1/2}(t)^2 \vphantom{\frac{t}{2}} }
\end{align*}

The second equality in \eqref{Eq:Painlevé:LinkBetweenPII} gives $ -\qb_{-1/2}(t) = \qb_{1/2}(t) $, which implies $ \qb_{ - 1/2}'(t) + \qb_{ 1/2}'(t)= 0 $ and $ - \qb_{ - 1/2}(t)^2    + \qb_{ 1/2}(t)^2 = 0 $, i.e. $ R = 0 $. The second equality in \eqref{Eq:Painlevé:LinkBetweenPII} then yields
\begin{align*}
2^{-1/3}(\sigmab_{\!1} - \sigmab_{\!0})(\widehat{t}\,) & = \qb_{1/2}(t) = 2^{-1/3} \frac{\qb_0'( \widehat{t}\, )}{\qb_0( \widehat{t}\, )}
\end{align*}
which concludes the proof.
\end{proof}


\begin{proof}[Second proof]
Using \cite[(5.9), (5.23), prop. 19 \& (3.23)]{ForresterWitteTau2and4} that read
\begin{align*}
H\crochet{n} & := H(t)\vert_{\alpha_1 = n} \\
\sigmab_{\!a}(t) & = -2^{1/3} H(\widetilde{t})_{\alpha_1 = a}, \qquad \widetilde{t} := -2^{+ 1/3} t \\
q[n] & := \qb_{n + \alpha_1 + 1/2} \\
H\crochet{n + 1} - H\crochet{n} & = q[n] \qquad\qquad\mbox{(Toda equation)}
\end{align*}
one obtains \cite[(3.11)]{ForresterWittePainleve2RMTcft} for $ a = 0 $, namely
\begin{align*}
\sigmab_{\! n + 1}(t) - \sigmab_{\!n}(t) = -2^{1/3} \qb_{n + 1/2}(\widetilde{t})
\end{align*}

The second equality in \eqref{Eq:Painlevé:LinkBetweenPII} gives then the result for $ n = 0 $. 
\end{proof}

\medskip
\subsubsection{A characterisation of the KPZ universality class with max-independence, I}\label{SubSubSec:RMT:Painleve:CaracKPZ}

Having looked at the proof of theorem~\ref{Thm:GUEmax}, one can formalise when convergence towards Tracy-Widom or Gumbel distribution occurs~:

\medskip

\begin{shaded}
\begin{theorem}[Convergence of max-independence structures]\label{Thm:CvMaxIndep}
Let $ (W_k)_{1 \leq k \leq N} $ be a sequence of independent random variables. Define $ p_k \equiv p_k(x) := \Prob{W_k > x} $ with $ x \equiv x_n := \mu_n + \sigma_n X $. Suppose moreover that
\begin{enumerate}

\item there exists $ m \equiv m_n $ such that
\begin{align}\label{Eq:ThmCvMax:Condition1}
\sum_{k = 1}^{m - 1} p_k = O(\delta_n), \qquad 
\sum_{k = 1}^{m - 1} p_k^2 = O(\delta_n'), \qquad 
\delta_n, \delta_n' \tendvers{n}{+\infty} 0
\end{align}

\medskip
\item there exists $ \gamma_n $ such that $ \frac{n - m_n}{\gamma_n} \to +\infty $ and such that the following convergence holds in $ L^1(\Rr_+) $ for $\Phi_n(\cdot, X) $ for all $ X $~:
\begin{align}\label{Eq:ThmCvMax:Condition2}
\Phi_n(c, X) := \gamma_n \Prob{W_\pe{n - \gamma_n c} >  x_n} \tendvers{n}{+\infty} \int_X^{+\infty} \phi(c + y) dy = \int_{X + c}^{+\infty} \phi =: \Phi(X + c)
\end{align}

\medskip
\item $ \Phi \in L^1([X, +\infty) ) \cap L^2([X, +\infty) ) $ for all $ X \in \Rr $ and one has the dominated convergence
\begin{align}\label{Eq:ThmCvMax:Condition3}
\norm{\Phi_n(\cdot, X) - \Phi(\cdot + X)}_{L^p(\Rr_+)} \tendvers{n}{+\infty} 0, \qquad p \in \ensemble{1, 2}
\end{align}
\end{enumerate}

Then, $ \phi \in L^1([X, +\infty), y dy) $ for all $ X $ and
\begin{align*}
\Prob{\max_{1 \leq k \leq n} W_k \leq \mu_n + \sigma_n X} \tendvers{n}{+\infty} e^{ - \phi * \id_+(X) }
\end{align*}
\end{theorem}
\end{shaded}

\begin{remark}
Note that
\begin{align*}
\phi_{\TW_2}(x) = q(x)^2, \qquad \phi_{\Gumbel(1)}(x) = e^{-x}
\end{align*}

As a result, a max-independence structure reduces the problem of transitions from Tracy-Widom to Gumbel studied e.g. in \cite{JohanssonGumbelTW} to a problem of transition in the right large deviations of the random variable $ W_N $ (see remark~\ref{Rk:LargeDeviationCondition}). 
\end{remark}

\medskip

\begin{proof}
One has with \eqref{Eq:ThmCvMax:Condition1}
\begin{align*}
P_n = \sum_{k = m}^n p_k + O(\delta_n), \qquad
P_n^{(2)} = \sum_{k = m}^n p_k^2 + O(\delta_n')
\end{align*}

Moreover, 
\begin{align*}
P_n^* & := \sum_{k = m}^n p_k 
               = \int_m^n p_\pe{u} du 
               = \int_0^{n - m} \Prob{W_\pe{n - v} >  x_n} dv
               = \gamma_n \int_0^{\frac{n - m}{\gamma_n}} \Prob{W_\pe{n - \gamma_n c} >  x_n} dc
\end{align*}

Now, the large deviation result \eqref{Eq:ThmCvMax:Condition2} (see remark~\ref{Rk:LargeDeviationCondition}) implies that
\begin{align*}
\gamma_n \Prob{W_\pe{n - \gamma_n c} >  x_n} \tendvers{n}{+\infty} \int_{X + c}^{+\infty} \phi(y) dy = \Phi(X + c)
\end{align*}
and this convergence being in $ L^1 $ by hypothesis implies by the dominated convergence \eqref{Eq:ThmCvMax:Condition3} that
\begin{align*}
P_n^* \tendvers{n}{+\infty} \int_0^{+\infty} \Phi(X + c) dc = \int_\Rr (X - t)_+ \phi(t) dt
\end{align*}

Last, the convergence in $ L^2 $ allows to have the integrability of the remainder in the Poisson approximation. This concludes the proof.
\end{proof}


\medskip
\begin{remark}\label{Rk:LargeDeviationCondition}
Suppose that
\begin{align*}
\mu_n = \eta\,n^\nu, \qquad 
\sigma_n = n^\varsigma, \qquad 
\gamma_n = o(n)
\end{align*}
and set
\begin{align*}
N := n - \gamma_n c 
\end{align*}

Then,
\begin{align*}
x_n & = \mu_n + \sigma_n X = \mu_{N + \gamma_n c} + \sigma_{N + \gamma_n c}\, X  
                  = \mu_{N + \gamma_{_{N + o(N)}} c} + \sigma_{N + \gamma_{_{N + o(N)}} c}\, X \\
                & \approx \eta (N + \gamma_N c)^\nu + \prth{ N + \gamma_N c}^\varsigma X \\
                & = \eta N^\nu\prth{ 1 + \frac{\gamma_N}{N} c}^\nu + N^\varsigma \prth{ 1 + \frac{\gamma_N}{N} c}^\varsigma X \\
                & = \eta N^\nu + \nu\eta c\, \frac{\gamma_N }{N^{1 - \nu} } + O(\gamma_N^2 N^{-(2 - \nu) } ) +  N^\varsigma X + \varsigma c X \frac{\gamma_N}{N^{1 - \varsigma}}  + O(X \gamma_N^2 N^{-(2 - \varsigma) } ) \\
                & = \eta N^\nu + N^\varsigma \prth{ X + \nu\eta c \, \frac{\gamma_N }{N^{1 - \nu +  \varsigma} } } + O\prth{ \frac{\gamma_N^2}{N^{2 - \nu}} + \frac{\gamma_N}{N^{1 - \varsigma}} }  
\end{align*}

As a result, if $ \gamma_N = N^{1 - \nu + \varsigma} /(\nu\eta) $, one ends up with
\begin{align*}
x_n = x_{N + \gamma_n c} \approx \eta N^\nu + N^\varsigma \prth{ X +  c }   
\end{align*}

This exchange of diverging parameter (using the variable $ N $ in place of $n$) explains why one gets a limiting function of $ X + c $. With $N$ in place of $n$, the condition \eqref{Eq:ThmCvMax:Condition2} becomes a large deviation condition~:
\begin{align}\label{Eq:ThmCvMax:Condition2bis}
\Phi_N(Y) := \gamma_N \Prob{W_N >  \mu_N + \sigma_N Y} \tendvers{n}{+\infty}   \Phi(Y) := \int_Y^{+\infty} \phi
\end{align}
with $ \mu_N := \eta \, N^\nu $, $ \sigma_N := \tau N^\varsigma $ (which amounts to replace $ X $ by $ X/\tau $) and $ \gamma_N = N^{1 - \nu + \varsigma} $ (with $ \tau = \nu\eta $).

The case of the $ \gue $ Tracy-Widom distribution corresponds to $ \nu = \frac{1}{2} $ and $ \varsigma = -\frac{1}{6} $, hence $ 1 - \nu - \varsigma = \frac{2}{3} $ which is exactly the fine tuning of parameters required to have the confluence property of the modified sigma-Painlev\'e function. The case of the Rider theorem \eqref{CvLaw:Ginibre:Rider} can also be treated in a similar way by modifying $ (\mu_n, \sigma_n) $ with lower order terms. 
\end{remark}

\begin{remark}
One can a priori take $ m = 1 $, i.e. $ \delta_n = \delta_n' = 0 $. The first hypothesis is thus somehow superfluous, but since all the mass is concentrated in a neighbourhood of $n$, it is not wrong to neglect this part.
\end{remark}

\medskip
\subsubsection{Ultimate remarks}


\begin{remark}
We mention the problem of large deviations for $ \lambdab_{1, N} $ which is studied e.g. in \cite{BorotNadal, MajumdarNadal}. With theorem~\ref{Thm:GUEmax}, this problem becomes a problem of large deviations for the maximum of independent random variables, hence a problem of large deviations for sums of parametric indicators. The Poisson approximation with remainder is still available in this regime~; an alternative is provided by the renormalisation flow approach developed in the context of maxima of independent random variables in \cite{AngelettiBertinAbry, BertinGyorgyi} and which is now directly applicable to the max-independent structure of $ \lambdab_{1, N} $.
\end{remark}


\medskip
\begin{remark}
The analogue of the proof of theorem~\ref{Thm:TW2max} would consist in using a decomposition of the type 
\begin{align*}
\Prob{\lambdab_{1, N}  \leq s} = \exp\prth{ -\sum_{k \geq 1} \int_s^{+\infty} Q_k^{(N)}(x) dx } = \prod_{k \geq 1} \Prob{ Z_k(\gue_N) \leq s }
\end{align*} 
i.e. an equality in law with an infinite sequence of independent random variables depending on $N$ (which would not be a consequence of the max-independence structure with only $N$ random variables). Mimicking the proof of theorem~\ref{Thm:TW2max} could then lead directly to the convergence in law of $ \frac{ Z_k(\gue_N) - 2\sqrt{N}}{N^{-1/6}} $ towards $ Z_k(\Ai) $ or some $ \We_k $, asking then the following problem~:

\begin{question}\label{Q:CvGUE2TWwithMax}
Can one obtain Tracy and Widom's convergence in law \eqref{Eq:CV:GUE2TW} writing 
\begin{align*}
\frac{\lambdab_{1, N} - 2\sqrt{N}}{N^{-1/6}} \eqlaw \max_{k \geq 1} \ensemble{ \frac{Z_k(\gue_N) - 2\sqrt{N}}{N^{-1/6}} } \cvlaw{N}{+\infty} \max_{k \geq 0} \We_k
\end{align*}
for infinite sequences of random variables $ (Z_k(\gue_N))_{k \geq 1} $ and $ (\We_k)_{k \geq 1} $~?
\end{question}

\medskip

The conjectural random variables $ \We_k $ obtained at the limit might differ from the $ Z_k(\Ai) $ defined in \eqref{Def:Law:PsiSquare:phi} or the $ Z_k'(\Ai) $ defined in \eqref{Def:Law:PsiSquare:TW2bis} as several decompositions are possible by associativity of the max operation. In addition, no commuting operator for the $ \gue $ kernel operator acting on $ L^2([s, +\infty), \mu) $ or its translated version has been found so far. We hope to come back to this problem in the future.
\end{remark}
%
%
%
%
%
%
%
%
%
%
%
%
%
%
%
%
%
%
%
%
%
%
%
%
%
%
%
%
%
%
%
%
%
%
%
%
%
%
%
%
%
%
%
%
%
%
%
%
%
%
%
%
%
%
%
%
%
%
%
%
%
%
%
%
%
%
%
%
%
%
%
%
%
%
%
%
%
%
%
%
%
%
%
%
%
%
%
%
%
%
%
%
%
%
%
%
%
%
%
%
%
%
%
%
%
%
%
%
%
%
%
%
%
%
%
%
%
%
%
%
%
%
%
%
%
%
%
%
%
%
%
%
%
%
%
%
%
%
%
%
%
%
%
%
%
%
%
%
%
%
%
%
%
%
%
%
%
%
%
%
%
%
%
%
%
%
%
%
%
%
%
%
%
%
%
%
%
%
%
%
%
%
%
%
%
%
%
%
%
%
%
%
%
%
%
%
%
%
%
%
%
%
%
%
%
%
%
%
%
%
%
%
%
%
%
%
%
%
%
%
%
%
%
%
%
%
%
%
%
%
%
%
%
%
%
%
%
%
%
%
%
%
%
%
%
%
%
%
%
%
%
%

\medskip\medskip\medskip
\section{Max-independence structure in the symmetric Schur measure}\label{Sec:SchurMeasure}

The relevant notations for this section are given in Annex~\ref{Annex:SymFunc}.

\subsection{Motivations}\label{Subsec:SchurMeasure:Motivations}
The Schur measure was introduced by Okounkov \cite{OkounkovSchurMes, OkounkovNATO} following a study of the $z$-measure \cite{OkounkovZmeas} introduced by Borodin and Olshanski \cite{BorodinOlshanskiZmeas1, BorodinOlshanskiZmeas2} to generalise the Plancherel measure~; this last measure saw a regain of activity in relation with the Ulam problem on the longuest increasing subsequence of a random uniform permutation by Baik, Deift and Johansson \cite{BaikDeiftJohansson}. Such generalisations were designed to highlight the general mechanism explaining the apparition of a determinantal structure in the Poisson-Plancherel measure and we refer to the surveys \cite{JohanssonRandomGrowthRMT, JohanssonHouches, OlshanskiSurvey} for a zoo of models that fit into this framework. The general explanation is as follows: the Schur measure of parameters $ (\Ae, \Be) $ is a measure on partitions $ \Yy_\infty := \cup_{n \geq 0} \Yy_n $ defined by
\begin{align}\label{Def:Schur:Measure}
\Pp_{\!\! \Ae, \Be}(\lambda) := H\crochet{-\Ae \Be} s_\lambda\crochet{\Ae} s_\lambda\crochet{\Be} 
\end{align}

The specialisation $ s_\lambda\crochet{\Ae} $ is written $ \varphi(s_\lambda) $ in \cite{OlshanskiSurvey} and should be understood as an ``abstract'' algebra morphism of $ \Lambdab $ where the generators $ (h_k)_k $, $ (e_k)_k $ or $ (p_k)_k $ have been attributed a particular value, the value of the Schur function resulting then from an algebraic formula expressing it with the generators, for instance the Jacobi-Trudi formula \eqref{Eq:Schur:JacobiTrudi}. This measure is a probability measure if the numbers such defined are real, for instance if $ \Ae = \overline{\Be} \subset \Cc $, or if $ \Ae $ and $ \Be $ are positive specialisations of the Schur functions (see Annex~\ref{Annex:SymFunc}).  

The main result of interest is the determinantal character (with explicit kernel) of the shifted point process $ (\lambdab_k - k + a)_{k \geq 1} $ (with $ a \in \ensemble{0, 1, \frac{1}{2} } $ depending on the conventions), $ \lambdab $ being the random partition distributed according to $ \Pp_{\!\! \Ae, \Be} $. So far, seven proofs of this determinantal nature have been given and we refer to them for this fact: fermionic Gaussian ``integral'' \cite{OkounkovSchurMes}, Riemann-Hilbert problems \cite{BorodinRHP}, Jacobi-Trudi formula \cite{RainsBDR}, ratio of alternants formula \cite{JohanssonRandomGrowthRMT}, Giambelli formula and $L$-ensembles \cite{BorodinNovikov}, coupling on Gelfand-Tsetlin graph \cite{ForresterNagaoProj} and Macdonald operator \cite{AggarwalSchur}. See also \cite{BorodinOkounkov} for a discussion on the kernel.

\medskip

The success of this generalisation prefigures the numerous extensions of the framework in many different directions, for instance the dynamical version of ``Schur processes'' \cite{OkounkovReshetikhin1, OkounkovReshetikhin2} and, in fine, the creation of the theory of integrable probability \cite{BorodinGorinSurvey, BorodinPetrovIntProb}.


\begin{remark}
We adopt here the convention of e.g. Johansson \cite{JohanssonRandomGrowthRMT} that writes this measure on partitions over $ \Zz $, as opposed to the Okounkov convention \cite{OkounkovSchurMes, OkounkovNATO} (the half-infinite wedge space/fermionic Fock space has a basis indexed by Maya diagrams written on $ \Zz + \frac{1}{2} $ by convention). Such conventions are of course equivalent, but the formuli differ by some square roots~; formuli are given here for the process $ (\lambdab_k - k)_{k \geq 1} $ and not $ (\lambdab_k - k + \frac{1}{2})_{k \geq 1} $.
\end{remark}

\medskip
\subsection{The Poisson-Plancherel measure}\label{Subsec:SchurMeasure:Plancherel}

The Plancherel measure on $ \Yy_n $ is defined by \cite{BorodinOkounkovOlshanskiPlancherel, JohanssonPlancherel} 
\begin{align}\label{Def:Plancherel:Measure}
\Pp_{n}(\lambda) :=   \frac{ d_\lambda^2 }{ n! } \Unens{\lambda \vdash n}
\end{align}
where $ d_\lambda $ is the dimension of the irreducible module of $ \Sg_n $ indexed by $ \lambda $ (see e.g. \cite[I-6]{MacDo}). This measure can be understood as the conditioning of a general measure $ \Pp_\infty $ on $ \Yy_\infty $ by the event $ \ensemble{\lambdab \vdash n} $ (where $ \lambdab \sim \Pp_\infty $), and the law of $ \abs{\lambdab} $ gives then the choice of the relevant randomisation for $ \lambdab_n \sim \Pp_n $ to obtain a determinantal structure. Since $ d_\lambda = n! \, s_\lambda\crochet{ \Eee } $ by \eqref{Eq:DimWithSchur}, setting $ \abs{\lambdab} \sim \Poisson(\xi) $ gives the Schur measure of parameters $ \Ae = \Be = \sqrt{\xi} \Eee $. We thus define the Poisson-Plancherel measure on $ \Yy_\infty $ by
\begin{align}\label{Def:PoissonPlancherel:Measure}
\Pp_{\PoPl(\xi)}(\lambda) := e^{- \xi} \xi^{\abs{\lambda}} \prth{ \frac{d_\lambda}{\abs{\lambda} !} }^2  
\end{align}

One then has
\begin{align*}
\widetilde{\lambdab} := (\lambdab_k - k)_{k \geq 1} \sim \DPP(\Ke_{ \PoPl(\xi)})
\end{align*}
implying in particular 
\begin{align*}
\Prob{ \lambdab_1^{(\xi)} < s} = \det(I - \Keb_{ \PoPl(\xi) })_{\ell^2(s + \Nn) } = \det(I - \Keb_{ \PoPl(\xi), s })_{\ell^2(\Nn) }
\end{align*}
where $ \Ke_{\!\PoPl(\xi), s}(x, y) = \Ke_{\!\PoPl(\xi)}(x + s, y + s) $. 

\medskip

The corresponding specialisation of the Schur kernel $ K_{\PoPl(\xi)} : \Zz^2 \to \Rr_+ $ in this case is the \textit{discrete Bessel kernel} \cite{BorodinOkounkovOlshanskiPlancherel, JohanssonPlancherel} that has several forms~:
\begin{itemize}

\medskip
\item the \textit{Hankel convolutive form} is \cite[(3.31)]{JohanssonPlancherel}
\begin{align}\label{Def:Kernel:HankelConv:Poisson-Plancherel}
\Ke_{\PoPl(\xi)}(x, y) := \sum_{k \geq 1} B_k^{(\xi)}(x) B_k^{(\xi)}(y), \qquad B_k^{(\xi)}(x) := J_{x + k}(2\sqrt{\xi}) 
\end{align}
where $ J_k $ is the $k$-th Bessel function of the first kind defined by~: 
\begin{align}\label{Def:SpecialFunc:Bessel1}
J_k(x) := \crochet{z^k} e^{x (z - z\inv)/2 } 
\end{align}

\medskip
\item the \textit{Christoffel-Darboux form} is (\cite[(23)]{OkounkovNATO}, \cite[(1.10)]{JohanssonPlancherel})
\begin{align}\label{Def:Kernel:Ratio:Poisson-Plancherel}
\Ke_{\PoPl(\xi)}(x, y) = \sqrt{\xi} \frac{ B_1^{(\xi)}(x ) B_0^{(\xi)}(y ) -  B_0^{(\xi)}(x ) B_1^{(\xi)}(y ) }{x - y} \Unens{x \neq y} + \sqrt{\xi} \dot{J}_{x}(2\sqrt{\xi}) \Unens{x = y}
\end{align}
with $ \dot{J}_x(t) := \frac{\partial}{\partial \alpha} J_\alpha(t)\big\vert_{\alpha = x} $ (see  e.g. \cite[10.2.4]{NISThandbook}) 

\medskip
\item and the \textit{integral form} is \cite[(22)]{OkounkovNATO}  
\begin{align}\label{Def:Kernel:Integral:Poisson-Plancherel}
\begin{aligned}
\Ke_{\PoPl(\xi)}(x, y) & = \oint_{\Uu \times \rho \Uu} \frac{z^{-x} \omega^{-y} }{1 - z \omega} e^{ \sqrt{\xi}( z - z\inv  + \omega - \omega\inv ) } \frac{d^*z}{z} \frac{d^*\omega}{\omega}, \qquad \rho < 1 \\
               & = \oint_{\Uu \times R \Uu} \frac{z^{-x} \omega^{+y} }{1 - z \omega\inv} e^{ \sqrt{\xi}( z - z\inv  -( \omega - \omega\inv) ) } \frac{d^*z}{z} \frac{d^*\omega}{\omega}, \qquad R > 1						
\end{aligned}
\end{align}

\end{itemize}

\medskip

Using the expansion $ \frac{ 1 }{1 - z \omega} = \sum_{k \geq 0} (z\omega)^k $ for $ z \in \Uu $ and $ \omega \in \rho\Uu $ with $ \rho < 1 $, one obtains \eqref{Def:Kernel:HankelConv:Poisson-Plancherel} from \eqref{Def:Kernel:Integral:Poisson-Plancherel} and \eqref{Def:SpecialFunc:Bessel1}.

The Hankel convolutive form \eqref{Def:Kernel:HankelConv:Poisson-Plancherel} is in direct analogy with the Hankel convolutive form \eqref{Def:Kernel:AiWithInt} of the Airy kernel~: $ \Keb_{\PoPl(\xi), s} $ is the square of the (discrete) Hankel convolution operator 
\begin{align*}
\Heb(B_s^{(\xi)}) : f \in \ell^2(\Nn) \mapsto \sum_{k \geq 0} B_s^{(\xi)}( \cdot + k) f(k) \in \ell^2(\Nn)
\end{align*}
namely
\begin{align}\label{Eq:SquareHankel:PoissonPlancherel}
\Keb_{\! \PoPl(\xi), s } = \Heb(B_s^{(\xi)})^2  
\end{align}
which is analogous to \eqref{Eq:SquareHankel:Airy} on $ L^2(\Rr_+) $. 


One could hope to extend the proof of theorem~\ref{Thm:TW2max} using a discrete analogue of $ \det(I - \Kb_t) = \exp\prth{ -\int_t^{+\infty} \trace( (I - \Kb_s)\inv \dot{\Kb}_s) ds   } $. Such a formula reads
\begin{align}\label{Eq:Fredholm=ExpTraceDiscrete}
\det(I - \Keb_N)_{\ell^2(\Nn)} = \exp\prth{ - \sum_{\ell \geq N} \trace\prth{ \log(I - \Keb_{\ell + 1}) - \log(I - \Keb_\ell ) }_{\ell^2(\Nn)} }
\end{align}

To simplify it, one needs the theory of Orthogonal Polynomials on the Unit Circle (OPUC).

\medskip
\subsection{OPUCs}\label{Subsec:SchurMeasure:OPUCs}

The main reference for this \S~is \cite{SimonOPUC} from which we will only present the properties that are required in the sequel. 

\medskip

A family $ (\Phi_k)_{k \geq 0} $ of monic polynomials (i.e. $ \Phi_n(z) = z^n + \cdots $) is orthogonal in $ L^2(\Uu, \mu) $ for a probability measure $ \mu $ on $ \Uu $ if 
\begin{align*}
\bracket{\Phi_k, \Phi_\ell}_{\!L^2(\mu)} := \oint_\Uu \Phi_k \overline{\Phi_\ell} \, d\mu = \norm{\Phi_k}_{L^2(\mu)}^2 \delta_{k, \ell}
\end{align*}

Defining the anti-$ L^2(\mu) $-unitary map $ ^{*, n} $ by $ P^{*, n}(z) := z^n \overline{P(\overline{z}\inv) } $ and dropping the $n$ by convention when the polynomial is of degree $n$, one can show that \cite[(2.17), (2.18)]{SimonOPUC}
\begin{align*}
\Phi_{n + 1}(z) = z \Phi_n(z) - \overline{\alpha_n} \Phi_n^*(z), \qquad \alpha_n := - \overline{\Phi_{n + 1}(0)}
\end{align*}

The coefficient $ (\alpha_k)_{k \geq 0} $ are called the \textit{Verblunsky coefficients} and the previous recursion is called the \textit{Szeg\"o recursion}. One can easily prove that \cite[(2.19)]{SimonOPUC}
\begin{align}\label{Eq:OPUC:NormWithVerblunsky}
\norm{\Phi_{n + 1}}_{L^2(\mu)}^2 = \prth{1 - \abs{\alpha_n}^2} \norm{\Phi_n }_{L^2(\mu)}^2 \qquad \Longrightarrow \qquad  \norm{\Phi_n }_{L^2(\mu)}^2 = \prod_{k = 0}^{n - 1} \prth{1 - \abs{\alpha_k}^2}
\end{align}
and that for a non trivial measure $ \mu $ (i.e. not supported on a finite set of points) $ \abs{\alpha_k} < 1 $ \cite[thm. 2.1]{SimonOPUC}. As a result, one has
\begin{align}\label{Eq:OPUC:IneqNorm}
\norm{\Phi_{n + 1}}_{L^2(\mu)} \leq \norm{\Phi_n }_{L^2(\mu)}
\end{align}

For a probability measure $ \mu $, one defines the Toeplitz determinant $ D_n(\mu) $ as
\begin{align*}
D_n(\mu) := \det\prth{ \oint_\Uu z^{i - j} d\mu(z) }_{0 \leq i, j \leq n}
\end{align*}

Hence, using linear combinations of lines and columns one gets 
\begin{align*}
D_n(\mu) := \det\prth{ \oint_\Uu \Phi_i(z) \overline{\Phi_j(z) } d\mu(z) }_{0 \leq i, j \leq n} = \prod_{j = 0}^n \norm{\Phi_j}_{L^2(\mu)}^2 = \prod_{j = 0}^{n - 1} (1 - \abs{\alpha_j}^2)^{n - j}
\end{align*}

This implies the Szeg\"o theorem \cite[(8.2)]{SimonOPUC}
\begin{align}\label{Eq:SzegoTheorem}
F(\mu) := \prod_{j \geq 0} (1 - \abs{\alpha_j}^2)  = \lim_{n \to +\infty} D_n(\mu)^{1/n}  = \lim_{n \to +\infty} \frac{D_{n + 1}(\mu)}{D_n(\mu)} = \lim_{n \to +\infty} \norm{\Phi_n}_{L^2(\mu)}^2
\end{align}

Note that $ F(\mu) $ can be equal to 0, in which case $ (\alpha_k)_k \notin \ell^2(\Nn) $ (see \cite[thm. 2.2]{SimonOPUC} for a consequence).

\medskip
\subsection{Independence structure in the Poisson-Plancherel measure}\label{Subsec:SchurMeasure:PoPlIndep}

The Cauchy formula for the Cauchy kernel \eqref{Eq:Schur:CauchyIdentity} in abstract alphabets $ \Ae, \Be $ and the (restricted) orthogonality of the Schur functions \eqref{Eq:Schur:RestrictedOrthogonality} gives
\begin{align*}
\int_{\Ue_N} H\crochet{\widehat{\omega}(\Ae U + \Be U\inv)} dU & = \int_{\Ue_N}\sum_{\lambda, \mu}  s_\lambda\crochet{\widehat{\omega}\Ae} s_\lambda(U) s_\mu\crochet{\widehat{\omega}\Be}s_\mu(U\inv) dU \\
                      & = \sum_{\ell(\lambda) \leq N}  s_\lambda\crochet{\widehat{\omega}\Ae} s_\lambda\crochet{\widehat{\omega}\Be} \\
                      & =  \sum_{ \lambda_1 \leq N}  s_\lambda\crochet{ \Ae} s_\lambda\crochet{ \Be}
\end{align*}

As a result,
\begin{align}\label{Eq:SchurMeas:BorOk}
\Proba{\Schur(\Ae, \Be)}{\lambdab_1 \leq  N} = H\crochet{-\Ae\Be} \int_{\Ue_N} H\crochet{\widehat{\omega}(\Ae U + \Be U\inv)} dU
\end{align}

Using Weyl integration formula \eqref{Eq:WeylIntegration} and the Andr\'eieff-Heine-Szeg\"o formula \eqref{Eq:AndreieffHeineSzego}, the random unitary matrix integral in the RHS can be transformed into the Toeplitz determinant of symbol $ z \mapsto H\crochet{\widehat{\omega}(z\Ae + z\inv \Be )} $ (see remark~\ref{Rk:WeylCUEToeplitz}). The determinantal character of $ \lambdab_1 $ allows moreover to write the LHS as a Fredholm determinant on $ \ell^2(N + \Nn^*) $. The formula thus obtained is the celebrated Borodin-Okounkov formula \cite{BorodinOkounkov} that was also discovered by Geronimo and Case the context of the inverse scattering method \cite{CaseGeronimo} and that meanwhile had several other proofs and generalisations \cite{BaikDeiftRains, BasorWidom, BoettcherBDR, BoettcherBorOkBis, BoettcherWidomSzJa}.  

Let us only consider the Plancherel case $ \Ae = \Be = \sqrt{\xi} \Eee $. One has $ H\crochet{ -\Ae\Be} = H\crochet{-\xi \Eee} = e^{- \xi} $, and $ H\crochet{\widehat{\omega}\Ae U} = e^{p_1\crochet{ \sqrt{\xi} U } } = e^{  \sqrt{\xi} \trace(U) } $, hence
\begin{align*}
\Proba{\PoPl(\xi)}{\lambdab_1 \leq  N} = e^{-\xi} \int_{\Ue_N} e^{\sqrt{\xi} \trace(U + U\inv) } dU = e^{-\xi} \det \Tt_N(\Ie_\xi), \qquad \Ie_\xi(z) := e^{\sqrt{\xi} (z + z\inv) }
\end{align*}

Define the probability measure on $ \Uu $
\begin{align*}
\mu_\xi(d\theta) := \frac{ e^{\sqrt{\xi} (z + z\inv) } }{ I_0(\sqrt{\xi}) } \frac{d^*z}{z}, \qquad I_0(\sqrt{\xi}) := \crochet{z^0} e^{\sqrt{\xi} (z + z\inv) }
\end{align*}
where the normalisation $ I_0(\sqrt{\xi}) $ is the $ 0 $-th modified Bessel function defined by
\begin{align*}
I_k(s) := \crochet{z^k} e^{s (z + z\inv) }
\end{align*}

We also define~:
\begin{itemize}

\item $ \prth{ \Phi_k^{(\xi)} }_{k \geq 0} $ as the OPUCs for $ L^2(\Uu, \mu_\xi) $, 

\item $ (\alpha_k(\xi))_{k \geq 0} $ as their Verblunsky coefficients.

\end{itemize}

Then, writing the Toeplitz determinant as a product of norms of OPUCs gives
\begin{align*}
\Proba{\PoPl(\xi)}{\lambdab_1 \leq N} & = e^{-\xi} \prod_{k = 0}^{N - 1} \oint_\Uu \abs{ \Phi_k^{(\xi)}(z) }^2 e^{\sqrt{\xi} (z + z\inv) } \frac{d^*z}{z} \\
                  & = e^{-\xi} \prod_{k = 0}^{N - 1} I_0(\sqrt{\xi}) \norm{ \Phi_k^{(\xi)} }_{L^2(\Uu, \mu_\xi)}^2 \\
                  & = e^{-\xi} \prod_{k = 0}^{N - 1} I_0(\sqrt{\xi}) \prod_{j = 0}^{k - 1} (1 - \abs{\alpha_j(\xi)}^2)
\end{align*}

Set
\begin{align}\label{Eq:SchurQl:InvNorm}
\begin{aligned}
\pi_k(\xi) & := I_0(\sqrt{\xi}) \norm{ \Phi_k^{(\xi)} }_{L^2(\Uu, \mu_\xi)}^2 = \oint_\Uu \abs{ \Phi_k^{(\xi)}(z) }^2 e^{\sqrt{\xi} (z + z\inv) } \frac{d^*z}{z} \\
q_k(\xi) & := \frac{1}{\pi_k(\xi)}
\end{aligned}
\end{align}

Letting $ N\to +\infty $ gives the the Szeg\"o theorem \eqref{Eq:SzegoTheorem}
\begin{align*}
1 = e^{-\xi} \prod_{k \geq  0} \pi_k(\xi)
\end{align*}

Note that when the measure $ \mu $ is not a probability measure, one needs to take into account the normalisation factor and replace the norms by the $ \pi_k $s. In this setting, the Szeg\"o theorem/the absolute convergence of the product gives
\begin{align*}
\pi_k(\xi) \tendvers{k}{+\infty} 1 \qquad\Longleftrightarrow \qquad  \norm{ \Phi_k^{(\xi)} }_{L^2(\Uu, \mu_\xi)}^2 \tendvers{k}{+\infty} I_0(\sqrt{\xi})\inv = \prod_{j \geq 0} (1 - \abs{\alpha_j(\xi)}^2)
\end{align*}

As a consequence of the previous computations, 
\begin{align*}
\Proba{\PoPl(\xi)}{\lambdab_1 \leq N} = \prod_{k \geq N} \pi_k(\xi)\inv =: \prod_{k \geq N} q_k(\xi) = \prod_{\ell \geq 0} q_{\ell + N}(\xi)
\end{align*}



We are now ready to state the

\begin{shaded}
\begin{theorem}[$ \lambdab_1^{\PoPl(\xi)} $ is a maximum of a sequence of independent random variables]\label{Thm:PoissonPlancherel}
There exist i.i.d. non negative random variables $ (Z_k^{\PoPl(\xi)})_{k \geq 0} $ such that
\begin{align}\label{EqMax:PoissonPlancherel}
\boxed{\lambdab_1^{\PoPl(\xi)} \eqlaw \max_{k \geq 0}\! \ensemble{  Z_k^{\PoPl(\xi)} - k }  }  
\end{align}

Moreover  
\begin{align}\label{Eq:DescriptionPoissonPlancherelZk}
Z_1^{\PoPl(\xi)} \eqlaw \inf\!\ensemble{j \geq 0 \, /\, \forall m \geq j, \ B_j(\xi) = 0}
\end{align}
where the $ (B_j(\xi))_{j \geq 0} $ are independent $ \ensemble{0, 1} $-Bernoulli random variables of expectation $ \abs{\alpha_j(\xi) }^2   $.
\end{theorem}
\end{shaded}


\begin{proof}
One thus has
\begin{align*}
\Proba{\PoPl(\xi)}{\lambdab_1 \leq N} = \prod_{\ell \geq 0} q_{\ell + N}(\xi)
\end{align*}
with 
\begin{align*}
0 \leq q_\ell (\xi) = \frac{1}{I_0(\sqrt{\xi})} \norm{ \Phi_\ell^{(\xi)} }_{L^2(\Uu, \mu_\xi)}^{-2} \tendvers{\ell}{+\infty} 1
\end{align*}

Moreover, since norms of OPUCs are decreasing by \eqref{Eq:OPUC:IneqNorm}, $ (q_\ell(\xi))_\ell $ is decreasing. This can also be seen using \eqref{Eq:OPUC:NormWithVerblunsky}, noticing that $ 1 - \abs{\alpha_k(\xi)}^2 < 1 $ and writing
\begin{align*}
q_\ell (\xi) & = \frac{1}{I_0(\sqrt{\xi})} \norm{ \Phi_\ell^{(\xi)} }_{L^2(\Uu, \mu_\xi)}^{-2} = \prod_{k \geq 0} (1 - \abs{\alpha_k(\xi)}^2) \times \prod_{j = 0}^{\ell - 1} \frac{1}{1 - \abs{\alpha_j(\xi) }^2} \\
               & = \prod_{k \geq \ell} (1 - \abs{\alpha_k(\xi)}^2)
\end{align*}

As a result, there exists a random variable $ Z^{\PoPl(\xi)} \in \Nn $ such that 
\begin{align*}
q_\ell (\xi) = \Prob{ Z^{\PoPl(\xi)}  \leq \ell }
\end{align*}
and in particular, with an i.i.d. sequence $ (Z^{\PoPl(\xi)}_\ell)_\ell $ of such random variables
\begin{align*}
\Proba{\PoPl(\xi)}{\lambdab_1 \leq N} & = \prod_{\ell \geq 0}  \Prob{ Z_\ell^{\PoPl(\xi)}  \leq \ell + N } = \prod_{\ell \geq 0}  \Prob{ Z_\ell^{\PoPl(\xi)} - \ell  \leq   N }\\
                     & = \Prob{ \max_{\ell \geq 0} \!\ensemble{ Z_\ell^{\PoPl(\xi)} - \ell } \leq N }
\end{align*}

Last, \eqref{Eq:OPUC:NormWithVerblunsky} yields
\begin{align*}
\Prob{ Z_1^{\PoPl(\xi)} = \ell } & = q_\ell(\xi) - q_{\ell - 1}(\xi) = \frac{1}{I_0(\sqrt{\xi}) } \prth{ \frac{1}{\norm{\Phi_\ell^{(\xi) } }^2} - \frac{1}{ \norm{\Phi_{\ell - 1}^{(\xi) } }^2 } } \\
                & = \frac{1}{I_0(\sqrt{\xi})}\times \frac{1}{\norm{\Phi_{\ell - 1}^{ (\xi) } }^2 } \times \prth{ \frac{1}{1 - \abs{\alpha_{\ell - 1}(\xi) }^2 } - 1 } \\ 
                & = \prod_{k \geq 0} (1 - \abs{\alpha_k(\xi)}^2) \times \prod_{j = 0}^{\ell - 2} \frac{1}{1 - \abs{\alpha_j(\xi) }^2} \times \frac{ \abs{\alpha_{\ell - 1}(\xi) }^2 }{1 - \abs{\alpha_{\ell - 1}(\xi)}^2 } \\
                & =: \abs{\alpha_{\ell - 1}(\xi) }^2  \prod_{k \geq \ell} (1 - \abs{\alpha_k(\xi)}^2) \\ 
                & =  \Prob{ B_{\ell - 1}(\xi) = 1} \prod_{k \geq \ell} \Prob{ B_k(\xi) = 0 }  \\
                & = \Prob{ \taub_{\! \partial}(\xi) = \ell } 
\end{align*}
where
\begin{align*}
\taub_{\!\partial}(\xi) := \inf\!\ensemble{ j \geq 0 \, /\, B_m(\xi) = 0 \ \forall m \geq j}
\end{align*}
is the entrance time to the ``cemetary'' $ \partial $ for the ``chain'' $ (B_k(\xi))_k $ (here an independent sequence).
\end{proof}

\medskip
\subsection{The general symmetric Schur measure}\label{Subsec:SchurMeasure:General}

We now use \eqref{Eq:SchurMeas:BorOk} in the case where $ \Ae = \overline{\Be} $ if $ \Ae, \Be \subset \Cc $, i.e. $ g_k\crochet{\Ae} = \overline{g_k\crochet{\Be}} $ for $ g_k = p_k $, $ e_k $ or $ h_k $. One has in particular $ H\crochet{\Ae\overline{\Ae}} = \exp\prth{ \sum_{k \geq 1} \frac{1}{k} \abs{p_k\crochet{\Ae}}^2 } $ and $ H\crochet{  \widehat{\omega} (z\Ae  + z\inv\overline{\Ae} ) } = \abs{H\crochet{ z \widehat{\omega} \Ae } }^2 $ for $ z \in \Uu $.

We define the probability measure on $ \Uu $
\begin{align*}
\mu_\Ae(d\theta) := \frac{ \abs{H\crochet{ z \widehat{\omega} \Ae } }^2 }{ \Ze(\Ae) } \frac{d^*z}{z}, \qquad \Ze(\Ae) := \crochet{z^0} \abs{H\crochet{ z \widehat{\omega} \Ae } }^2
\end{align*}
and we define the OPUCs for $ L^2(\Uu, \mu_\Ae) $ as $ \prth{ \Phi_k^{(\Ae)} }_{k \geq 0} $ and their Verblunsky coefficients as $ (\alpha_k(\Ae))_{k \geq 0} $. As a result, 
\begin{align*}
\Proba{\Schur(\Ae, \overline{\Ae})}{\lambdab_1 \leq N}  =  H\crochet{- \Ae\overline{\Ae}} \prod_{k = 0}^{N - 1} \Ze(\Ae) \prod_{j = 0}^{k - 1} (1 - \abs{\alpha_j(\Ae)}^2)
\end{align*}

We moreover suppose that the specialisation $ \Ae $ is well-defined in the sense that $ 0 < H\crochet{\Ae\overline{\Ae}}  < \infty $. The Szeg\"o theorem is then 
\begin{align*}
H\crochet{\Ae\overline{\Ae}} = \prod_{k \geq 0} \Ze(\Ae) \prod_{j = 0}^{k - 1} (1 - \abs{\alpha_j(\Ae)}^2)
\end{align*}
and the absolute convergence of the product implies that 
\begin{align*}
\Ze(\Ae) = \prod_{j \geq 0}  (1 - \abs{\alpha_j(\Ae)}^2)\inv
\end{align*}

Thus, 
\begin{align*}
\Proba{\Schur(\Ae, \overline{\Ae})}{\lambdab_1 \leq N} & =   \prod_{k \geq N} \Ze(\Ae)\inv \prod_{j = 0}^{k - 1} (1 - \abs{\alpha_j(\Ae)}^2)\inv \\
                & =: \prod_{k \geq N} q_k(\Ae) = \prod_{\ell \geq 0} q_{\ell + N}(\Ae)
\end{align*}
where
\begin{align*}
q_k(\Ae) := \Ze(\Ae)\inv \prod_{j = 0}^{k - 1} (1 - \abs{\alpha_j(\Ae)}^2)\inv = \prod_{j \geq k}  (1 - \abs{\alpha_j(\Ae)}^2)  
\end{align*}
defines a random variable $ Z^{\Schur(\Ae, \overline{\Ae})} $ such that
\begin{align*}
q_k(\Ae) = \Prob{Z^{\Schur(\Ae, \overline{\Ae})} \leq k } \quad \Longrightarrow\quad \Prob{Z^{\Schur(\Ae, \overline{\Ae})} = k } = \abs{\alpha_{k - 1}(\Ae)}^2 \prod_{j \geq k} (1 - \abs{\alpha_j(\Ae)}^2)
\end{align*}

As a result, one has the

\begin{shaded}
\begin{theorem}[$ \lambdab_1^{\Schur(\Ae, \overline{\Ae})} $ is a maximum of a sequence of independent random variables]\label{Thm:SchurMeas}
For a sequence $ (Z^{\Schur(\Ae, \overline{\Ae})}_k)_{k \geq 0} $ of i.i.d. random variables equal in law to $ Z^{\Schur(\Ae, \overline{\Ae})} $, one has
\begin{align}\label{EqMax:SchurMeas}
\boxed{\lambdab_1^{\Schur(\Ae, \overline{\Ae})} \eqlaw \max_{k \geq 0}\! \ensemble{  Z_k^{\Schur(\Ae, \overline{\Ae})} - k }  }  
\end{align}

Moreover  
\begin{align}\label{Eq:DescriptionSchurMeasZk}
Z_1^{\Schur(\Ae, \overline{\Ae})} \eqlaw \inf\!\ensemble{j \geq 0 \, /\, \forall m \geq j, \ B_j(\Ae) = 0}
\end{align}
where the $ (B_j(\Ae))_{j \geq 0} $ are independent $ \ensemble{0, 1} $-Bernoulli random variables of expectation $ \abs{\alpha_j(\Ae) }^2   $.
\end{theorem}
\end{shaded}


\medskip
\subsection{Remark on the discrete Fuchs lemma}\label{SubSec:SchurMeasure:Fuchs}

In the vein of \S~\ref{SubSec:TW:SecondApproach}, we have the~:

\begin{shaded}
\begin{lemma}[Discrete unnormalised and varying interval Fuchs lemma]\label{Lemma:Fuchs:Discrete}
$ $

$ \bullet $ For $ I \subset \Nn $, let $ \Keb_\ell : \ell^2(I) \to \ell^2(I) $ be self-adjoint. Let $ (g_\ell, \lambda_\ell) $ be such that $ \Keb_\ell g_\ell = \lambda_\ell g_\ell $ for all $ \ell \in \Nn $. Define $ \Delta u_k := u_{k + 1} - u_k $ and $ \nabla u_k := u_k - u_{k - 1} $. Then,
\begin{align*}
\Delta\lambda_\ell = \frac{ \bracket{ \Delta\Keb_\ell g_{\ell + 1}, g_\ell}_{\! \ell^2(I)} }{ \bracket{ g_{\ell + 1}, g_\ell}_{\! \ell^2(I)} }
\end{align*}

\medskip
$ \bullet $ For $ I_\ell := \ell + \Nn := \{\ell, \ell + 1, \dots, \} $, let $ \Keb : \ell^2(I_\ell) \to \ell^2(I_\ell) $ be self-adjoint with a kernel independent of $ \ell $. Let $ (g_\ell, \lambda_\ell) $ be such that $ \Keb_\ell g_\ell = \lambda_\ell g_\ell $ (with an index to denote the action on $ \ell^2(I_\ell) $). Then,
\begin{align*}
\lambda_{\ell + 1} = \lambda_\ell \frac{ \bracket{g_\ell, g_{\ell + 1} }_{\ell^2(I_{\ell + 1}) } }{\bracket{g_\ell, g_{\ell + 1} }_{\ell^2(I_\ell) } }  
\end{align*}
\end{lemma}
\end{shaded}


\begin{proof}
A discrete differentiation of $ \Keb_\ell g_\ell = \lambda_\ell g_\ell $ yields to
\begin{align*}
\Delta\Keb_\ell g_{\ell + 1} + \Keb_\ell \Delta g_\ell = \Delta\lambda_\ell g_{\ell + 1} + \lambda_\ell \Delta g_\ell
\qquad\Longleftrightarrow\qquad
(\Keb_\ell - \lambda_\ell I) \Delta g_\ell = - \Delta(\Keb_\ell - \lambda_\ell I) g_{\ell + 1}
\end{align*}

Since $ \Keb_\ell^* = \Keb_\ell $ and $ (\Keb_\ell - \lambda_\ell I) \Delta g_\ell \in \im(\Keb_\ell - \lambda_\ell I) $, one has $ (\Keb_\ell - \lambda_\ell I) \Delta g_\ell \perp \ker(\Keb_\ell - \lambda_\ell I) = \Rr g_\ell $. In the first case, one takes the scalar product with $ g_\ell $ to obtain the result. 

In the second case, one has
\begin{align*}
(\Delta \Keb_\ell) g_{\ell + 1} & = \sum_{r \geq \ell + 1} \Ke(\cdot, r) g_{\ell + 1}(r) - \sum_{r \geq \ell } \Ke(\cdot, r) g_{\ell + 1}(r) \\
                  & = - \Ke(\cdot, \ell) g_{\ell + 1}(\ell)
\end{align*}

Hence, with $ \bracket{\cdot, \cdot}_\ell := \bracket{\cdot, \cdot}_{\ell^2(I_\ell)} $
\begin{align*}
\bracket{ (\Delta \Keb_\ell) g_{\ell + 1} , g_\ell}_\ell 
                  & = - \bracket{\Ke(\cdot, \ell) g_{\ell + 1}(\ell) , g_\ell}_\ell \\
                  & = - \bracket{\Ke(\ell, \cdot) , g_\ell}_\ell g_{\ell + 1}(\ell) \\
                  & = - \Keb_\ell g_\ell(\ell)  g_{\ell + 1}(\ell) \\
                  & = - \lambda_\ell g_\ell(\ell) g_{\ell + 1}(\ell)
\end{align*}

Finally
\begin{align*}
\Delta\lambda_\ell & = \frac{ \bracket{ \Delta\Keb_\ell g_{\ell + 1}, g_\ell}_\ell }{ \bracket{ g_{\ell + 1}, g_\ell}_\ell } \\
                  & = - \lambda_\ell  \frac{ g_\ell(\ell) g_{\ell + 1}(\ell) }{ \bracket{ g_{\ell + 1}, g_\ell}_\ell }  
\end{align*}
namely
\begin{align*}
\lambda_{\ell + 1}  & = \lambda_\ell \prth{ 1 -  \frac{ g_\ell(\ell) g_{\ell + 1}(\ell) }{ \bracket{ g_{\ell + 1}, g_\ell}_\ell }  } \\
                    & = \lambda_\ell \frac{ \bracket{ g_{\ell + 1}, g_\ell}_\blue{\ell + 1}  }{\bracket{ g_{\ell + 1}, g_\ell}_\ell } 
\end{align*}
which gives the result.
\end{proof}

\medskip

Define $ \psi^* = \bracket{\psi, \cdot}_{\ell^2(I)} $ for the relevant choice of $ I $ (depending on which case we are concerned with).

In the case of the Poisson-Plancherel measure, a commuting difference operator was introduced in \cite[prop. 2.12]{BorodinOkounkovOlshanskiPlancherel}. It reads on $ \ell^2(\Nn) $ 
\begin{align*}
\De_\xi := \Delta\Mg_\id\nabla - \Mg_{\beta_s}, \qquad \beta_s(k) := \frac{(k + s)(k + 1 - 2\sqrt{\xi})}{\sqrt{\xi}}
\end{align*}

Moreover, for $ k = \pe{2\sqrt{\xi} + \xi^{1/6} x} $ and $ s = \pe{2\sqrt{\xi} + \xi^{1/6} S} $, it converges in the weak sense to $ \Le_{TW, S} := DXD - X(X + S) $.

By \eqref{Def:Kernel:HankelConv:Poisson-Plancherel}, one has $ \Delta\Keb_{\PoPl(\xi), \ell} = B_\ell^{(\xi)}\otimes (B_\ell^{(\xi)})^* $ acting on $ \ell^2(\Nn) $, and
\begin{align*}
\Delta\lambda_\ell = \frac{ \bracket{   g_{\ell + 1}, B_\ell^{(\xi)}}_{\! \ell^2(\Nn)} 
                            \bracket{    g_\ell, B_\ell^{(\xi)} }_{\! \ell^2(\Nn)} }{ \bracket{ g_{\ell + 1}, g_\ell}_{\! \ell^2(\Nn)} }
\end{align*}
which is not known to be non negative. The same can be said about the second expression in lemma~\ref{Lemma:Fuchs:Discrete} in case of $ \Keb_{\PoPl(\xi) } $ acting on $ \ell^2(\ell + \Nn) $.

\medskip
\subsection{A characterisation of the KPZ universality class with max-independence, II}\label{SubSec:SchurMeasure:CaracKPZ}  

Similarly to \S~\ref{SubSubSec:RMT:Painleve:CaracKPZ}, one can formalise when a convergence towards the Tracy-Widom or the Gumbel distribution occurs in the case of a max-independence structure obtained as a randomisation of the sequence of integers~:

\medskip

\begin{shaded}
\begin{theorem}[Convergence of Schur max-independence structures]\label{Thm:CvMaxRandomisedIntegers}
Let $ (Z_k(\xi))_{k \geq 1} $ be a sequence of i.i.d. random variables with a diverging parameter $ \xi \to +\infty $. Define $ p_k \equiv p_k(x_\xi) := \Prob{ Z(\xi) - k > x} =: \overline{F}_\xi(x + k) $ with $ x \equiv x_\xi := \pe{m_\xi + \sigma_\xi X } $. Suppose moreover that one has~:
\begin{enumerate}

\item either the local CLT
\begin{align}\label{Eq:ThmCvMaxSchur:LocalTCL}
\sigma_\xi \Prob{ Z(\xi) = \pe{m_\xi + \sigma_\xi X} } \tendvers{\xi}{+\infty} \phi(x)
\end{align}

\medskip
\item or (equivalently) the right LDP
\begin{align}\label{Eq:ThmCvMaxSchur:PGD}
\sigma_\xi \Prob{ Z(\xi) \geq  m_\xi + \sigma_\xi X  } \tendvers{\xi}{+\infty} \Phi(X) := \int_X^{+\infty} \phi 
\end{align}

\end{enumerate}

Suppose moreover that the previous convergences occur, as functions of $X$, in $ L^1([a, +\infty)) $ for all $a \in \Rr$ and that $ \id \cdot \phi^p \in L^1([a, +\infty)) $ for $ p \in \ensemble{1, 2} $. 

\medskip

Then,
\begin{align*}
\Prob{\max_{k \geq 0} \ensemble{Z_k(\xi) - k} \leq m_\xi + \sigma_\xi X} \tendvers{n}{+\infty} e^{ - \phi * \id_+(X) }
\end{align*}
\end{theorem}
\end{shaded}

\medskip

\begin{proof}
The Poisson approximation gives 
\begin{align*}
\Prob{ \max_{k \geq 0} \ensemble{Z_k(\xi) - k} \leq x} = \exp\prth{ - \sum_{k \geq 0} p_k } + O\prth{ \sum_{k \geq 0} p_k^2 }
\end{align*}

One has on the one hand
\begin{align*}
\sum_{k \geq 0} p_k  & = \sum_{\ell \geq x} \overline{F}_{\!\xi}(\ell) = \int_x^{+\infty}  \overline{F}_{\!\xi}(t) dt
\end{align*}
as one defines
\begin{align*}
\overline{F}_{\!\xi}(\ell) := \Prob{Z(\xi) > \ell } = \Prob{Z(\xi) > \pe{\ell} } = \overline{F}_{\!\xi}(\pe{\ell})
\end{align*}
for all $ \ell \in \Rr $ (equivalently, one can also take $ \Prob{Z(\xi) \geq \pe{\ell} } $, which does not change anything to the asymptotic analysis). 

On the other hand, 
\begin{align*}
\sum_{k \geq 0} p_k  & = \sum_{k \geq 0}  \sum_{\ell \geq k + x} p_\xi(\ell) , \qquad p_\xi(\ell) := \Prob{ Z(\xi) = \ell }  \\ 
               & = \sum_{ \ell \in \Zz } (\ell - x)_+ p_\xi(\ell) \\
               & = \int_\Rr (y - x)_+ p_\xi(\pe{y}) dy
\end{align*}
which shows that with this particular max-independence structure, the Poisson approximation already gives the final form.

Changing variable $ t = m_\xi + \sigma_\xi T $ and $ y = m_\xi + \sigma_\xi Y $ then gives
\begin{align*}
\sum_{k \geq 0} p_k  & = \int_X^{+\infty}  \sigma_\xi \overline{F}_{\!\xi}(m_\xi + \sigma_\xi T) dT \hspace{+1.4cm} 
                 \tendvers{\xi}{+\infty} \int_X^{+\infty} \Phi(T) dT \\
                     & = \int_\Rr (Y - X)_+ \sigma_\xi p_\xi(\pe{m_\xi + \sigma_\xi Y}) dY
                 \tendvers{\xi}{+\infty} \int_\Rr (Y - X)_+ \phi(Y) dY
\end{align*}
using respectively \eqref{Eq:ThmCvMaxSchur:PGD} and \eqref{Eq:ThmCvMaxSchur:LocalTCL}. Last, the same convergence but in $ L^2 $ allows to control the speed of convergence in the Poisson approximation into the form
\begin{align*}
\sigma_\xi \sum_{k \geq 0} p_k^2   & = \int_\Rr (Y - X)_+ \prth{ \sigma_\xi p_\xi(\pe{m_\xi + \sigma_\xi Y}) }^2 dY \\
                 & \tendvers{\xi}{+\infty} \int_\Rr (Y - X)_+ \phi(Y)^2 dY
\end{align*}
i.e. $ \sum_{k \geq 0} p_k^2  = O(\sigma_\xi\inv) $. This concludes the proof.
\end{proof}

\medskip
\subsection{Discussion}\label{SubSec:SchurMeasure:Discussion}

The equation \eqref{Eq:Fredholm=ExpTraceDiscrete} reads
\begin{align*}
\det(I - \Keb_N)_{\ell^2(\Nn)} & = \exp\prth{ - \sum_{s \geq N} \trace\prth{ \log(I - \Keb_{s + 1}) - \log(I - \Keb_s ) }_{\ell^2(\Nn)} } \\
                  & =  \exp\prth{ - \sum_{s \geq N} \trace\prth{ \log(I - (I - \Keb_s )\inv(\Keb_{s + 1} - \Keb_s)  )  }_{\ell^2(\Nn)} }
\end{align*}
the operator $ I - \Keb_s $ being invertible as $ \det(I - \Keb_s)_{\ell^2(\Nn)} \neq 0 $.


In the case of Poisson-Plancherel, the Hankel convolutive form \eqref{Def:Kernel:HankelConv:Poisson-Plancherel}/\eqref{Eq:SquareHankel:PoissonPlancherel} shows that
\begin{align*}
\Keb_N^{(\xi)} = \sum_{k \geq N} B_k^{(\xi)} \otimes (B_k^{(\xi)})^* \qquad \Longrightarrow \qquad \Keb_{s + 1}^{(\xi)} - \Keb_s^{(\xi)}  = - B_s^{(\xi)} \otimes (B_s^{(\xi)})^*
\end{align*}

This last operator is of rank one with non zero trace, which implies 
\begin{align*}
\det\prth{ I - \Keb_N^{(\xi)} }_{\ell^2(\Nn)} & = \exp\prth{ - \sum_{s \geq N} \trace\prth{ \log(I + (I - \Keb_s^{(\xi)} )\inv B_s^{(\xi)} \otimes (B_s^{(\xi)})^*  )  }_{\ell^2(\Nn)} }\\
                  & =  \exp\prth{ - \sum_{s \geq N} \log\prth{ 1 + \bracket{ (I - \Keb_s^{(\xi)} )\inv B_s^{(\xi)} , B_s^{(\xi)}   }_{\!\! \ell^2(\Nn)} } } \\
                  & = \prod_{s \geq N} \frac{1}{1 + \bracket{ (I - \Keb_s^{(\xi)} )\inv B_s^{(\xi)} , B_s^{(\xi)} }_{\!\! \ell^2(\Nn)} }
\end{align*}

Since $ \det\prth{ I - \Keb_N^{(\xi)} }_{\ell^2(\Nn)} = \prod_{s \geq N} q_s(\xi) $, one gets 
\begin{align}\label{Eq:SchurQl:1over}
q_s(\xi) = \frac{1}{ 1 + \bracket{ (I - \Keb_s^{(\xi)} )\inv B_s^{(\xi)} , B_s^{(\xi)} }_{\!\! \ell^2(\Nn)} } 
\end{align}

\medskip

One can also write \eqref{Eq:Fredholm=ExpTraceDiscrete} as
\begin{align*}
\det(I - \Keb_N)_{\ell^2(\Nn)} & = \exp\prth{ + \sum_{s \geq N} \trace\prth{ \log(I + (I - \Keb_{s + 1} )\inv(\Keb_{s + 1} - \Keb_s)  )  }_{\ell^2(\Nn)} } \\
                  & =  \exp\prth{ + \sum_{s \geq N} \trace\prth{ \log(I - (I - \Keb_{s + 1} )\inv B_s^{(\xi)} \otimes (B_s^{(\xi)})^*  )  }_{\ell^2(\Nn)} } \\
                  & = \prod_{s \geq N} \prth{1 - \bracket{ (I - \Keb_{s + 1}^{(\xi)} )\inv B_s^{(\xi)} , B_s^{(\xi)} }_{\!\! \ell^2(\Nn)} }
\end{align*}
implying also that
\begin{align}\label{Eq:SchurQl:1minus}
q_s(\xi) = 1 - \bracket{ (I - \Keb_{s + 1}^{(\xi)} )\inv B_s^{(\xi)} , B_s^{(\xi)} }_{\!\! \ell^2(\Nn)} 
\end{align}

Several expressions for $ q_s(\xi) = \Prob{Z^{\PoPl(\xi)} \leq s } $ are thus given by \eqref{Eq:SchurQl:InvNorm}, \eqref{Eq:SchurQl:1over} and \eqref{Eq:SchurQl:1minus}. To these ones, one can add the resolvant representation of Borodin \cite[(4.1), (4.3)]{BorodinDiscretePainleve} that reads
\begin{align}\label{Eq:SchurQl:BorodinPainleve}
q_s(\xi) =  \frac{1}{\Ree_s[s, s]}  , \qquad \Reb_s := \Keb_s(I - \Keb_s)\inv = I - (I - \Keb_s)\inv 
\end{align}

Using the fact that $ \vert\!\vert B_s^{(\xi)} \vert\!\vert^2_{ \ell^2(\Nn)} = 1 $, we see that this representation is equivalent to \eqref{Eq:SchurQl:1over}~; but \cite[(4.3)]{BorodinDiscretePainleve} gives another expression of $ \Ree_s[s, s] $ using the theory of discrete Riemann-Hilbert problems for discrete integrable operators constructed in \cite{BorodinDRHP}. In particular, it gives a discrete Painlev\'e equation analogous to the continuous one given in \S~\ref{SubSec:RMT:Painleve}. 

Such an expression can also be obtained by applying Cramer's rule for the inverse of a matrix and a block decomposition formula as in \cite[first proof \& (2)]{BasorWidom}. Notice that the function $ \He_N $ defined in \eqref{Def:Law:PsiSquare:GUE} is also defined by squared norms of Orthogonal Polynomials on the Real Line (OPRL) and that the $ q_s(\xi) $ are related with squared norms of OPUCs. The two theories are thus analogous and one can ask about similar convergence issues.


\begin{remark}\label{Rk:SeveralRandomisationStructures}
The analogue of \eqref{Def:Law:PsiSquare:phi} i.e.
\begin{align*}
\Prob{Z_\ell(\phi) \leq t} =  \exp\prth{ -\int_t^{+\infty} \norm{  \Hb(\phi_s)^\ell \phi_s  }_{ L^2(\Rr_+) }^2  ds }
\end{align*}
is obtained by writing
\begin{align*}
\det(I - \Keb_N)_{\ell^2(\Nn)} & = \exp\prth{  \sum_{\ell \geq 1} \frac{1}{\ell} \sum_{s \geq N} \trace\prth{  ( (I - \Keb_s )\inv(\Keb_{s + 1} - \Keb_s)  )^\ell  }_{\ell^2(\Nn)} } \\
                  & =  \prod_{\ell \geq 1} \exp\prth{    \frac{1}{\ell} \sum_{s \geq N} \trace\prth{  \big( (I - \Keb_s )\inv( - B_s^{(\xi)} \otimes (B_s^{(\xi)})^* \big)^\ell  }_{\ell^2(\Nn)} } \\
                  & = \prod_{\ell \geq 1} \exp\prth{  -  \frac{1}{\ell} \sum_{s \geq N} \bracket{    (I - \Keb_s )\inv B_s^{(\xi)} , B_s^{(\xi)}  }_{\ell^2(\Nn)}^\ell }  \\
                  & =: \prod_{\ell \geq 1} \Prob{ \Zb_{\! \ell}^{(\xi)} \leq N }
\end{align*}

This is thus a different edge randomisation structure~! The same can be said about the $ GU\!E_N $. Several such structures are thus possible.
\end{remark}

The relevant convergence here was originally proven by Baik-Deift-Johansson \cite{BaikDeiftJohansson} and later reproven by Johansson \cite{JohanssonPlancherel}, Okounkov \cite{OkounkovBDJconj} and Borodin-Okounkov-Olshanski \cite{BorodinOkounkovOlshanskiPlancherel}. It reads
\begin{align}\label{CvLaw:Plancherel:BDJ}
\frac{\lambdab_1^{\PoPl(\xi)} - 2\sqrt{\xi}}{\xi^{1/6}} \cvlaw{\xi}{+\infty} \TW_2
\end{align}

With $ \TW_2 $ given by \eqref{EqMax:TW2} and in the same vein as question~\ref{Q:CvGUE2TWwithMax}, one can ask the~:
\begin{question}\label{Q:CvBDJ2TWwithMax}
Can one prove the Baik-Deift-Johansson convergence \eqref{CvLaw:Plancherel:BDJ} in the same vein as in theorem~\ref{Thm:NewExprTW2}, using the convergence of a certain range $ k \equiv k(\xi) $ of the random variables
\begin{align*}
\frac{Z^{\PoPl(\xi)} - k(\xi) - 2\sqrt{\xi}}{\xi^{1/6}}  
\end{align*}

Same question for the $ \Zb_{\! k}^{(\xi)} $ of remark~\ref{Rk:SeveralRandomisationStructures}.
\end{question}

$ $

\medskip\medskip
\appendix
\section{Symmetric functions}\label{Annex:SymFunc}

\subsection{Generalities}\label{SubSec:AnnexSymFunc:General}
\subsubsection{Series}\label{SubSubSec:AnnexSymFunc:General:Series}

We denote by $ \Uu := \ensemble{z \in \Cc : \abs{z} = 1}$ the unit circle and set
\begin{align*}
d^*z := \frac{dz}{2i \pi}, \qquad \frac{d^*Z}{Z} := \prod_{k = 1}^n \frac{d^*z_k}{  z_k } \quad \text{ if } \quad Z := (z_1, \dots, z_n) 
\end{align*}

We will use the notation 
\begin{align*}
\crochet{z^n} f(z) = a_n \quad \mbox{if} \quad f(z) = \sum_n a_n z^n
\end{align*}
to denote the $ n $-th (Fourier) coefficient of a Laurent series $ f $. Up to a dilation, one can suppose that $f$ has a radius of convergence at least equal to 1, in which case
\begin{align}\label{Def:FourierCoeff}
\crochet{z^n} f(z) = \oint_\Uu f(z) z^{-n} \frac{d^*z }{z} 
\end{align}

\subsubsection{Multi-index notation}\label{SubSubSec:AnnexSymFunc:General:MultiIndex}

We use the multi-index notation 
\begin{align}\label{Def:MultiIndexNotation}
X^\alpha := \prod_{k \geq 1} x_k^{\alpha_k}
\end{align}
and subsequently the notation $ \crochet{X^\alpha} f(X) $ for the multivariate Fourier coefficient. If $ N \in \Zz $ (in particular for $ N = -1 $) and $ X = (x_1, \dots, x_k) $, we set
\begin{align*}
X^N := x_1^N x_2^N \dots  x_k^N, \qquad X\inv := x_1\inv x_2\inv \dots  x_k\inv
\end{align*}

\medskip
\subsubsection{Partitions}\label{SubSubSec:AnnexSymFunc:General:Partitions}

The partitions of an integer $n$ are weakly decreasing sequences $ \lambda := (\lambda_1 \geq \lambda_2 \geq \cdots \geq \lambda_m) $ such that $ \abs{\lambda} := \sum_{i = 1}^m \lambda_i = n $. We will write $ \lambda \vdash n $ for $ \abs{\lambda} = n $. The length $m$ of $ \lambda $ will be denoted by $ \ell(\lambda) $. The transpose of a partition will be denoted by $ \lambda' $. We refer to \cite[ch. I.1]{MacDo} for generalities concerning them.

\medskip
\subsubsection{Symmetric functions and plethysm}\label{SubSubSec:AnnexSymFunc:General:SymFunc}

For notations an definitions concerning symmetric functions, the reference is Macdonald \cite{MacDo}. We denote by $  s_\lambda $ the Schur functions, $ p_k $ the power functions, $ h_k $ the homogeneous complete symmetric functions and $ e_k $ the elementary symmetric functions. The ring of symmetric functions is denoted by $ \Lambdab := \Cc\crochet{p_k}_{k\geq 1} = \Cc\crochet{h_k}_{k\geq 0} = \Cc\crochet{e_k}_{k\geq 0} $.

The use of $ \lambda $-ring/plethystic notations and alphabets will be done in the same vein as \cite{LascouxSym, HaimanMacDo}, using in particular the convention of \cite[\S~2]{HaimanMacDo} that writes plethysm with a bracket $ \crochet{\cdot}$: the plethysm of a symmetric function $ f $ in an abstract alphabet $ \Ae $ will be denoted $ f[\Ae] $. In the $ \lambda $-ring vision of e.g. \cite{LascouxSym}, symmetric functions are polynomial functors in the vein of \cite[Appendix A p. 149]{MacDo}. As a result, the term functor will be sometimes used to designate some quantities.

\begin{remark}
We recommend the reading of \cite[\S~2]{HaimanMacDo} to become familiar with the notion and the use of plethysm in the manipulation of symmetric functions. 
\end{remark}


\medskip
\subsubsection{Plethystic sum and product}\label{SubSubSec:AnnexSymFunc:General:PethSumProd}

Recall that $ p_k(x_1, x_2, \dots) := \sum_{\ell \geq 1} x_\ell^k $. Given the sets of variables or ``alphabets'' $ X := \ensemble{ x_k }_{k \geq 1} $ and $ Y := \ensemble{ y_k }_{k \geq 1} $, one defines 
\begin{align*}
p_k\crochet{X + Y}  & =  p_k \crochet{X}  + p_k\crochet{Y}  \\
p_k\crochet{ X \cdot  Y } & =  p_k\crochet{X}  p_k\crochet{Y}  
\end{align*}

Note the difference of convention between $ p_k( X + Y ) = \sum_\ell (x_\ell + y_\ell)^k $ and $ p_k[X + Y] $.

With these definitions, $ X + Y $ is the alphabet obtained by concatenating $ X $ and $ Y $ and $ X \cdot Y = \ensemble{ x_k y_\ell }_{k, \ell \geq 1} $ is a ``tensor product'' alphabet. It is clear that $ + $ and $ \cdot $ are associative operations on the alphabets. From now on, we will forget the dot when considering a tensor alphabet and we will write $ X Y $ for it~; we will also forget the $ \ensemble{\cdot} $ for alphabets. For instance, we write $ t X $ for $ \ensemble{t} \cdot X $.

\medskip
\subsubsection{The $H$ functor}\label{SubSubSec:AnnexSymFunc:General:Hfunctor}

We define the generating series of the complete homogeneous symmetric functions by 
\begin{align}\label{Def:Hfunctor}
\boxed{H\crochet{X} := \sum_{n \geq 0} h_n(X) = \prod_{k \geq 1} \frac{1}{ 1 -  x_k }}
\end{align}

It clear that
\begin{align}\label{Eq:H:morphism}
H\crochet{X + Y} = H\crochet{X}H\crochet{Y}
\end{align}

The link betwen the $ h_k $'s and the $ p_k $'s can be summarised by the generating functions identity \cite[I-2, (2.10)]{MacDo}\footnote{We slightly change the convention of \cite[I-2, (2.10)]{MacDo} that uses instead $ P(tX) = \sum_{k \geq 1} t^k p_k\crochet{X} $.}
\begin{align}\label{Eq:H:H=expP}
H\crochet{X} = e^{P[X]}, \qquad P\crochet{X} := \sum_{k \geq 1} \frac{p_k\crochet{X}}{k}
\end{align}

With the usual convention $ p_\lambda := \prod_{k \geq 1} p_{\lambda_k} $, the expansion of $ e^P $ gives \cite[I (2.14)]{MacDo}
\begin{align}\label{Eq:H:H=expPinExtension}
H\crochet{X} = \sum_{\lambda } \frac{1}{ z_\lambda } p_\lambda\crochet{X} 
\end{align}

Replacing $ X $ by $ X Y $ yields
\begin{align*}
\sum_{\lambda } \frac{ 1 }{z_\lambda} p_\lambda\crochet{X} p_\lambda\crochet{Y} = \sum_{n \geq 0} h_n\crochet{XY}  = H\crochet{XY}
\end{align*}

\medskip
\subsubsection{Plethystic difference}\label{SubSubSec:AnnexSymFunc:General:PlethDiff}

The difference of two alphabets $ X - Y $ is defined as the formal alphabet such that \cite[1-3, ex. 23]{MacDo}
\begin{align*}
H\crochet{ X - Y} = \frac{ H\crochet{ X } }{ H\crochet{ Y } }
\end{align*}

This last definition is equivalent to \cite[I (2.10)]{MacDo}
\begin{align*}
p_k\crochet{ X - Y } = p_k \crochet{ X } - p_k \crochet{ Y }
\end{align*}

Note that $ f(-X) := f(-x_1, -x_2, \dots) \neq f\crochet{-X} $. To differentiate between these two operations, we define the ``minus sign'' alphabet
\begin{align}\label{Def:Alphabet:Epsilon}
\varepsilon := \ensemble{-1}
\end{align}
which is such that $ f\crochet{\varepsilon X} = f(-X) $.

\medskip
\subsubsection{Involution $ \omega $}\label{SubSubSec:AnnexSymFunc:General:OmegaInvolution}

The fundamental involution $ \omega $ \cite[I-2, (2.7)]{MacDo} is defined by $ \omega(h_n) = e_n $ and $ \omega(e_n) = h_n $ (i.e. $ \omega^2 = I $). Equivalently, it can be defined by $ \omega(s_\lambda) = s_{\lambda'} $ \cite[I-3, (3.8)]{MacDo}. Since $ s_\lambda\crochet{-X} = (-1)^{\abs{\lambda}} s_{\lambda'}\crochet{X} = s_{\lambda'}\crochet{\varepsilon X} $ (see \cite{HaimanMacDo} or \cite[I-3, (3.10)]{MacDo}), we thus have $ \omega(s_\lambda\crochet{X}) = s_\lambda\crochet{-\varepsilon X} $. As a result, we define the alphabet
\begin{align}\label{Eq:Schur:OmegaInvolution}
\boxed{\widehat{\omega} := -\varepsilon}
\end{align}
which is such that $ \omega(s_\lambda\crochet{X}) = s_\lambda\crochet{\widehat{\omega} X} $.

\medskip
\subsubsection{Cauchy identity}\label{SubSubSec:AnnexSymFunc:General:CauchySchur}

The Cauchy identity is \cite[I-4, (4.3)]{MacDo}
\begin{align}\label{Eq:Schur:CauchyIdentity}
H \crochet{ X  Y} = \sum_\lambda s_\lambda\crochet{X} s_\lambda\crochet{Y}
\end{align}

Acting plethystically on $X$ with $ \omega $ gives the \textit{dual} Cauchy identity
\begin{align}\label{Eq:Schur:DualCauchyIdentity}
H \crochet{\widehat{\omega}XY} = \sum_\lambda (-1)^{ \abs{ \lambda } } s_{\lambda'}\crochet{X} s_\lambda\crochet{Y}
\end{align}

Replacing $ X $ by $ tX $ and taking $ \crochet{t^n} $ in \eqref{Eq:Schur:CauchyIdentity} and \eqref{Eq:H:H=expPinExtension} gives
\begin{align}\label{Eq:Schur:PlethisticHnSchur}
h_n\crochet{ XY } = \sum_{\lambda \vdash n } \frac{1}{z_\lambda } p_\lambda\crochet{X} p_\lambda\crochet{Y} = \sum_{\lambda \vdash n } s_\lambda\crochet{ X }  s_\lambda\crochet{ Y }
\end{align}

\medskip
\subsubsection{Exponential alphabet}\label{SubSubSec:AnnexSymFunc:General:ExpAlphabet}

The alphabet $ \Eee $ is defined by setting
\begin{align}\label{Def:Alphabet:Exp}
p_k\crochet{\Eee } = \Unens{k = 1} \qquad\Longleftrightarrow\qquad H\crochet{t \Eee} = e^t 
\end{align}

Schur-Weyl duality gives \cite[I-7 (7.7) p. 114]{MacDo} 
\begin{align*}
s_\lambda = \sum_{\mu \vdash \abs{\lambda}} \chi^\lambda_\mu \frac{p_\mu }{z_\mu }
\end{align*}

For $ \lambda \vdash n $, one has $ p_\lambda\crochet{\Eee} = \prod_k p_{\lambda_k}\crochet{\Eee} = \Unens{ \lambda = 1^n} $, hence
\begin{align*}
s_\lambda\crochet{\Eee} = \sum_{\mu \vdash n} \chi^\lambda_\mu \frac{p_\mu\crochet{\Eee} }{z_\mu } = \frac{ \chi^\lambda_{1^n}  }{z_{1^n} } = \frac{\chi^\lambda(\id_{\Sg_n})}{n!}
\end{align*}

Since $ \chi^\lambda = \tr{\rho^\lambda} $ with $ \rho^\lambda : \Sg_n \rightarrow \Mat_{d_\lambda}(\Cc) $ (or $ \textrm{End}(V^\lambda) $ with $ V^\lambda \simeq \Cc^{d_\lambda} $), the dimension $ d_\lambda $ of the irreducible module $ V^\lambda $ is given by $ \chi^\lambda(\id_{\Sg_n}) = \tr{I_{d_\lambda}} $, hence \medskip
\begin{align}\label{Eq:DimWithSchur}
d_\lambda = \abs{\lambda}!\, s_\lambda\crochet{\Eee} 
\end{align}

\medskip
\subsubsection{The Schur functions}\label{SubSubSec:AnnexSymFunc:General:Schur}

The Schur functions have several equivalent definitions. One is for instance the Jacobi-Trudi identity \cite[I-3 (3.4)]{MacDo}
\begin{align}\label{Eq:Schur:JacobiTrudi}
s_\lambda = \det\prth{h_{\lambda_{i + \ell - j}}}_{1 \leq i, j \leq \ell}, \quad \forall \ell \geq \ell(\lambda)
\end{align}
which is equivalent, with the simple manipulations of e.g. \cite[\S~3.1]{BarhoumiCUErevisited} to the Fourier coefficient form
\begin{align}\label{Eq:Schur:Fourier}
s_\lambda\crochet{X} = \crochet{U^\lambda} H\crochet{ XU - U^{\varepsilon R} }
\end{align}
with $ U^{\varepsilon R} := \ensemble{ u_i/u_j, j < i} $

These functions satisfy the property of ``restricted'' orthogonality \cite{DiaconisShahshahaniCUE}/\cite[VI-9 rk. 2]{MacDo}
\begin{align}\label{Eq:Schur:RestrictedOrthogonality}
\int_{\Ue_n} s_\lambda(U) s_\mu(U\inv) dU = \delta_{\lambda, \mu}\Unens{\ell(\lambda) \leq n}
\end{align}
where $ dU $ is the normalised Haar measure of the unitary group $ \Ue_n $, whose diagonal part is given by the $ CU\!E_n $ measure $ \abs{\Delta(e^{i\thetab})}^2 \frac{d\thetab}{n!(2\pi)^n} $~; see \cite{DiaconisShahshahaniCUE}/\cite[\S~2.3]{BarhoumiCUErevisited}.

The Schur functions are orthogonal for the Hall scalar product $ \bracket{\cdot, \cdot} $ that can be defined by \cite[I-4]{MacDo}
\begin{align}\label{Eq:Schur:Orthogonality}
\bracket{s_\lambda, s_\mu } = \delta_{\lambda, \mu} 
\end{align}

\subsection{The Schur measure and the Schur kernel}\label{SubSec:AnnexSymFunc:Schur}

The Schur measure is defined in \eqref{Def:Schur:Measure}. The fact that it is a probability measure for a positive specialisation comes from the Cauchy identity \eqref{Eq:Schur:CauchyIdentity}. The law of the total parts of the measure is given by
\begin{align}\label{Eq:SchurMeas:LawTotaSum}
\Espr{\Schur(\Ae, \Be)}{t^{\abs{\lambdab}}} = H\crochet{(t - 1) \Ae\Be}
\end{align}

One has used the scaling property $ s_\lambda\crochet{tX} = t^\abs{\lambda} s_\lambda\crochet{X} $. In particular, for $ \Ae = (p_k)_k $ and $ \Be = (q_k)_k $, one gets 
\begin{align*}
\Espr{\Schur(\Ae, \Be)}{t^{\abs{\lambdab}}} = \prod_{i, j} H\crochet{(t - 1) p_i q_j} = \prod_{i, j} \frac{1 - p_i q_j}{1 - tp_i q_j} = \prod_{i, j} \Esp{ t^{\Geom(p_iq_j)} }
\end{align*}
hence the equality in law $ \abs{\lambdab^{\Schur(\Ae, \Be)} } = \sum_{i, j} \Geom(p_i q_j) $ with independent random variables in the RHS. Of course, since the only ingredients needed for such an equality in law are the scaling of the symmetric functions and the Cauchy identity, it is also valid for e.g. the Macdonald measure \cite{BorodinCorwinMacDo}, but with random variables different from the geometric ones.

The Schur kernel operator is 
\begin{align}\label{Eq:SchurMeas:SchurOperator}
\Keb_{\Schur(\Ae, \Be)} := \Heb(a_{\Ae, \Be})\Heb(\widetilde{a}_{\Ae, \Be}\inv)
\end{align}
where the fundamental function $ a \equiv a_{\Ae, \Be} $ is given by 
\begin{align}\label{Def:Function:SchurMeas:SchurKernelFunction}
a_{\Ae, \Be}(z) := H_\varepsilon\crochet{z\Ae - z\inv \Be} 
\end{align}
with $ H_\Ae\crochet{X} := H\crochet{\Ae X} $, $ \widetilde{f}(z) := f(z\inv) $ and $ f\inv := 1/f $ here.

\medskip

Supposing that $ a(z) \equiv a_{\Ae, \Be}(z) = \sum_{k \in \Zz} \widehat{a}_k z^k $, the Schur kernel can take the following forms~:
\begin{itemize}

\item Hankel convolutive form \cite{BasorWidom}~:
\begin{align}\label{Def:Kernel:HankelConv:Schur}
\Ke_{\Schur(\Ae, \Be)}[x, y] = \sum_{k \geq 1} \widehat{a}_k(x) \widehat{a}_k(y), \qquad \widehat{a}_k(x) := \widehat{a}_{k + x} 
\end{align}

\medskip
\item Integral form \cite[(2.1), (2.2)]{BorodinOkounkov}~:
\begin{align}\label{Def:Kernel:Integral:Schur}
\Ke_{\Schur(\Ae, \Be)}[x, y]  = \crochet{z^x w^y} \frac{a(z)/a(w)}{z/w - 1} \Unens{\abs{z/w} > 1}
\end{align}

\medskip
\item Christoffel-Darboux form : see \cite[rk. 2]{BorodinOkounkov} for the expression.
%
%
%

\end{itemize}

\medskip
\section{Prolate hyperspheroidal wave functions}\label{Annex:PSWF} 

The main reference for this \S~is \cite{SlepianIV}. We also refer to \cite{BeskrovnyKolobov, FriedenPSWF, Shkolnisky}. We nevertheless use the terminology of \cite{HeurtleyI, HeurtleyII}, i.e. ``hyperspheroidal'' instead of ``generalised'' as there are several possible generalisations of the prolate spheroidal wave functions, listed for instance in \cite[\S~30.12 p. 704]{NISThandbook} or \cite[(3.2)]{BeskrovnyKolobov}. See also \cite{CasperGrunbaumYakimovZurrianAlgAiryComm, GrunbaumVasquezZubelli} for interesting developments of generalised prolate spheroidal wave functions in relation with the Airy kernel operator.

\medskip
\subsection{Definitions}\label{SubSec:AnnexPSWF:Definitions} 

We define the following \textit{prolate hyperspheroidal wave function} operator
\begin{align}\label{Def:Operators:PHSWF}
\PHO_{N, c} := \frac{d}{dx}(1 - x^2)\frac{d}{dx} + c^2(1 - x^2) - \frac{\rho^2}{x^2}, \qquad \rho^2 \equiv \rho_N^2 := N^2 - \frac{1}{4}
\end{align}

This operator commutes in $ L^2([0, 1]) $ with the operator $ \Jb_{\!\! N, c} : L^2([0, 1]) \to L^2([0, 1]) $ with kernel
\begin{align}\label{Def:Kernel:JN}
\Je_{\! N, c}(x, y) := \sqrt{cxy} \, J_N( cxy ), \qquad J_m(x) := \crochet{z^m}e^{x (z - z\inv)/2}
\end{align}
whose origin in optics comes, amongst others, from \cite{HeurtleyI, HeurtleyII, SlepianIV} 
\begin{quote}
``the modes in a maser interferometer with confocal spherical mirrors of circular cross section''.
\end{quote}

The common eigenvectors of $ \PHO_{N, c} $ and $ \Jb_{\!\! N, c} $ are called \textit{prolate hyperspheroidal wave functions} and are denoted by $ \psi_k^{(N, c)} $. 

\medskip

The operator $ \PHO_{1/2, c} $ is the more classical \textit{prolate spheroidal wave function} operator (up to an additive constant) that commutes with the \textit{sine kernel operator} $ \Kb_{\!\sinc} $ with kernel $ (x, y) \mapsto \sinc(\pi(x - y)) $ in $ L^2([-1, 1]) $ \cite[(23), (26)]{SlepianPollak}~; see Mehta \cite[ch. 6 \& 18]{MethaBook} for a use of its eigenvectors in Random Matrix Theory and \cite[ch.~30 p. 697]{NISThandbook} for properties of these functions. Note that the Bessel function satisfies $ J_{1/2}(x) = \sqrt{\frac{2}{\pi x}} \sin(x) $, hence $ \PHO_{1/2, c} $ also commutes with the operator of kernel $ (x, y) \mapsto \sin(cxy) $ on $ L^2([0, 1]) $ which is equal, up to a symmetrisation that does not change the (real) eigenvectors and eigenvalues, to the square root of $ \Kb_{\!\sinc} $ that has kernel $ e^{icxy} $, see e.g. \cite[(6.3.17)]{MethaBook}.

\medskip

The operator $ \PHO_{N, 0} $ is a particular specialisation of an intertwinning of the Jacobi differential operator $ \Le_{\alpha, \beta} := \frac{d}{dx}(1 - x^2) \frac{d}{dx} + (\beta - \alpha - x( \alpha + \beta )) \frac{d}{dx} $ whose eigenfunctions are the Jacobi polynomials. More precisely, if $ P_k^{(\alpha, \beta)} $ are the Jacobi polynomials \cite[ch. IV]{Szego}, defining the odd Zernike polynomials by \cite[(16)]{Shkolnisky}
\begin{align}\label{Def:Function:ZernikePols}
T_k^{(N)}(x) := \tau_k^{(N)} x^{N + \frac{1}{2}} P_k^{(N, 0)}(1 - 2x^2), \qquad \tau_k^{(N)} := \sqrt{2(2k + N + 1)}
\end{align}
one has
\begin{align*}
\PHO_{N, 0}  T_k^{(N)} = - \kappa_k^{(N)} T_k^{(N)}, \qquad \kappa_k^{(N)} := (N + 2k + \tfrac{1}{2})(N + 2k + \tfrac{3}{2})
\end{align*}

The tri-band relation of orthogonal polynomials implies that the (semi-infinite) matrix of the endomorphism of multiplication by $ X^2 $ in the basis of the $ T_k^{(N)} $ is tri-diagonal. Since the matrix of $ \PHO_{N, 0} $ in this basis is diagonal and $ \PHO_{N, c} = \PHO_{N, 0} + c^2(1 - X^2) $, one concludes that the matrix of $ \PHO_{N, c} $ is tri-diagonal in the $ (T_k^{(N)})_{k \geq 0} $ basis (with explicit coefficients coming from the Jacobi polynomials). As a result, expanding the eigenvectors $ (\psi_\ell^{(N, c)})_{\ell \geq 1} $ in this basis according to 
\begin{align*}
\psi_\ell^{(N, c)} = \sum_{k \geq 0} d_{\ell, k}^{\, (N, c)} T_k^{(N)},
\qquad
\PHO_{N, c}  \psi_\ell^{(N, c)} = \lambda_\ell^{(N, c)} \psi_\ell^{(N, c)}
\end{align*}
one obtains the semi-infinite matrix equation \cite[(30)]{Shkolnisky}
\begin{align*}
\Mb^{(N, c)}d_{\ell, \cdot}^{\, (N, c)} = \lambda_\ell^{(N, c)} d_{ \ell, \cdot}^{\, (N, c)}, \qquad \Mb^{(N, c)} := \mathrm{Mat}_{(T_k^{(N)})_k }(\PHO_{N, c})
\end{align*}

In the same vein as for the prolate spheroidal wave functions, one can expand $ \psi_\ell^{(N, c)} $ and $\lambda_\ell^{(N, c)} $ into powers of $c$~; see e.g. \cite{SlepianIV}.

\medskip
\subsection{Rescaling}\label{SubSec:AnnexPSWF:Rescaling} 

$ $

\begin{shaded}
\begin{lemma}[Limiting hyperspheroidal operators]\label{Lemma:Rescaling}
Set
\begin{align*}
x := 1 - \frac{X}{2}N^{-2/3}, \qquad 
y := 1 - \frac{Y}{2}N^{-2/3}, \qquad 
c := 2N - s N^{1/3}
\end{align*}

Then, locally uniformly in $ \Rr_+ $
\begin{align*}
\frac{N^{1/6}}{\sqrt{2}} \Je_{2N, 2N - sN^{1/3}}(1 - XN^{1/3}, 1 - YN^{1/3}) 
\tendvers{N}{+\infty} \Ai(X + Y + s)
\end{align*}
and, on test functions that are $ \Ce^2 $ with exponential decay at infinity (say)
\begin{align*}
\frac{N^{-2/3}}{4} \PHO_{N, N - sN^{1/3}/2} \tendvers{N}{+\infty} \Le_{TW, s} := \frac{d}{dX}X\frac{d}{dX} - X(X + s)
\end{align*}

As a result, the eigenvalues $ \Lambda_k^{(N, N - sN^{1/3}/2)} $ of $ \Jb_{\! N, N - sN^{1/3}/2} $ satisfy
\begin{align*}
N^{-1/6} \Lambda_k^{(N, N - sN^{1/3}/2)} \tendvers{N}{+\infty} \lambda_k(s) 
\end{align*}
where $ \lambda_k(s) $ is the $k$-th eigenvalue of $ \Hb(\Ai_s) $ acting on $ L^2(\Rr_+) $.
\end{lemma}
\end{shaded}

\begin{remark}
If one defines
\begin{align*}
\widetilde{\Je}_{N, c}(X, Y) & := \Je_{N, c}(e^{-X}, e^{-Y})e^{-(X + Y)/2}, \\ 
\widetilde{J}_{N, c}(X) & := \sqrt{c} \, e^{-X} J_N(c\, e^{-X}), \\
\Psi_\ell^{(N, c)}(X) & := \psi_\ell^{(N, c)}(e^{-X})e^{-X/2}
\end{align*}
and the associated operator $ \widetilde{\Jb}_{\! N, c} $ acting on $ L^2(\Rr_+) $ with kernel $ \widetilde{\Je}_{N, c} $, then, one has
\begin{align*}
\widetilde{\Jb}_{\!N, c} & = \Hb( \widetilde{J}_{N, c} ) \ActsOn L^2(\Rr_+) \\
\widetilde{\Jb}_{\!N, c}\Psi_\ell^{(N, c)} & = \Lambda_\ell^{(N, c)}\Psi_\ell^{(N, c)}
\end{align*}
and lemma~\ref{Lemma:Rescaling} can be rephrased into the weak convergence of operators of the Hankel type in $ L^2(\Rr_+) $ given by the convergence of their symbol. This convergence can be strengthen into a strong convergence depending on the functional space in which belongs the symbol, see e.g. \cite[ch. 5.4]{NikolskiToeplitz} or \cite[ch. 4.2]{BoettcherSilbermann}. 

The operator $ \Hb( \widetilde{J}_{N, c} ) $ is thus, in a certain sense, more natural than the original Heurtley-Slepian operator $ \Jb_{\! N, c} $ (since one remains in the framework of Hankel operators on $ \Rr_+ $). The eigenvectors of this operator are the $ \psi_\ell^{(N, c)}(e^{-x})e^{-x/2} $, which is somewhat similar to the exponential change of variable of Whittaker functions \cite[\S~4]{BorodinCorwinMacDo}/\cite{GLOGaussGivental, GLONewRepr}. The commuting differential operator to $ \widetilde{\Jb}_{\!N, c} $ on $ L^2(\Rr_+) $ is easily proven to be
\begin{align*}
\widetilde{\PHO}_{N, c} := \frac{d}{dx} \prth{e^{2x} - 1} \frac{d}{dx} + 2\prth{\rho^2 - 1} e^{2x} + c^2  e^{-2x}, \qquad \rho^2 := N^2 - \tfrac{1}{4}
\end{align*} 
\end{remark}

\medskip

\begin{proof}
Using the limit \cite[(4.39)]{JohanssonHouches}/\cite[(10.19.8) p. 232]{NISThandbook} 
\begin{align}\label{Eq:LimBesselToAiry}
\nu^{1/3} J_{2\nu }(2\nu - X \nu^{1/3}) \tendvers{\nu}{+\infty} \Ai(X) 
\end{align}
one easily proves the first claim.

The second claim comes from the fact that $ dx = -\frac{N^{-2/3}}{2} dX $ which implies
\begin{align*}
\frac{N^{-2/3}}{4}\frac{d}{dx} (1 - x^2) \frac{d}{dx} & = \frac{d}{dX}\prth{ X - \frac{X^2 N^{-2/3}}{4}} \frac{d}{dX},
\end{align*}
and from the expansion at the second order of 
\begin{align*}
\frac{\rho^2}{x} & = \frac{N^2 - \frac{1}{4}}{ 1 - XN^{-2/3}/2 } = \prth{ N^2 - \frac{1}{4}}\prth{ 1 + \frac{XN^{-2/3}}{2} + \frac{X^2 N^{-4/3}}{4} + O(X^3 N^{-2})}
\end{align*}
that gets compensated from the term
\begin{align*}
c_N^2( 1 - x^2) = \prth{N - \frac{s N^{-1/3} }{2}} \prth{  XN^{-2/3} - \frac{X^2 N^{-4/3}}{4} }
\end{align*}
to produce the term $ X(X + s) $ after multiplication by $ \frac{N^{-2/3}}{4} $. Full details are left to the reader.
\end{proof}

\medskip
\section*{Acknowledgements}

The author thanks P. Biane, G. Borot, B. Eynard, P. Ferrari, J. Heiny, O. H\'enard, K. Johansson, A. N. Kirillov, A. Kuijlaars, P. Maillard, S. N. Majumdar, J. Najnudel, M. Noumi, T. Popov, D. Rosen, G. Schehr, N. Simm, C. A. Tracy, P. Tran, R. Tribe, B. Westbury, O. Zaboronski and N. Zygouras for interesting discussions, remarks and insights concerning various stages of development of this work. 

A particular thanks is given to P. J. Forrester for the reference \cite{WuXuZhao}.

\medskip
\section*{Dedicatory}

This work is dedicated to the memory of my father Ahmed Mohsen Barhoumi (10/02/1949-24/10/2023).


\bibliographystyle{amsplain}

$ $


\end{document}